\documentclass{amsart}
\usepackage{graphicx} 

\title{Taking Model-Complete Cores}
\author{Manuel Bodirsky, Bertalan Bodor, Paolo Marimon}

\usepackage{mathtools}

\usepackage{amssymb}
\usepackage{geometry}
\usepackage{tikz}
\usepackage{verbatim}
\usepackage[hidelinks]{hyperref}
\usepackage[english]{babel}
\usepackage{wasysym}
\usepackage[shortlabels]{enumitem}
\usepackage{mathtools}
\usepackage[colorinlistoftodos]{todonotes}

\usepackage{colortbl}

\DeclareMathOperator{\AGL}{AGL}
\DeclareMathOperator{\PGL}{PGL}
\DeclareMathOperator{\AGammaL}{A \Gamma L}
\DeclareMathOperator{\PGammaL}{P \Gamma L}

\usepackage{etoolbox}

\makeatletter
\newcommand{\includeCroppedPdf}[2][]{\begingroup%
    \edef\temp@mdfivesum{\pdf@filemdfivesum{#2.pdf}}%
    \ifcsstrequal{#2mdfivesum}{temp@mdfivesum}{}{%
        \immediate\write18{pdfcrop #2 #2-crop.pdf}}%
        \immediate\write\@auxout{\string\expandafter\string\gdef\string\csname\space #2mdfivesum\string\endcsname{\temp@mdfivesum}}%
    \includegraphics[#1]{#2-crop.pdf}\endgroup}
\makeatother

\theoremstyle{definition}
\newtheorem{theorem}{Theorem}[section]
\newtheorem{lemma}[theorem]{Lemma}
\newtheorem{corollary}[theorem]{Corollary}
\newtheorem{definition}[theorem]{Definition}
\newtheorem{notation}[theorem]{Notation}
\newtheorem{example}[theorem]{Example}

\newtheorem{proposition}[theorem]{Proposition}

\newtheorem{fact}[theorem]{Fact}

\newtheorem{conjecture}[theorem]{Conjecture}

\theoremstyle{remark}
\newtheorem{remark}[theorem]{Remark}
\newtheorem{question}[theorem]{Question}

\numberwithin{equation}{section}

\newcommand\red[1]{\textcolor{red}{#1}}
\newcommand\blue[1]{\textcolor{blue}{#1}}

\newcommand\green[1]{\textcolor{green!50!black}{#1}}

\DeclareMathOperator{\I}{\operatorname{I}}
\DeclareMathOperator{\MI}{\operatorname{MI}}

\DeclareMathOperator{\M}{\operatorname{M}}

\newcommand{\RM}{\mathrm{MR}}
\newcommand{\dM}{\mathrm{dM}}
\newcommand{\mZ}{\mathbb Z}
\newcommand{\aut}{\operatorname{Aut}}
\newcommand{\en}{\operatorname{End}}

\newcommand{\eend}{\operatorname{End}}

\newcommand{\sym}{\operatorname{Sym}}
\newcommand{\id}{\operatorname{id}}

\newcommand{\dom}{\operatorname{Dom}}

\newcommand{\typ}{\operatorname{tp}}

\newcommand{\csp}{\operatorname{CSP}}

\newcommand{\atoms}{({\mathbb Q};=)}

\newcommand{\order}{(\mathbb{Q};<)}
\newcommand{\upairs}{{{\mathbb Q}\choose 2}}

\newcommand{\john}{\mathfrak{J}}
\newcommand\johnord[1]{\mathcal{J}^{<}(#1)}
\newcommand\johnneq[1]{\mathcal{J}^{\neq}(#1)}

\newcommand{\tp}{\operatorname{tp}}
\newcommand{\betw}{\operatorname{Betw}}

\newcommand{\fa}{\mathfrak{A}}
\newcommand{\fb}{\mathfrak{B}}
\newcommand{\fc}{\mathfrak{C}}
\newcommand{\fx}{\mathfrak{X}}
\newcommand{\fy}{\mathfrak{Y}}
\newcommand{\fz}{\mathfrak{Z}}
\newcommand{\fd}{\mathfrak{D}}
\newcommand{\fe}{\mathfrak{E}}

\newcommand{\bA}{\mathfrak{A}}
\newcommand{\bB}{\mathfrak{B}}
\newcommand{\bC}{\mathfrak{C}}
\newcommand{\bF}{\mathfrak{F}}
\newcommand{\cc}{\mathbb{S}}

\newcommand{\same}{\equiv}

\newcommand\ignore[1]{}
\newcommand{\age}{\operatorname{Age}}
\newcommand{\Aut}{\operatorname{Aut}}
\newcommand{\Sym}{\operatorname{Sym}}
\newcommand{\End}{\operatorname{End}}

\newcommand\qup[1]{\mathbb{Q}_\nearrow^{#1}}

\newcommand\qneq[1]{\mathbb{Q}^{(#1)}}

\newcommand{\shuff}{\oplus}

\newcommand{\rotate}{\circlearrowright}
\newcommand{\kernel}{\operatorname{Ker}}

\DeclareMathOperator{\Th}{Th}

\newcommand{\pau}[1]{\textcolor{blue}{#1}}

\begin{document}

\begin{abstract} 
A first-order theory $T$ is a model-complete core  theory if every first-order formula is equivalent modulo $T$ to an existential positive formula; a
core  companion of a theory $T$ is a model-complete core  theory $S$ such that every model of $T$ maps homomorphically to a model of $S$ and vice-versa. Whilst core  companions may not exist in general, if they exist, they are unique. Moreover, $\omega$-categorical theories always have a core companion, which is also $\omega$-categorical. 
We show that many model-theoretic properties, such as stability, $\mathrm{NIP}$, simplicity, and $\mathrm{NSOP}_k$ for ${k\in\mathbb{N}_{>0}}$, are preserved by moving to the core companion of a complete theory. On the other hand, we show that the classes of theories of structures interpretable 
over $({\mathbb N};=)$ and over $({\mathbb Q};<)$
are both \emph{not} closed under taking core companions. 
The first class is contained in the class of theories of $\omega$-stable first-order reducts of finitely homogeneous relational structures, which was studied by Lachlan in the 80's. We conjecture the two classes to be equal. To support our conjecture we prove that all structures in Lachlan's class are trace definable in $(\mathbb{N}; =)$, confirming a conjecture of Walsberg.
\end{abstract}

\maketitle

\tableofcontents

\ignore{
\section*{Notes}

Colors for changes:
\begin{itemize}
\item \red{red}: Manuel
\item \blue{blue}: Paolo
\item \green{green}: Bertalan
\end{itemize} 
If some colored text is approved by another author, the color can be removed. 
}

\section{Introduction}
A theory $T$ has quantifier elimination if every first-order formula is equivalent to a quantifier-free formula modulo $T$. Model-completeness is a natural weakening of quantifier elimination: instead of requiring that every formula is equivalent to a quantifier-free formula, we ask
that they are equivalent to existential formulas.
Some theories $T$ might not be model-complete, but might have a \emph{companion} which is model-complete: $S$ is called a \emph{companion} of $T$ if $S$ and $T$ 
have 
the same universal consequences
(equivalently, every model of $T$ embeds in a model of $S$ and vice-versa).
A \emph{model companion of $T$} is a companion $S$ of $T$ which is model complete. 
If $T$ has a model companion $S$, then $S$ is unique up to equivalence of theories (see, e.g.,~\cite{HodgesLong}). 

There is also a positive version of the  concept of a model companion, inspired by the theory of graph homomorphisms: 
we say that a theory $T$ is a \emph{core theory} if every existential formula is equivalent to an existential positive formula modulo $T$. A theory $S$ is the \emph{core companion} of $T$ if $S$ and $T$ have the same universal negative consequences (equivalently, every model of $T$ maps homomorphically to a model of $S$ and vice-versa), and $S$ is a model-complete core theory.
Again, if $T$ has a core companion $S$, then $S$ is unique up to equivalence of theories~\cite{BodHilsMartin-Journal}.

In this article,
we are often especially interested in model companions and core companions  of $\omega$-categorical theories. A theory is \emph{$\omega$-categorical} if it has a unique countable model up to isomorphism, and a structure is called $\omega$-categorical if its first-order theory is. Examples of structures with $\omega$-categorical theories are all finite structures, $(\mathbb{N}; =)$, $(\mathbb{Q}; {<})$, the Rado graph, and {all} infinite vector spaces over finite fields. By Saracino's theorem, every $\omega$-categorical theory has a model companion $S$, and $S$ is $\omega$-categorical~\cite{Saracino}. 
If $\bA$ is the 
countable model of an $\omega$-categorical theory $T$,  
then $T$ is model-complete if and only if 
the automorphism group $\Aut(\bA)$ is dense in the monoid of elementary self-embeddings of $\bA$ with respect to the topology of pointwise convergence (a consequence of the theorem of Engeler, Svenonius, and Ryll-Nardzewski; see~\cite{HodgesLong,BodJunker}). Every $\omega$-categorical theory also has a model-complete core
companion~\cite{Cores-journal,BodHilsMartin-Journal}. If $\bA$ is the countable model of an $\omega$-categorical theory $T$,  then $T$ is a model-complete core theory if and only if $\Aut(\bA)$ is dense in the endomorphism monoid $\End(\bA)$. Moreover, two $\omega$-categorical theories have the same universal negative consequences if and only if their countable models $\bA$ and $\bB$ are \emph{homomorphically equivalent}, i.e.,
there exists a homomorphism from $\bA$ to $\bB$ and vice versa. 
In particular, every finite structure $\bB$ is homomorphically equivalent to a finite structure $\bC$ all of whose endomorphisms are automorphisms, and which is called the \emph{core} of $\bB$ in the theory of graph homomorphisms, database theory, and constraint satisfaction (see, e.g.,~\cite{HNBook}). 

Typically, the model-companion and core companion of $T$ are \emph{simpler} than $T$. The simplification for the model companion is usually modest. 
For example, the first-order theory of the countable dense linear order with a start point but no endpoints, i.e., the first-order theory of $({\mathbb Q}_{\geq 0};<)$, is not model-complete, and its model companion is the first-order theory of $({\mathbb Q};<)$; it is simpler in the sense that for each $n \geq 1$, it has strictly fewer types over the empty set. For the core companion the model-theoretic simplification can be more radical. For example, the first-order theory $T_R$ of the Rado graph $R$ is model-complete (it even has quantifier elimination), but not a core theory, since $R$ contains infinite cliques $K_{\omega}$, and endomorphisms $e \colon R \to K_\omega$ which are not embeddings. 
The core companion of $T_R$ is the theory of $K_{\omega}$, which is $\omega$-stable (Definition~\ref{def:superstab}), while $T_R$ even has the independence property (IP) (Definition~\ref{def:NIP}).
In the following, we will say that a structure is model-complete (or a model-complete core) if its theory is model-complete (or has a model-complete core theory). 
A structure $\bC$ is a \emph{companion} (\emph{model companion}, 
\emph{core companion}) of a structure $\bB$ if the first-order theory of $\bC$ is a companion (model companion, 
core companion) of the first-order theory of $\bB$.

In some applications of model theory, we are only interested in the existential or existential positive theory of a given structure $\bB$. For instance, 
if we want to study the decidability of the existential theory of $\bB$ (which is for instance undecidable for $({\mathbb N};+,*,1)$~\cite{MatiyasevichDoklady} and not known to be decidable for $\bB = ({\mathbb Q};+,*,1)$~\cite{Anscombe2024}), then obviously any companion of 
$\bB$ has the same existential theory. Another example  
is the study of the computational complexity of $\csp(\bB)$, the \emph{constraint satisfaction problem} of $\bB$, which can be seen as the problem of deciding the primitive positive theory of $\bB$. If two structures satisfy the same universal negative sentences, or if they are even homomorphically equivalent, then they have the same CSP. 
In these applications, 
the fact that the model companion or core companion of $T$ is structurally simpler than $T$ is the reason why these concepts are so useful: 
for instance, instead of studying the complexity of $\csp(\bB)$, it might be simpler to study the complexity of the $\csp$ of the core companion of $\bB$. 

If we want to apply this method systematically to study $\csp(\bB)$ for structures $\bB$ from some class of structures $\mathcal C$, we ideally need 
that ${\mathcal C}$ satisfies two conditions:
\begin{itemize}
    \item every structure in ${\mathcal C}$ has a core companion, and 
    \item the core companion of ${\mathcal C}$ is again in the class ${\mathcal C}$. 
\end{itemize}
If these two conditions are satisfied, then we have reduced our task to the model-complete cores in ${\mathcal C}$, which is often simpler. 
It should be mentioned that the task to determine the model-complete core can come with considerable combinatorial problems: already if $\bB$ is a finite structure, the question whether $\bB$ is a core is coNP-complete~\cite{cores}. Also the following question remains open. 

\begin{question}[Question (5)  in~\cite{Book}]
 \label{quest:fin-hom}   
 Does there exist a finitely homogenizable structure (for the definition, see Section~\ref{sect:prelims}) whose model-complete core is not finitely homogenizable?
\end{question}

\subsection*{Contributions}
In this article, we show that many classes of interest in model theory, such as $\omega$-stability, stability, $\mathrm{NIP}$, $\mathrm{NSOP}_k$ for $k\in \mathbb{N}_{>0}$,
and simplicity (in the model theoretic sense) are closed under taking model-complete cores; we develop general methods for proving this. 
In particular, we build on two recent general accounts of model-theoretic properties: firstly, building on work on patterns and straight definitions~\cite{saharon2000not, bailetti2024walk}, we show preservation under taking the core companion of any model-theoretic property definable by omitting a \emph{semi-positive} straight pattern.
Moreover, we show that the core companion of a theory $T$ is interpretable in a weak sense (i.e., \emph{trace definable}~\cite{walsberg2026tracedefinabilityipreservation}) in $T$. These and other techniques allow us to show that most known model theoretic properties preserved by interpretations, and even some properties not preserved by interpretations (such as monadic stability and monadic $\mathrm{NIP}$), are preserved by the core companion of a theory.




However, some other properties, such as total categoricity, are not closed under taking model-complete cores. 
One of the main results of the present paper is the resolution of an open question from~\cite{bodirsky2021permutation} (Section 10): 
we show that the class of structures that are (first-order) interpretable in
a countably infinite structure with the empty signature (short: \emph{interpretable {in} equality}) is \emph{not} closed under taking model-complete cores. This is surprising, because the class is tame from a model-theoretic perspective, and at first sight appears to be quite robust. 
It is the central class of structures studied in an active area of theoretical computer science which studies automata theory with atoms and orbit-finite computation (\cite{DBLP:conf/lics/BojanczykKLT13,KlinLOT14-short,DBLP:journals/corr/BojanczykKL14,BojanczykTorunczyk18,OrbitFinDim,Orbit-finite-LP}). 

Interestingly, the model-complete core of our counterexample has an interpretation in $({\mathbb Q};<)$.
In fact, every structure interpretable {in} equality is $\omega$-stable, 
and hence its model-complete core is $\omega$-stable as well. 
Moreover, they are reducts of homogeneous structures with finite relational signature, and all $\omega$-stable reducts of homogeneous structures with finite relational signature are interpretable {in} $({\mathbb Q};<)$ (Theorem~\ref{lachlan_extra}). The class of $\omega$-stable structures which are reducts of homogeneous structures in a finite relational language, which we shall call \emph{Lachlan's class} ({see Section~\ref{sect:lachlan}}) was studied in depth by  Lachlan~\cite{LachlanSurvey, LachlanIndiscernible, CherlinLachlan} and other model theorists in the 1980s and 90s~\cite{HrushovskiTotallyCategorical, macpherson1991interpreting}. We conjecture that Lachlan's class equals the class of structures that are interdefinable with the model-complete core of a structure interpretable in equality.


To support our conjecture, we establish the weaker conjecture of Walsberg~\cite{walsberg2026tracedefinabilityipreservation, Walsbergshort, Walsberglong} which states that all structures in Lachlan's class are trace definable in $(\mathbb{Q}; =)$. 

We also investigate the classes of model-complete cores and model companions of structures interpretable in $(\mathbb{Q}; <)$.
In particular, we show that the class of structures interpretable in $({\mathbb Q};<)$ is not closed under taking {model companions}. {We note that} here we use a proof technique which is {quite} different from the one for our analogous result for structures interpretable in equality mentioned earlier: {in the case of equality our proof is more group-theoretic in nature, whereas in the case of $({\mathbb Q};<)$ our proof uses a direct model-theoretical argument about definable orders.}
A concrete counterexample is the so-called \emph{generic permutation}, which is the \emph{model companion}
of a structure with a two-dimensional  interpretation {in} $({\mathbb Q};<)$, but which is itself not interpretable {in} $({\mathbb Q};<)$.
The generic permutation plays an important role in the recent systematic study of NIP structures that are homogeneous in a binary signature~\cite{SimonRankOne,rosy,guingona2015common}.

The typical strategy to prove that a structure $\fa$ is not interpretable in a structure $\fb$ is by exhibiting a model-theoretic property which is preserved by interpretations, and to prove that this property holds in $\fb$, but not in $\fa$. Alternative strategies in the $\omega$-categorical context (see, for example,~\cite[Theorem A.21]{Walsbergshort}) rely on type counting arguments, with $\fb$ having too fast growth-rate in its number of types to be interpretable in $\fa$. Neither of these strategies can work to prove that a structure $\fa$ does not interpret its model-complete core  $\fb$: essentially all model theoretic properties which are preserved under interpretations that hold for $\fa$ will hold for $\fb$ (Section~\ref{sec:preserved}), and $\fb$ only has a slower growth-rate in its number of types than $\fa$. Indeed, proving non-interpretability for structures which are very similar from a model theoretic perspective is a major challenge. To overcome it, we develop a toolkit of techniques and ideas 
to prove non-interpretability. We believe this to be another major contribution of this paper.

The closure of the class of structures with a first-order interpretation in $({\mathbb Q};<)$ under taking model-complete cores and reducts is quite interesting: 
it is quite robust (e.g., closed under interpretations)  and contained in the class of first-order reducts of NIP homogeneous structures with a finite relational signature of maximal arity two. We do not know an example that would separate these two classes. 

\section{Preliminaries}
\label{sect:prelims}
	We write $[n]$ for the set $\{1,\dots,n\}$. We refer the reader to~\cite{HodgesLong} for 
   an introduction to first-order logic (which for us always includes equality) and
    proofs of the various model-theoretic facts mentioned below.
   
   All the structures in this article are assumed to be \emph{relational}, i.e., they have a signature $\tau$ which consists of relation symbols only. 
    If $R \in \tau$ is a relation symbol of arity $k$, and $\bA$ is a $\tau$-structure, we write $R^{\bA} \subseteq A^k$ for the respective relation of $\bA$, and we write $\dom(\bA)$ for the domain of $\bA$ (usually, $\dom(\bA)$ will also be denoted by $A$).

    A structure $\bB$ obtained from $\bA$ by dropping some of the relations is called a \emph{reduct} of $\bA$, and conversely $\bB$ is called an \emph{expansion} of $\bA$. 
    An expansion $\bB$ of $\bA$ is called a \emph{first-order expansion} if all relations of $\bB$ are first-order definable in $\bA$. We say that a structure $\fb$ is a \emph{first-order reduct} of a structure $\fa$ if 
it is a reduct of a first-order expansion of $\fa$. Two structures are called \emph{interdefinable} if they are first-order reducts of each other.
If $\fa$ is a structure and $B \subseteq A$, then we write $\fa|_B $ for the structure  with domain $B$ and all the relations of the form $R \cap B^n$ where $n \in {\mathbb N}$ and $R$ is definable (without parameters) in $\fa$.

    A structure $\fb$ is called \emph{finitely bounded} if there are finitely many finite structures $\bF_1,\dots,\bF_n$ such that a finite structure $\fa$  embeds into $\fb$ if and only if it does not embed $\bF_i$ for {any} $i \in [n]$. 

    A structure $\bA$ is called \emph{homogeneous} if every isomorphism between finite substructures of $\bA$ can be extended to an automorphism of $\bA$. (Some authors use \emph{ultrahomogeneous} to prevent confusion with other notions of homogeneity; in our context, however, there is no such danger of confusion). We call a structure \emph{finitely homogeneous} if it is homogeneous and has a finite relational signature, and we call it \emph{finitely homogenizable} if it is interdefinable with a finitely homogeneous structure. We say that a finitely homogeneous structure is \emph{binary} if all of its relations have arity $\leq 2$.
    Finitely homogeneous structures are \emph{$\omega$-categorical}, i.e., their first-order theory has at most one countable model up to isomorphism~(see, e.g.,~\cite[Theorem 7.4.1]{HodgesLong}). 
    

    It is often useful to study $\omega$-categorical structures with notions from permutation group theory. 
    If $G$ is a permutation group on a set $A$, then
    an \emph{orbit of $G$ on $A^k$} (or an \emph{orbit of $k$-tuples}) is a set of the form
    $\{(\alpha(a_1),\dots,\alpha(a_k)) \mid \alpha \in G\}$ for some $(a_1,\dots,a_k) \in A^k$. 
   In an $\omega$-categorical structure $\fa$, the first-order definable relations of arity $k \in {\mathbb N}$ are precisely the finite unions of \emph{orbits of $k$-tuples} of $\Aut(\fa)$~\cite[Corollary 7.3.5]{HodgesLong}.\\

   We shall use the standard convention of denoting tuples of variables by overlined letters at the end of the alphabet $\bar{x},\bar{y},\bar{z},\dots$ and tuples of elements by overlined letters at the beginning of the alphabet $\bar{a},\bar{b},\bar{c}\dots$.

   Sometimes it will be best to think in terms of theories rather than structures. Given a first-order theory $T$ in the signature $\tau$, a model $\fa$  of $T$, and $C\subseteq\fa$, let $\tau\cup\{C\}$ be the expansion of the {signature} $\tau$ by constant symbols naming each element of $C$. A \emph{type} (or \emph{$\tau$-type} if $\tau$ is not clear from the context)  in the variables $\bar{x}$ over $C$ is a maximally consistent set $p$ of $\tau\cup\{C\}$-formulas with free variables belonging to $\bar{x}$. We write $S_{\bar{x}}(C)$ for the set of all types in the variable $\bar{x}$ over $C$ and we just write $S_1(C)$ when $\bar{x}$ is a single variable. Given a $k$-tuple $(a_1, \dots, a_k):=\bar{a}\in A^k$, its \emph{type over $C$} is the set of all $\tau\cup\{C\}$-formulas satisfied by $\bar{a}$ in the variables $(x_1, \dots, x_k)$ (where for each $1\leq i\leq k$, $x_i$ substitutes $a_i$). We say that $\bar{a}$ \emph{realizes the type} $p\in S_{\bar{x}}(C)$ if $p$ is the type of $\bar{a}$ over $C$. We say that a $\tau\cup\{C\}$-formula $\varphi(\bar{x})$ \emph{isolates} a type $p\in S_{\bar{x}}(C)$ if $\varphi\in p$ and for every formula $\psi\in p$ and for every model $\fa$ of $T$ we have $\fa\models \forall \bar{x}(\varphi (\bar{x})\rightarrow \psi(\bar{x}))$.
   If $\fa$ is $\omega$-categorical, then all types in finitely many variables over finitely many parameters are isolated by a single formula. In particular, all these types are
   realized in $\fa$ (this property is known as being $\omega$-saturated~\cite[Exercise 7.2.11]{HodgesLong}). Moreover, two $k$-tuples of $\fa$ have the same type (over $\emptyset$) if and only if they are in the same orbit of $\Aut(\fa)$~\cite[Corollary 7.3.3]{HodgesLong}.
   

\subsection{Positive theories and core companions}
In Section~\ref{sec:preserved}, we will be looking at properties preserved when taking the core companion  of a theory. In particular, we will study properties preserved amongst theories with the same universal negative consequences. This naturally leads us to the study of existential positive formulas (which are preserved under homomorphisms) and more broadly of positive model theory. This branch of model theory has recently received substantial attention (see, for example,~\cite{ben2003positive, PosModT, kamsma2023bilinear, BodHilsMartin-Journal}), and the way positive model theoretic properties are defined is a fundamental inspiration to our method to prove that traditional model theoretic properties are preserved when moving to the core companion. In this paper we will adopt the notation of~\cite{kamsma2025positivelogicintroductionmodel} and we direct the reader to this introductory text and~\cite[Chapter 2]{Book} for proofs of the facts we mention below. 

\begin{definition}
An \emph{existential positive} formula $\phi(\bar{y})$ is a formula of the form $\exists \bar{x} \; \psi(\bar{x};\bar{y})$, where $\psi(\bar{x};\bar{y})$ is quantifier-free and contains no negations.
\end{definition}
\begin{definition}
    A \emph{positive theory} is a set of \emph{$h$-inductive sentences}, i.e.,  sentences of the form 
    \[\forall \bar{x}(\phi(\bar{x})\rightarrow  \psi(\bar{x})),\]
    where $\phi(\bar{x})$ and $\psi(\bar{x})$ are existential positive. An \emph{$h$-universal sentence} (sometimes known as a universal negative sentence) is a sentence of the form 
    \[\forall \bar{x} (\phi(\bar{x})\rightarrow\bot),\]
    where $\phi(\bar{x})$ is positive quantifier-free. 
\end{definition}

\begin{definition} Let $T$ be a positive theory and let $\phi(\bar{x})$ be an existential positive formula. We say that an existential positive formula $\psi(\bar{x})$ is \emph{an obstruction of $\phi(\bar{x})$} if  
\[T\vdash \forall \bar{x} ((\phi(\bar{x})\wedge\psi(\bar{x}))\rightarrow \bot).\]
\end{definition}

\begin{definition} A model $\fa$ of a positive theory $T$ is called \emph{positively closed} if every homomorphism $f \colon \fa\to \fb$, where $\fb\vDash T$, is an \emph{immersion}, i.e., is such that for every existential positive formula $\phi(\bar{x})$ and every tuple $\bar{a}$ from $A$,
\[\fa\vDash\phi(\bar{a}) \text{ if and only if } \fb\vDash\phi(f(\bar{a})).\]
For any model $\fa$ of a positive theory $T$ there is a positively closed model $\fb\vDash T$ such that $\fa$ maps homomorphically to $\fb$~\cite[Lemma 1.20]{ben2003positive}. 
\end{definition}

For the following fact see, for example,~\cite[Proposition 2.1.14]{bodirsky2021complexity}:

\begin{fact}\label{fact:hconsequences} Let $S$ and $T$ be first-order theories. The following are equivalent:
\begin{itemize}
    \item every $h$-universal consequence of $T$ is also an $h$-universal consequence of $S$; and
    \item every model of $S$ maps homomorphically to a model of $T$.
\end{itemize}
\end{fact}

\begin{definition} A first-order theory $T$ is called \emph{model-complete} if every embedding between models of $T$ is elementary. A first-order theory $S$ is a \emph{core theory} if homomorphisms between models of $S$ are embeddings. In particular, we say that $T$ is a \emph{model-complete core} theory if homomorphisms between its models are elementary embeddings.
\end{definition}

By the \emph{positive fragment} of a first-order theory we mean its set of $h$-inductive consequences. Model-complete core theories are axiomatised by their positive fragment~\cite[Proposition 4.17]{BodHilsMartin-Journal}. A positive theory $S$ is called Boolean if all of its models are positively closed~\cite{dobrowolski2022kim} (see~\cite{kamsma2025positivelogicintroductionmodel} for a more detailed discussion). Sometimes this notion is referred to as  ``positively model-complete''~\cite{haykazyan2019spaces}. Note that if $S$ is a model-complete core  theory, then the positive fragment of $S$ is Boolean. Conversely, if a positive theory is Boolean, then it is a model-complete core  theory.

\begin{definition}[\cite{BodHilsMartin-Journal}]  A first-order theory $S$ is the \emph{core companion} of a first-order theory $T$ if:
\begin{itemize}
    \item $S$ is a model-complete core  theory;
    \item $S$ and $T$ have the same $h$-universal consequences.
\end{itemize}
\end{definition}
{It follows from} Fact~\ref{fact:hconsequences} that 
if $S$ and $T$ have the same $h$-universal consequences, {then} every model of $S$ has a homomorphism to a model of $T$ and vice-versa.


\subsection{Interpretations}
Our terminology and notation follows standard conventions in model theory; we refer for instance to Hodges~\cite{HodgesLong} or Tent and Ziegler~\cite{Tent-Ziegler}.

\begin{definition}
    A first-order reduct is called \emph{quantifier-free} if all defining formulas can be chosen to be quantifier-free.
    We say that $\fa$ and $\fb$ are \emph{interdefinable} if they are reducts of one another, and that they are \emph{bidefinable} if $\fb$ is isomorphic to a structure which is interdefinable with $\fa$.
\end{definition}

\begin{remark}
    If at least one of the structures $\fa$ and $\fb$ is countable and $\omega$-categorical, then 
    \begin{itemize}
        \item $\fa$ and $\fb$ are interdefinable if and only if $\aut(\fa)=\aut(\fb)$, and 
        \item $\fa$ and $\fb$ are bidefinable if and only if  $\aut(\fa)$ and $\aut(\fb)$ are isomorphic as permutation groups.
    \end{itemize}
\end{remark}

\begin{definition}
	Let $\fa$ and $\fb$ be structures. We say that $I$ is an \emph{interpretation} of $\fb$ in $\fa$ if there exists some $d\in \omega$ such that $I$ is a partial surjective map from $A^d$ to $B$ such that for every atomic formula $\varphi(x_1,\dots,x_k)$ in the signature of $\fb$ the relation
\[
\{(a_1^1,\dots,a_1^d,\dots,a_k^1,\dots,a_k^d): \fb\models \varphi(I(a_1^1,\dots,a_1^d),\dots,I(a_k^1,\dots,a_k^d))\}
\]
	is definable in $\fa$.
	The number $d$ above is called the \emph{dimension} of the interpretation.
	We say that $\fa$ \emph{interprets} $\fb$, or $\fb$ is \emph{interpretable} in $\fa$ if there exists an interpretation of $\fb$ in $\fa$. For a fixed structure $\fa$,
    we write $\I(\fa)$ for the class of structures which are interpretable in it.
\end{definition}

	Let us consider some interpretation $I$ as in the definition above. By applying the definition for the formula $(x=x)$ we obtain that the domain set of $I$ needs to be definable. Moreover, by applying the definition for the formula $(x=y)$ we also obtain that the kernel of the map $I \colon \dom(I)\rightarrow B$ also needs to be definable in $\fa$. This  essentially means that the domain of $\fb$ can be thought of as a definable quotient of some definable substructure of some power of $\fa$, and every relation of $\fb$ needs to be represented by some set of tuples which is also definable in $\fa$.
    Note that the interpretability relation is transitive, i.e., if $\fb\in \I(\fa)$ and $\fc\in \I(\fb)$ then $\fc\in \I(\fa)$.


\begin{definition}
	An interpretation (of $\fb$ in $\fa$) $I$ is called \emph{full} if every relation $R$ for which $I^{-1}(R)$ is definable in $\fa$ is definable in $\fb$. 
\end{definition}

	It is clear from the definition above a structure $\fb$ is interpretable in $\fa$ if and only if $\fb$ has some expansion which has a full interpretation in $\fa$.

    \begin{remark}\label{rem:action}
        Note that if $I$ is a $d$-dimensional interpretation of $\fb$ in $\fa$, then the following is a well-defined continuous action of $\aut(\fa)$ 
    on $B$: for $\alpha \in \aut(\fa)$ and $b = I(a_1,\dots,a_d)$, 
    then $\alpha \cdot b := I(\alpha(a_1,\dots,a_d))$. This action gives rise to a continuous group homomorphism from $\aut(\fa)$ to $\aut(\fb)$, which we denote by $h_I$. 
    In the case when $\fa$ is $\omega$-categorical it is easy to see that $I$ is full if and only if this homomorphism is surjective.
    \end{remark}

	A special case of a structure with a full interpretation is the full power of $\fa$.
	
\begin{definition}
	The \emph{$d$-th full power} of $\fa$, denoted by $\fa^{[[d]]}$, is defined to be the structure whose domain is $A^d$, and whose relations are

\[
\{((a_1^1,\dots,a_1^d),\dots,(a_k^1,\dots,a_k^d))\in (A^d)^k: \fa\models \varphi(a_{i_1}^{j_1},\dots,a_{i_{\ell}}^{j_{\ell}})\}
\]

	where $1\leq i_1,\dots,i_{\ell}\leq k, 1\leq j_1,\dots,j_{\ell}\leq d$ and $\varphi(x_1,\dots,x_{\ell})$ is an atomic formula in the signature of $\fa$.
\end{definition}

{An $\omega$-categorical structure $\fa$ is called \emph{Ramsey} if the age of 
its expansion by all first-order definable relations is a Ramsey class. We do not define the Ramsey property for classes of finite structures, since we do not need it; all we need to know is that the Ramsey property is in fact a property of the automorphism group of $\fa$, viewed as a topological group; see~\cite{Topo-Dynamics}.} 

\begin{lemma}\label{hom_power}
	If $\fa$ is finitely homogeneous, then so is $\fa^{[[d]]}$. If $\fa$ is additionally finitely bounded/Ramsey, then so is $\fa^{[[d]]}$.
\end{lemma}


\begin{proof}
	The homogeneity and the finite boundedness part follows essentially from Proposition 4.2.19 in~\cite{bodirsky2021complexity}. The Ramsey property follows from the fact that $\aut(\fa)$ and $\aut(\fa^{[[d]]})$ are isomorphic as topological groups (the isomorphism is given by the coordinate-wise action of $\aut(\fa)$).
\end{proof}

\begin{lemma}\label{reduct_mc_core}
    Let $\fa$ and $\fb$ be structures so that $\fb$ has a $d$-dimensional interpretation $I$ in $\fa$. Then $\fb$ has a model-complete core expansion which is homomorphically equivalent to a reduct of the $d$-th full power of $\fa$.
\end{lemma}

\begin{proof}
	Let $\fa^*$ and $\fb^*$ be the expansions of the structures $\fa$ and $\fb$, respectively, by all first-order definable formulas. Then clearly $\fb^*$ is a model-complete core, and $I$ is a pp-interpretation of $\fb^*$ in $\fa^*$. By Lemma 3.8 in~\cite{wonderland} this implies that $\fb^*$ is homomorphically equivalent to some $d$-th \emph{pp-power} of $\fa^*$, in the sense of~\cite{wonderland}, Definition 3.6, 
    which is clearly a reduct of $(\fa^*)^{[[d]]}$ and thus also of $\fa^{[[d]]}$. 
\end{proof}


	Next we show that interpretations preserve being a reduct of a finitely bounded homogeneous Ramsey structure. In order to do this we need the following Ramsey transfer theorem.
	
\begin{theorem}[\cite{MottetPinskerCores}, Theorem 4]\label{mc_transfer}
	Let $\fa$ be a finitely bounded homogeneous relational Ramsey structure, and let $\fc$ be the model-complete core of some reduct of $\fa$. Then {$\fc$} has a finitely bounded homogeneous Ramsey expansion {with the same signature as $\fa$}.
\end{theorem}

	This theorem combined with Lemma~\ref{hom_power} implies the following.

\begin{lemma}
	Let $\fa$ be a structure which is a reduct of a finitely bounded homogeneous Ramsey structure. Then any structure which is interpretable in $\fa$ has a finitely bounded homogeneous Ramsey expansion. 
\end{lemma}

\begin{proof}
	We can assume without loss of generality that $\fa$ itself is a finitely (bounded) homogeneous relational Ramsey structure. Then by Lemma~\ref{hom_power} the same is true for any full power of $\fa$. Let $\fb\in \I(\fa)$. Then $\fb$ has some expansion $\fc$ which is a model-complete core and which is homomorphically equivalent to some reduct of some $\fa^{[[d]]}$ (by Lemma~\ref{reduct_mc_core}). {The statement of the lemma then follows from Theorem~\ref{mc_transfer}.}
\end{proof}

\begin{corollary}\label{reduct_hom}
	All structures with a first-order interpretation in $({\mathbb Q};<)$ 
    are reducts of homogeneous finitely bounded Ramsey structures.
\end{corollary}

	We will present an alternative proof (and also a stronger version) of Corollary~\ref{reduct_hom} in Subsection~\ref{sect:inter_q} (Remark~\ref{rem:ramsey-exp}).

	Note that some authors
    allow parameters in the defining formulas in the definition of an interpretation.     
    However, if the structure $\fa$ interprets its expansions with finitely many constants then this does not make a difference as shown by the following argument.
	We first show the folklore fact that any structure with at least 2 elements interprets all finite structures (cf.~\cite[Lemma 2.4.4]{bodirsky2021complexity}).

\begin{proposition}\label{disjoint_unions}
	Let $\fa$ be any structure with at least two elements. Then $\I(\fa)$ is closed under taking finite disjoint unions.
\end{proposition}

\begin{proof}
	Let $I \colon A^d\rightarrow B$ and $J \colon A^e\rightarrow C$ be interpretations of $\fb$ and $\fc$ in $\fa$, respectively. We can assume without loss of generality that $d=e$ and that $B$ and $C$ are disjoint.
	Let us define the partial map
$$
I\oplus J \colon  A^{d+2}\rightarrow B\cup C, (a_1,\dots,a_{d+2})\mapsto
\begin{cases*}
I(a_1,\dots,a_d) \text{ if $a_{d+1}=a_{d+2}$}\\
J(a_1,\dots,a_d) \text{ if $a_{d+1}\neq a_{d+2}$}
\end{cases*} 
$$
	wherever it is defined. Then it is easy to see that $I\oplus J$ is an interpretation of the disjoint union of $\fb$ and $\fc$ in $\fa$.
\end{proof}

\begin{corollary}\label{inter_finite}
	Every structure with at least 2 elements interprets all finite structures. 
\end{corollary}

\begin{proof}
    	Let $\fb$ be a finite structure. 
	Every structure interprets the 1-element pure set. 
    If $\fa$ is a structure with at least two elements, then Proposition~\ref{disjoint_unions}  implies that $\fa$ also interprets the pure set $(B;=)$ expanded by all constants, which is a finite union of the 1-element pure set. Finally, $\fb$ is clearly a first-order reduct of the expansion of $(B;=)$ by all constants. Hence, $\fa$ interprets $\fb$.  
\end{proof}
	
\begin{lemma}\label{q_add_constants}
	The structures $\order$ and $\atoms$ interpret all of their respective expansions by finitely many constants.
\end{lemma}

\begin{proof}
	Let us observe that $\aut(\mathbb{Q};<,c_1,\dots,c_n)$ is isomorphic to $(\aut\order)^{n+1}\times \id(\{1,\dots,n\})$ as a permutation group. This means that $(\mathbb{Q};<,c_1,\dots,c_n)$ is bidefinable with a disjoint union of $n+1$ copies of $\order$ and $n$ copies of the 1-element pure set. 
	
	Similarly, $\aut(\mathbb{Q};=,c_1,\dots,c_n)=\sym(\mathbb{Q})_{c_1,\dots,c_n}$ and thus $(\mathbb{Q};=,c_1,\dots,c_n)$ is bidefinable with the disjoint union of the infinite pure set and $n$ copies of the 1-element pure set. Thus, both statements of the lemma follow from Proposition~\ref{disjoint_unions}.
\end{proof}

	Note that the lemma above may fail for other classes of $\omega$-categorical structures. In particular, the Rado graph does not interpret its expansion by a constant~\cite{Radowithconstant}.
	We finally note an interesting consequence of Lemma~\ref{q_add_constants} which we will need in our later discussions. 

\begin{lemma}\label{inter_constants}
	Let $\fa$ be any structure which interprets all of its expansions with finitely many constants. Then $\I(\fa)$ is closed under expansions by finitely many constants.
\end{lemma}

\begin{proof}
	Let $\fa$ be as in the hypothesis of the lemma, let $I$ be an interpretation of a structure $\fb$ in $\fa$, and let $c_1,\dots,c_n\in B$.
    Let $d_1,\dots,d_m\in A$ such that each $c_i\in B$ is represented by some tuple from $\{d_1,\dots,d_m\}$.
    Then it is easy to see that $I$ is also an interpretation of $(\fb;c_1,\dots,c_n)$ in $(\fa;d_1,\dots,d_m)$. 
{By our assumption,} $\fa$ interprets $(\fa; d_1,\dots,d_m)$, so it also interprets $(\fb;c_1,\dots,c_n)$.
\end{proof}

\begin{corollary}\label{q_add_constants2}
	$\I(\order)$ and $\I(\atoms)$ are closed under expansions by finitely many constants.
\end{corollary}

\begin{notation}
    For a class $\mathcal{C}$ of 
$\omega$-categorical 
structures let us write
        $\M(\mathcal{C})$ for the class of all structures which are interdefinable with 
the model-complete core of some structure in $\mathcal{C}$.

When applying the operators $\I$ and $\M$ to classes of structures, we sometimes omit brackets, i.e., we write 
    $\MI({\mathcal C})$ instead of $\M(\I({\mathcal C}))$. 
    If ${\mathcal C} = \{\bB\}$, we might also write 
    $\I(\bB)$ and $\M(\bB)$ instead of 
    $\I(\{\bB\})$ and $\M(\{\bB\})$.
\end{notation}

The following theorem illustrates that 
the class $\MI({\mathcal C})$ is quite robust. 
\begin{theorem}\label{prop:IMI-MI}
    For any class of $\omega$-categorical structures ${\mathcal C}$ 
    \begin{align*} \I(\MI({\mathcal C})) = \MI({\mathcal C}). 
    \end{align*}
\end{theorem}
\begin{proof}
    Suppose that $\fb$ is a structure with an interpretation in $\fa \in {\mathcal C}$, let $\fc$ be the model-complete core of
    $\fb$, and suppose that $\fd$ has a $d$-dimensional first-order interpretation in $\fc$. It suffices to show that $\fd \in \MI({\mathcal C})$. 

    By Lemma~\ref{reduct_mc_core}, $\fd$ has a model-complete core expansion $\fd^*$ which is homomorphically equivalent to a reduct $\fd'$ of the $d$-th full power of $\fc$. We may assume that $\fd = \fd^*$ is a model-complete core; otherwise, apply the argument to the expansion of $\fd$ by all first-order definable relations and then take a reduct.

    Let $\tau$ be the signature of $\fd$. 
    If $R \in \tau$ is a $k$-ary, then there is a formula $\phi_R$ with $dk$ free variables that defines $R$ over $\fc$. Since $\fc$ is a model-complete core, 
    we may assume that the formula $\phi_R$ is existential positive (see, for example,  \cite[Proposition 2.6.13]{bodirsky2021complexity}). 
    Let $\fe$ be the $\tau$-structure with domain $B^d$ and the relations defined by $\phi_R$ for every $R \in \tau$; clearly, 
    $\fe \in \I(\fb)$. 
    By assumption, there are homomorphisms $f \colon \fb \to \fc$ and
    $g \colon \fc \to \fb$. 
    Then the map $(a_1,\dots,a_d) \mapsto (f(a_1),\dots,f(a_d))$ is a homomorphism from $\fe$ to $\fd'$, and the map 
    $(a_1,\dots,a_d) \mapsto (g(a_1),\dots,g(a_d))$ is a homomorphism from $\fd'$ to $\fe$. It follows that $\fe$ is homomorphically equivalent to $\fd'$ and thus homomorphically equivalent to $\fd$. Hence, $\fd \in \M(\I(\fb)) \subseteq \M(\I(\I(\fc))) = \MI(\fc)$. It follows that
    $\I(\MI({\mathcal C})) = \MI({\mathcal C})$. 
\end{proof}

\begin{remark} Note that defining interpretability as a relation between theories rather than structures, the above theorem could also be proven outside of the $\omega$-categorical context. 
\end{remark}

\section{Properties preserved by taking model-complete cores}\label{sec:preserved}
Some properties of structures were already known to be preserved under taking model-complete cores: 
for instance, this is the case for $\omega$-categoricity~\cite{Cores-journal,BodHilsMartin-Journal}, for homogeneity~\cite[Proposition 4.7.7.(1)]{Book}, for the Ramsey property~\cite{Bod-New-Ramsey-classes}, 
and for the property to be a reduct
of a finitely homogeneous Ramsey structure~\cite{MottetPinskerCores}. 
Also, it is known that if $\bB$ is $\omega$-categorical, then for every $n$, the number of orbits of $n$-tuples in the model-complete core of $\bB$ is bounded by the number of orbits of $n$-tuples in $\bB$ (\cite[Proposition 4.7.7.(4)]{Book}; in $\omega$-categorical structures, orbits of $n$-tuples are in 1-to-1 correspondence with $n$-types over the empty set). In fact, the same holds for some other variants of orbit growth functions, such as the \emph{labelled} or \emph{unlabelled growth} of structures (see Definition~\ref{def:lu} for the latter). From this, one can prove that in the $\omega$-categorical setting, model theoretic properties that can be characterised in terms of bounds on certain orbit growth functions are preserved by going to the core companion  of a theory. This is the case, for example, for $\mathrm{NIP}$~\cite{Macpherson-RapidGrowth}, monadic stability~\cite{braunfeld2022monadic,bodor2025labelled}, and cellularity~\cite{bodor2024classification,bodirsky2021permutation}. However, this leads to indirect proofs under the additional assumption of $\omega$-categoricity, and only works for few model theoretic properties. In this section, we will prove that most model theoretic properties are preserved by going to the core companion  by general techniques. In particular, we summarise our results in the following theorem.

\begin{theorem}\label{thm:mc-pres} Let $T$ be a first-order theory. Suppose that $T$ has a core companion  $S$. Let $P$ be any of the properties listed in the first column of Table~\ref{table1}. Suppose that $T$ has $P$. Then, so does $S$. 

\end{theorem}

\def\t{1.1} 

\renewcommand{\arraystretch}{1.4}
\newcommand{\thickhline}{\noalign{\hrule height \t pt}}
\begin{center}
\begin{table}[]
\begin{tabular}{!{\vrule width \t pt}c|c|c!{\vrule width \t pt}}
\thickhline
\rowcolor[HTML]{C0C0C0} 
\textbf{\begin{tabular}[c]{@{}c@{}}Model theoretic property\\ preserved by the core companion\end{tabular}} & \textbf{Reference for definition} & \textbf{\begin{tabular}[c]{@{}c@{}}Techniques yielding \\ preservation result\end{tabular}} \\ \thickhline
stability  &  Definition~\ref{def:stability} &  $\spadesuit, \diamondsuit, \clubsuit$\\ \hline
$\mathrm{NIP}$ & Definition~\ref{def:NIP} & $\spadesuit, \diamondsuit, \clubsuit$\\ \hline
simplicity &~\cite[Definition 2.1]{bailetti2024walk} & $\diamondsuit$\\ \hline
$\mathrm{NTP_2}$ & \cite[Definition 2.1]{bailetti2024walk} & $\diamondsuit$ \\ \hline
${\mathrm{NSOP}_1=\mathrm{NSOP}_2}$ & {\cite[Definition 2.2]{dvzamonja2004maximality}} & $\diamondsuit$\\ \hline
$\mathrm{NSOP}_k$ {for $k\geq 3$} & Definition~\ref{def:NSOPn} & $\heartsuit$, ($\diamondsuit$ for $k= 3$)\\ \hline
$\mathrm{NSOP}$ + $\omega$-categoricity & Definition~\ref{def:nsop} & $\heartsuit$\\ \hline
$\lambda$-stability & Definition~\ref{def:superstab} & $\spadesuit, \heartsuit, \clubsuit$\\ \hline
superstability & Definition~\ref{def:superstab} & $\spadesuit, \heartsuit, \clubsuit$\\ \hline
not being multiply ordered & \cite[Theorem 4]{keisler1976six} & $\spadesuit, \heartsuit$\\ \hline
monadic stability & Definition~\ref{def:monadicNIP} & $\heartsuit$\\ \hline
monadic $\mathrm{NIP}$ & Definition~\ref{def:monadicNIP} & $\heartsuit$ \\ \hline
total transcendentality & Definition~\ref{def:MR} & $\spadesuit, \heartsuit, \clubsuit$ \\ \hline
strong minimality & Remark~\ref{rem:stronglymin}& $\heartsuit, \clubsuit$\\ \hline
$\mathrm{NIP}_k$ & \cite[Definition 2.4]{shelah2007definable} & $\diamondsuit, \clubsuit$\\ \hline
$\mathrm{NFOP}_k$ & \cite[Definition 2.1]{abd2025higher} & $\diamondsuit, \clubsuit$\\ \hline
strong dependence & \cite[Definition 1.2]{shelah2014strongly} & $\clubsuit$\\ \hline
Morley rank $\leq \mu$ & Definition~\ref{def:MR} & $\heartsuit, \clubsuit$\\ \hline 
$\mathrm{dp}$-rank $\leq \mu$ & \cite[Definition 4.12]{simon2015guide} & $\clubsuit$ \\ \hline 
finite 
$U$-rank & ~\cite[Definition 3.16]{pillay1996geometric} & $\clubsuit$ \\ \hline 
finite
$\mathrm{op}$-dimension & \cite[Definition 1.9]{guingona2015common} & $\clubsuit$\\ \hline
near linear Zarankiewicz bounds & \cite[Section 3]{walsberg2026tracedefinabilityiimodeltheoretic} & $\clubsuit$\\ \hline
strong Erd\H{o}s-Hajnal property & \cite[Section 3]{walsberg2026tracedefinabilityiimodeltheoretic} & $\clubsuit$\\ 
\thickhline
\end{tabular}

\vspace{1em}
Above, $k\in{\mathbb{N}_{>0}}$, $\lambda$ and $\mu$ are cardinals with $\lambda\geq\aleph_0$.

\vspace{1em}
{\sc Techniques for preservation results:}
\begin{itemize}
    \item[($\diamondsuit$)] semi-positive straight definitions. Subsection~\ref{sub:patterns};
    \item[($\heartsuit$)] variant of straight definitions. 
    Proposition~\ref{prop:nsopn}, Corollary~\ref{cor:NSOP}, and Subsections~\ref{sub:monadic} and~\ref{sub:finer};
    \item[($\spadesuit$)] type counting argument. Subsection~\ref{sub:typecounting};
    \item[($\clubsuit$)] trace definability. Subsection~\ref{sub:trace}.
\end{itemize}

\vspace{1em}
\caption{We list the various model theoretic properties that we prove to be preserved under core companions in this section. 
Note: we do not define all of these properties, but we give references to the relevant definitions. We favour references including multiple definitions to the original reference where a concept is introduced. In the last column we indicate which methods can be used to show the relevant preservation result.
}\label{table1}
\end{table}
\end{center}

 We develop several methods  to show a model theoretic property is preserved when moving to the
core companion. This results, for example, in three different proofs for why stability and $\mathrm{NIP}$ are preserved by taking the 
core companion: {firstly,} because they are both defined by omitting a (semi-positive) pattern (Subsection~\ref{sub:patterns}), {secondly,} because they can be characterised by bounds on the number of types (Subsection~\ref{sub:typecounting}), and {thirdly,} because they are preserved under trace definability (Subsection~\ref{sub:trace}). However, for example, simplicity can only be shown to be preserved by the methods of Subsection~\ref{sub:patterns} since it cannot be characterised by type counting or trace definability criteria~\cite{Walsbergshort}.

	Note that if we prove that a model-theoretic property is closed under taking model-complete cores, then we obtain in particular that it is closed under taking model companions, because we can apply the result to the expansion by all relations defined by negating atomic formulas; the model-companion of the original structure is then isomorphic to a reduct of the model-complete core of the expanded structure. For some of the properties we mention, such as stability and $\mathrm{NIP}$, it is probably a folklore result that they are preserved by moving to the model companion. This could be proven, for example, by a type counting argument of the form of those in Subsection~\ref{sub:typecounting}. However, our techniques have a higher degree of generality and cover several major model theoretic properties for which type counting arguments would not be available (e.g., simplicity and $\mathrm{NSOP}_k$ for $k\in\mathbb{N}$). 
    
    A notable model-theoretic property that is not covered by our techniques and for which we do not know whether it is preserved by taking core companions is $\mathrm{NSOP}$.
We thank~\cite{day2026notion} for pointing this out, allowing us to correct a mistake in an earlier draft of this paper. Nevertheless, we show that $\mathrm{NSOP}$ is preserved when moving to the core companion under the additional assumption of $\omega$-categoricity in Corollary~\ref{cor:NSOP}.

\subsection{{Notes} on monster models}
	In this section, we assume familiarity with basic terminology from model theory, and refer to the text-book of Tent and Ziegler~\cite{Tent-Ziegler}. For example, we will use notions such as `elementary embedding' without giving a definition; the other sections of this article can be read independently and do not rely on the notions that we use in this section and that are not defined. For facts about monster models in a positive setting we refer the reader to~\cite{kamsma2025positivelogicintroductionmodel}.
\begin{notation}\label{not_monster} Following standard model-theoretic conventions, we will speak of first-order formulas and theories having a given model theoretic property in what is known as a \textit{monster model} $\fa^\star$ for the given first-order theory $T$ (see, for example,~\cite[$\S$ 6.1]{Tent-Ziegler} for the relevant definitions). Such a model is chosen so that, for a sufficiently large infinite cardinal $\kappa$, $\fa^\star$ is $\kappa$-saturated:\footnote{Usually, a monster model is also chosen to be sufficiently strongly homogeneous. We will never make use of this additional property.} types of finite tuples over sets of parameters of cardinality $<\kappa$ are realised in $\fa^\star$. The fact that $\fa^\star$ is $\kappa$-saturated further implies that  it is $\kappa^+$-\textit{universal} in the sense that every model $\fa$ of cardinality $<\kappa^+$ elementarily embeds into $\fa^\star$. By a \textit{small} subset of the monster model $\fa^\star$, we mean a subset of cardinality $<\kappa$. We shall consider all sets of parameters, tuples, and models as being small and living inside the monster model and for a tuple $\bar{a}$, we write $\vDash \phi(\bar{a})$ to denote that $\fa^\star\vDash \phi(\bar{a})$.
\end{notation}

The main use of monster models we  make in this paper is that most model theoretic properties are only witnessed at a sufficient level of saturation\footnote{$\omega$-saturation will be sufficient in all of our uses.} and that working in a monster model spares us from some bookkeeping when moving to small elementary extensions of some small model.\\

Since sometimes we will be working with a positive theory, we will need an equivalent notion of monster model in this context. This is standard practice~\cite{ben2003positive, PosModT}. We need to introduce a couple of notions:

\begin{definition} Let $\fa$ be a positively closed model of the positive theory $T$. Let $B\subseteq \fa$ be a set of parameters and $a$ be a tuple from $A$. The \textit{positive type} of $\bar{a}$ over $B$ is the set of existential positive formulas satisfied by $\bar{a}$ with parameters from $B$:
\[\mathrm{tp}^+(\bar{a}/B)=\{\phi(\bar{x},b)\vert \  \fa\vDash \phi(\bar{a},\bar{b}), \bar{b}\in B^k, k\in\mathbb{N}, \phi(\bar{x},\bar{y}) \text{ is existential positive}\}.\]
\end{definition}

\begin{definition} A first-order theory $T$ has the \textit{joint homomorphism property} (JHC) if for any models $\fa, \fb$ there is some model $\fc$ such that both $\fa$ and $\fb$ map homomorphically to $\fc$. The joint homomorphism property can be thought of as the positive analogue of completeness for a theory. Any positive theory $T$ can be extended to a positive theory $T'$ with the joint homomorphism property by adding to it the set of $h$-inductive sentences true in one of its positively closed models. 
\end{definition}

\begin{notation} From now on, by a first-order theory, we will always refer to a complete first-order theory. Similarly, every positive theory we consider has the joint homomorphism property. Note that the core companion  of a complete theory also has the joint homomorphism property. 
\end{notation}

\begin{notation} When a positive theory $T$ has the joint homomorphism property, we can build for arbitrarily large infinite cardinals $\kappa$ a positive monster model $\fa^{(\star, +)}$. This is a positively closed model which is positively $\kappa$-saturated (in the sense that positive types over $\leq\kappa$-many parameters are realised) for sufficiently large $\kappa$. 
Again such a monster model will be positively $\kappa^+$-universal in the sense that any positively closed model of cardinality $\leq \kappa$ maps homomorphically to $\fa^{(\star, +)}$. 
\end{notation}

\begin{remark}\label{rem:samemonster}
    Note that every model of a model-complete core theory $S$ is positively closed. This is because every homomorphism between models of $S$ is an elementary embedding and so an immersion. Moreover, any complete model-complete core theory has the joint homomorphism property. For this reason any monster model $\fa^\star$ for $S$ as a standard first-order theory is also a positive monster model for $S$ as a positive theory (i.e., for the positive theory given by the set of $h$-inductive consequences of $S$, which as remarked above axiomatises $S$). When working with a monster model $\fb^\star$ of a complete theory $T$ and a monster model $\fa^\star$ of its core companion  $S$, we choose these so that there is a homomorphism $f \colon \fa^\star\to\fb^\star$. This is always achievable using the fact that $S$ and $T$ have the same $h$-universal consequences and by choosing $\fb^\star$ to be sufficiently saturated. 
\end{remark}

\subsection{Model theoretic properties omitting patterns}\label{sub:patterns}

Model theory has developed several notions of tameness for a first-order theory $T$. Most of these properties are defined in terms of formulas not coding some forbidden configuration (which then ensures they
are preserved under interpretation). In this section, we shall work in a general framework to describe the shape of a model theoretic property defined in terms of some omitted pattern. Such a framework was originally suggested by Shelah~\cite[Definition 5.17]{saharon2000not} and mentioned in~\cite{kruckman2024new}. We will follow the presentation given by Bailetti~\cite{bailetti2024walk}, who considerably elaborates on it. 
In this subsection, we prove that all model theoretic properties defined in terms of formulas omitting a semi-positive straight pattern (in the sense of Definitions~\ref{def:patternedprop} and~\ref{def:semipositive}) are preserved when going to the core companion  of a theory.
This includes many mainstream model theoretic properties appearing on the model theoretic map of the universe available on \href{www.forkinganddividing.com}{forkinganddividing.com} (i.e., stability, $\mathrm{NIP}$, 
simplicity, $\mathrm{NTP_2}$).

Given the wide range of model theoretic properties 
that are 
preserved by going to the core companion  of a theory and the fact that we adopt the general framework of~\cite{bailetti2024walk} to talk about model theoretic properties, we will only formally introduce two of these properties. 
We have chosen 
\emph{stability} (Definition~\ref{def:stability} below), and $\mathrm{NIP}$ (Definition~\ref{def:NIP}). These are two of the most studied model-theoretic properties~\cite{pillay1996geometric, simon2015guide}. We refer the reader to~\cite{bailetti2024walk} for the definitions of the other model theoretic properties we discuss in this section.

\begin{notation} When writing a $\tau$-formula as $\phi(\bar{x};\bar{y})$, we mean that we consider the variables of $\phi$ as partitioned into two tuples of variables $\bar{x}$ and $\bar{y}$.
\end{notation}

\begin{definition}\label{def:stability} Let $T$ be a first-order theory. We say that a $\tau$-formula $\phi(\bar{x};\bar{y})$ has the \emph{order property} $\mathrm{OP}$ if there are\footnote{Recall from Notation~\ref{not_monster} that parameters are taken to live in a monster model $\fa^\star$ of $T$. We will keep this notation throughout this section. Later, in the context of positive theories, we will take the parameters from a positive monster model $\fa^{(\star, +)}$.} $(\bar{a}_i \bar{b}_i)_{i\in\mathbb{N}}$ such that
\[\vDash \phi(\bar{a}_i, \bar{b}_j) \text{ if and only if } i<j\;.\]
We say that a theory $T$ is \emph{stable}, or $\mathrm{NOP}$, if no $\tau$-formula has the order property. Note that from our definition the theory of a finite structure is stable.
\end{definition}

	The following structures have stable theories: $(\mathbb{N};=)$, the Johnson graph (see Definition~\ref{def:Johnson}),  an equivalence relation with infinitely many infinite classes,  infinite vector spaces over finite fields, and algebraically closed fields of fixed characteristic.

\begin{definition}\label{def:NIP}
  Let $T$ be a first-order theory. We say that a $\tau$-formula $\phi(\bar{x};\bar{y})$ has the \emph{independence property} $\mathrm{IP}$ if there are $(\bar{a}_i)_{i\in\mathbb{N}}$ and $(\bar{b}_S)_{S\subseteq\mathbb{N}}$ such that 
    \[\vDash \phi(\bar{a}_i; \bar{b}_S) \text{ if and only if } i\in S.\]
    We say that $T$ is \emph{dependent}, or $\mathrm{NIP}$ if no $\tau$-formula has the independence property.
\end{definition}

	It is easy to see that stability implies $\mathrm{NIP}$ for a first-order theory. Some structures whose theory is $\mathrm{NIP}$ (and not stable) are  $(\mathbb{Q};<)$, the dense local order $S(2)$ (see Notation~\ref{not:S2}), and real closed fields.

As one can see, stability is defined in terms of formulas not coding a copy of the half-graph; that is, the {bipartite graph with bipartition $(\mathbb{N}\times\{1\},\mathbb{N}\times\{2\})$} and where ${((i,1),(j,2))}$ forms an edge if and only if $i<j$. Similarly, $\mathrm{NIP}$ is defined in terms of formulas not coding some bipartite graph.
Indeed, a common feature of several model-theoretic tameness properties is that they can be defined in terms of $\tau$-formulas in a theory not coding some forbidden configuration. For this reason, most model theoretic tameness properties are known by an acronym starting with an ``N'' to indicate the word ``not'' and end with the letter ``P'' to indicate the word ``property''. Below, we list some of the main classes of model-theoretically tame theories, other than stable and $\mathrm{NIP}$ theories:
\begin{itemize}
    \item $\mathrm{NSOP}$ theories, where $\mathrm{SOP}$ stands for the \emph{strict order property} (Definition~\ref{def:nsop});
    \item \emph{simple theories}, also known as $\mathrm{NTP}$ theories, where $\mathrm{TP}$ is the \emph{tree property}. More specifically, the tree property $\mathrm{TP}$ is defined in terms of having the $k$-tree property $k\textnormal{-}\mathrm{TP}$ for some $k\in\mathbb{N}$. For more on simple theories, the reader may consult~\cite{wagner2000simple, kim2013simplicity};
    \item $\mathrm{NTP}_2$ theories, where $\mathrm{TP}_2$ is the tree property of the second kind. Again, $\mathrm{TP}_2$ is defined in terms of having the property $k\textnormal{-}\mathrm{TP}_2$ for some $k\in\mathbb{N}$; 
    \item $\mathrm{NSOP}_n$ theories for $n\geq 1$, where $\mathrm{SOP}_n$ is the \emph{$n$-th strong order property} (Definition~\ref{def:NSOPn}).
\end{itemize}
All of the above properties emerged from the work of Shelah~\cite{shelah1990classification, shelah1980simple, shelah1995toward} and Shelah and D{\v{z}}amonja for $\mathrm{NSOP}_1$ and $\mathrm{NSOP}_2$~\cite{dvzamonja2004maximality}.\footnote{It was recently shown in~\cite{mutchnik2025textup} that $\mathrm{NSOP}_1$ and $\mathrm{NSOP}_2$ are the same properties for a first-order theory.} Rather than defining each of these properties individually, following~\cite{bailetti2024walk}, we will define the abstract notion of being a patterned property and point out that most model theoretic properties fall under this notion. 

\begin{definition}\label{def:patternedprop} Let $n\in\mathbb{N}$. An ${n}$\emph{-pattern} $(\mathcal{C}, \mathcal{I})$ consists of two sets $\mathcal{C}, \mathcal{I}\subseteq(\mathcal{P}(n)\times\mathcal{P}(n))\setminus\{(\emptyset,\emptyset)\} $. We call a pair $(A^+, A^-)\in\mathcal{C}$ a \emph{consistency condition} and a pair $(Z^+, Z^-)\in\mathcal{I}$ an \emph{inconsistency condition}.
\end{definition}
\begin{definition} Let $T$ be a first-order theory. Let $\phi(\bar{x};\bar{y})$ be a $\tau$-formula partitioned into the tuples of variables $\bar{x}$ and $\bar{y}$. Let $(\mathcal{C}, \mathcal{I})$ be an $n$-pattern. We say that $\phi(\bar{x};\bar{y})$ \emph{exhibits} $(\mathcal{C}, \mathcal{I})$ if there is a sequence of tuples $(\bar{b}_i)_{i<n}$ (living in the monster model) such that 
\begin{enumerate}
    \item for all $A=(A^+, A^-)\in\mathcal{C}$, the partial type
    \[\{\phi(\bar{x};\bar{b}_i)\mid i\in A^+\}\cup\{\neg\phi(\bar{x};\bar{b}_j)\mid j\in A^-\}\]
    is consistent; and
    \item for all $Z=(Z^+, Z^-)\in\mathcal{I}$, the partial type
    \[\{\phi(\bar{x};\bar{b}_i)\mid i\in Z^+\}\cup\{\neg\phi(\bar{x};\bar{b}_j)\mid j\in Z^-\}\]
    is inconsistent.
\end{enumerate}
 By a \emph{straight pattern},
 we mean a set
 \begin{equation}\label{eq:straightpattern}
    \mathcal{P}=\{(\mathcal{C}_n, \mathcal{I}_n) \mid n\in\mathbb{N}\}\;, 
 \end{equation}
    where for each $n\in\mathbb{N}$, $(\mathcal{C}_n, \mathcal{I}_n)$ is an $f(n)$-pattern, where $f$ is some function $f \colon \mathbb{N}\to\mathbb{N}$. We say that a formula $\phi(\bar{x};\bar{y})$ \emph{exhibits the straight pattern} $\mathcal{P}$ if it exhibits $(\mathcal{C}_n, \mathcal{I}_n)$ for all $n\in\mathbb{N}$. We say that a $\tau$-theory $T$ has $N\mathcal{P}$ if no $\tau$-formula $\phi(\bar{x};\bar{y})$ has $\mathcal{P}$. 
\end{definition}

\begin{definition}\label{def:semipositive}
	We say that a straight pattern $\mathcal{P}$ as in~\ref{eq:straightpattern} is \emph{semi-positive} if for each $n\in\mathbb{N}$ and $(Z_n^+, Z_n^{-})\in \mathcal{I}_n$, $Z_n^{-}=\emptyset$. We say that $\mathcal{P}$ is \emph{positive} if it is semi-positive and we also have that for each $(A_n^+, A_n^{-})\in \mathcal{C}_n$, $A_n^{-}=\emptyset$. We say that a model theoretic property $P$ of a formula is \emph{straightly definable} if there is a straight pattern $\mathcal{P}$ such that a given formula $\phi$ has P if and only if it exhibits $\mathcal{P}$.
This notion can be extended to properties of theories in the obvious way. We similarly define positive straight definability and semi-positive straight definability.
\end{definition}

\begin{example} It is easy to observe by compactness that $\mathrm{OP}$ is 
semi-positively straightly definable. In fact, the following holds~\cite{bailetti2024walk}: a $\tau$-formula $\phi(\bar{x};\bar{y})$ has $\mathrm{OP}$ if and only if it exhibits the straight pattern $\mathcal{OP}:=\{(\mathcal{C}_n, \mathcal{I}_n)\mid \ n\in\mathbb{N}\}$ where $(\mathcal{C}_n, \mathcal{I}_n)$ is the $n$-pattern with
\[\mathcal{C}_n =\{(\{i,i +1,\dots,n -1\},\{0,\dots,i-1\}) \mid i < n\}, \text{ and } I_n =\emptyset\;.\]
\end{example}
More generally, most model theoretic properties are straightly definable. Below, we record the fact that a large number of model theoretic properties are positively straightly definable. The majority of the properties below can be seen more straightforwardly to be semi-positively straightly definable (cf.~\cite[Proposition 3.3]{bailetti2024walk}):

\begin{fact}[\cite{bailetti2024walk, day2025results, shelah2008more}] Each of the properties 
\[\mathrm{XP}\in\{\mathrm{OP}, \mathrm{IP}, k\textnormal{-}\mathrm{TP}, k\textnormal{-}\mathrm{TP}_2, \mathrm{SOP}_1, \mathrm{SOP}_2, \mathrm{SOP}_3\} \]
is positively straightly definable. 
\end{fact}
Whilst $\mathrm{SOP}_n$ for $n\geq 4$ and $\mathrm{SOP}$ are straightly definable~\cite{bailetti2024walk, day2026notion}, these properties are not positively straightly definable~\cite{bailetti2024walk}.

We now want to define a way to uniformly convert a semi-positive straight pattern $\mathcal{P}$ into what {we} shall call its \emph{positivisation} $\mathcal{P}^+$ in the sense that
\begin{itemize}
    \item $\mathcal{P}^+$ makes sense as a property of an existential positive formula in a positive theory, and  
    \item in a first-order theory where every formula is equivalent to an existential positive formula, $\mathcal{P}^+$ is equivalent to $\mathcal{P}$. 
\end{itemize}
    Fortunately, there is already some literature in positive model theory dedicated to precisely this problem~\cite{PosModT}.
    For example, the standard positivised version of the order property (and hence of stability) is as follows.

\begin{example} Let $T$ be a positive theory with the joint homomorphism property. We say that an existential positive formula $\phi(\bar{x};\bar{y})$ has OP$^+$ if there is a sequence $(\bar{a}_i\bar{b}_i)_{i\in\mathbb{N}}$ (from a positive monster model of $T$) and an obstruction $\psi(\bar{x};\bar{y})$ of $\phi(\bar{x};\bar{y})$ such that 
\begin{align*} & \vDash \phi(\bar{a}_i; \bar{b}_j)  \text{ for } i< j \\
\text{ and } \quad & \vDash \psi(\bar{a}_i; \bar{b}_j)  \text{ for } j\leq i.
\end{align*}
\end{example}

\begin{definition}\label{def:positivisation}
    Let $\mathcal{P}:=\{(\mathcal{C}_n, \mathcal{I}_n) \mid n\in\mathbb{N}\}$ be a semi-positive straight pattern. Let  $Q\subseteq \mathbb{N}$ be 
    \[Q:=\{|Z^+| : (Z^+, \emptyset)\in\mathcal{I}_n \text{ for some }n\in\mathbb{N}\}\;.\]
    Let $T$ be a positive theory with the joint homomorphism property. We say that an existential positive formula $\phi(\bar{x};\bar{y})$ exhibits the \emph{positivisation $\mathcal{P}^+$ of $\mathcal{P}$} if there is an obstruction $\psi(\bar{x};\bar{y})$ of $\phi$ and for each $q\in Q$, an obstruction $\theta_{q}(\bar{y}_1, \dots, \bar{y_q})$ of $\exists x\bigwedge_{i=1}^q\phi(\bar{x}; \bar{y}_i)$ such that for every $n\in\mathbb{N}$ there is a sequence $(\bar{b}_i)_{i<f(n)}$ of {tuples} such that
\begin{enumerate}
    \item for every $A=(A^+, A^-)\in\mathcal{C}_n$, the partial type
    \[\{\phi(\bar{x};\bar{b}_i)\mid i\in A^+\}\cup\{\psi(\bar{x};\bar{b}_j)\mid j\in A^-\}\]
    is consistent; and
    \item for every $Z=(Z^+, \emptyset)\in\mathcal{I}_n$,
    where $Z^+=\{i_1, \dots, i_q\}$ has cardinality $q$,
    \[\vDash\theta_{q}(\bar{b}_{i_1}, \dots, \bar{b}_{i_q}).\]
\end{enumerate}
    We say that $T$ has $N\mathcal{P}^+$ if no existential positive formula in $T$ exhibits $\mathcal{P}^+$.
 \end{definition}

 \begin{remark} Whilst this is not strictly necessary for our purposes, our definition of the positivisation of a property is compatible with the current literature in positive model theory. In particular, for $\mathcal{P}\in\{\mathrm{OP}, \mathrm{IP}, k\textnormal{-}\mathrm{TP}, k\textnormal{-}\mathrm{TP}_2, \mathrm{SOP}_1, \mathrm{SOP}_2, \mathrm{SOP}_3\}$, we have that $\mathcal{P}^+$ as defined in Definition~\ref{def:positivisation} agrees with the definition of the corresponding model theoretic property in the standard positive model theory literature~\cite{PosModT}.
 \end{remark}

Note that when we say that a first-order complete theory has a positive model-theoretic property, we mean that its positive fragment has that model-theoretic property.

\begin{lemma}\label{negpres} Suppose that every model of $T$ has a homomorphism to a model of $S$. If an existential positive formula $\psi(\bar{x})$ is an obstruction of an existential positive formula $\phi(\bar{x})$ in $S$, then $\psi(\bar{x})$ is an obstruction of $\phi(\bar{x})$ in $T$. 
\end{lemma}
\begin{proof}
    Suppose by contrapositive $\psi$ is not an obstruction of $\phi$ in $T$. Then  there is a model $\fa$ of $T$ and $a\in A$ such that 
\[\fa\vDash \phi(\bar{a})\wedge\psi(\bar{a}).\]
By assumption, there is a homomorphism $f$ from $\fa$ to a model $\fb$ of $S$. Now, by moving the existential quantifier to the front, $\phi(\bar{x})\wedge\psi(\bar{x})$ is equivalent to an existential positive formula and so its truth for a tuple is preserved by homomorphisms. In particular, 
\[\fb\vDash \phi(f(\bar{a}))\wedge\psi(f(\bar{a})).\]
Thus, $\psi$ is not an obstruction of $\phi$ in $S$. This proves the contrapositive. 
\end{proof}

\begin{lemma}\label{pospres} Suppose that $S$ and $T$ are positive theories with the joint homomorphism property. Suppose that $S$ and $T$ have the same $h$-universal consequences. Let $\mathcal{P}$ be a semi-positive straight pattern. Then
$S$ has $\mathcal{P}^+$ if and only if $T$ has $\mathcal{P}^+$.
\end{lemma}
\begin{proof}
    This is a direct consequence of Lemma \ref{negpres} and Definition~\ref{def:positivisation}. Let $\phi(\bar{x};\bar{y})$ exhibit $\mathcal{P}^+$ in $S$. Let $\psi(\bar{x};\bar{y})$, and  $\{\theta_q(\bar{y}_1, \dots, \bar{y}_q) \mid q\in Q\}$ be as in Definition~\ref{def:positivisation} and $f$ the function from the definition of the straight pattern ${\mathcal P}$. Fix $n\in\mathbb{N}$. Then, there are $(\bar{b}_i)_{i<f(n)}$ witnessing $(\mathcal{C}_n, \mathcal{I}_n)$ as described in Definition~\ref{def:positivisation}. In particular, for each $(A^+, A^-)\in\mathcal{C}_n$, there is $\bar{a}_{(A^+, A^-)}$ living in the positive monster model $\fa^{(\star, +)}$ for $S$ such that
\[
 \bigwedge_{i\in A^+} \phi(\bar{a}_{(A^+, A^-)}; \bar{b}_i)\wedge \bigwedge_{j\in A^-} \psi(\bar{a}_{(A^+, A^-)}; \bar{b}_j)
\]
and also, for each $(Z^+, \emptyset)\in\mathcal{I}_n$, 
\[ \vDash\theta_{q}(\bar{b}_{i_1}, \dots, \bar{b}_{i_q}),\]
where $Z^+=\{i_1, \dots, i_q\}$. We can take the $(\bar{b}_i)_{i<f(n)}$  and the $\bar{a}_{(A^+, A^-)}$ for $(A^+, A^-)\in\mathcal{C}_n$ to live inside a small model $\fa\models S$.  Let $f \colon \fa\to\fb$ be a homomorphism from $\fa$ to a small model $\fb\vDash T$. Now, since all of $\phi, \psi$, and $\{\theta_{q}(\bar{y}_1, \dots, \bar{y}_{q}) \mid q\in Q\}$,  are existential positive, and so preserved by homomorphisms, we still have that the sequence $(f(\bar{b}_i))_{i<f(n)}$ from $\fb$ witnesses the fact that $\phi$ exhibits $(\mathcal{C}_n, \mathcal{I}_n)$ with respect to the same  $\phi, \psi$,  and $\{\theta_{q}(\bar{y}_1, \dots, \bar{y}_q) \mid {q}\in Q\}$. In particular, $\phi(\bar{x};\bar{y})$ exhibits $\mathcal{P}^+$ for $T$.
\end{proof}
The following is a consequence of the above lemma and the definition of core companion.

\begin{corollary}\label{corepres} Suppose that $S$ is the core companion  of $T$ and let $\mathcal{P}$ be a semi-positive straight pattern. Then $S$ has $\mathcal{P}^+$ if and only if $T$ has $\mathcal{P}^+$.
\end{corollary}


\begin{lemma}\label{normiffpos} Let $S$ be a model-complete core theory and let $\mathcal{P}$ be a semi-positive straight pattern. Then 
an {existential positive} formula $\phi(\bar{x};\bar{y})$  has $\mathcal{P}$ if and only if it has $\mathcal{P}^+$.
\end{lemma}
\begin{proof} This is just a consequence of the fact that every formula in a model-complete core theory is equivalent to an existential positive one (see, for example,  \cite[Proposition 2.6.13]{bodirsky2021complexity}). Suppose some formula $\phi(\bar{x};\bar{y})$ exhibits $\mathcal{P}$. Then, $\phi(\bar{x};\bar{y})$, $\neg\phi(\bar{x};\bar{y})$, and for each
$q\in Q$, 
\[\neg\exists \bar{x}\bigwedge_{i=1}^q\phi(\bar{x}; \bar{y}_i)\]
are equivalent to existential positive formulas modulo $S$, meaning that $\phi(\bar{x};\bar{y})$ also witnesses having $\mathcal{P}^+$. 
The right to left implication is trivial by Definition~\ref{def:positivisation}, but, crucially, also the reason for which we require $\mathcal{P}$ to be semi-positive rather than arbitrary straight pattern. Indeed, fix $n\in\mathbb{N}$ and suppose that there is a sequence $(\bar{b}_i)_{i<f(n)}$ and formulas $\psi$ and {$\theta_q$ for $q\in Q$} as in Definition~\ref{def:positivisation}. Given $A=(A^+, A^-)\in\mathcal{C}_n$, if
\[\{\phi(\bar{x};\bar{b}_i)\mid i\in A^+\}\cup\{\psi(\bar{x};\bar{b}_j)\mid j\in A^-\}\]
is consistent, then so is 
\[\{\phi(\bar{x};\bar{b}_i)\mid i\in A^+\}\cup\{\neg\phi(\bar{x};\bar{b}_j)\mid j\in A^-\}\]
since $\psi$ is an obstruction to $\phi$. For $Z=(Z^{+}, \emptyset)\in\mathcal{I}_n$, with $Z^+=\{i_1, \dots, i_q\}$, we know that $\vDash \theta_q(\bar{b}_{i_1}, \dots, \bar{b}_{i_q})$. Hence, since $\theta_q(\bar{y}_1, \dots, \bar{y}_q)$ is an obstruction to $\exists x\bigwedge_{i=1}^q\phi(\bar{x}; \bar{y}_i)$, 
 \[\{\phi(\bar{x};\bar{b}_i)\mid i\in Z^+\}\]
    is inconsistent. This concludes the argument that exhibiting $\mathcal{P}^+$ implies exhibiting $\mathcal{P}$.
\end{proof}

\begin{theorem}\label{thm:inheritance} Suppose that $S$ is the core companion  of $T$ and let $\mathcal{P}$ be a semi-positive straight pattern. If $T$ has $\mathrm{N}\mathcal{P}$, then so does $S$.
\end{theorem}
\begin{proof} We prove the contrapositive. Suppose that $S$ has $\mathcal{P}$. By Lemma~\ref{normiffpos}, $S$ has $\mathcal{P}^+$. Now, by Lemma \ref{pospres}, since $S$ and $T$ have the same $h$-universal consequences by definition, $T$ has $\mathcal{P}^+$. Hence, $T$ has $\mathcal{P}$, since, by definition, $\mathcal{P}^+$ implies $\mathcal{P}$.
\end{proof}

Hence, the core companion  preserves any model-theoretic property that is defined in terms of formulas not exhibiting a given semi-positive straight pattern. We note that whilst the properties $\mathrm{NSOP}_k$ for $k\geq 4$ are
not known to be semi-positively straightly definable (and are known to not be positively straightly definable~\cite{bailetti2024walk}) these are still amenable to the core ideas of our proof.

\begin{definition}\label{def:NSOPn} Let $\phi(\bar{x},\bar{y})$ be a formula with $|\bar{x}|=|\bar{y}|$. Let $n\geq 3$. We say that $\phi(\bar{x},\bar{y})$ has the \emph{$n$-strong order property} $\mathrm{SOP}_n$ if
\[\{\phi(\bar{x}_1; \bar{x}_2), \dots,\phi(\bar{x}_{n-1}; \bar{x}_{n}), \phi(\bar{x}_n; \bar{x}_1)\} \]
is inconsistent, and there is a sequence $(\bar{a}_i)_{i\in\mathbb{N}}$ such that $\phi(\bar{a}_i; \bar{a}_j)$ holds for all $i<j$ in $\mathbb{N}$. We say that  a theory $T$ is $\mathrm{NSOP}_n$ if none of its formulas has $\mathrm{SOP}_n$.
\end{definition}

\begin{proposition}\label{prop:nsopn} Let $S$ be the core companion  of $T$. Fix $n\geq 3$. If $T$ is $\mathrm{NSOP}_n$, then so is $S$.
\end{proposition}
\begin{proof} As always, we prove the statement by contrapositive. Suppose that $S$ witnesses $\mathrm{SOP}_n$ with the formula $\phi(\bar{x},\bar{y})$ (which may be chosen to be existential positive) and the sequence $(\bar{a}_i)_{i\in\mathbb{N}}$. Let 
$\theta$ be the existential positive formula equivalent to $\neg\delta$, 
where
\[\delta:=\exists \bar{x}_1, \dots \bar{x}_n \left(\bigwedge_{1\leq i<n} \phi(\bar{x}_i; \bar{x}_{i+1})\wedge\phi(\bar{x}_n; \bar{x}_1)\right)\;.\]
Let $f \colon\fa^\star\to\fb^\star$ be a homomorphism from a monster model $\fa^\star$ of $S$ to a monster model $\fb^\star$ of $T$. Then, since $\theta$ is still an obstruction of $\delta$ in $T$ and since both are existential positive formulas (modulo a tautology for $\delta$), we can deduce that $\delta$ is still false in $T$. Moreover, for any $i<j\in\mathbb{N}$, $\phi(f(\bar{a}_i); f(\bar{a}_{j}))$ still holds since $\phi$ is an existential positive formula.
\end{proof}

	Next we show that for $\omega$-categorical theories $\mathrm{NSOP}$ is also preserved by taking core companions. We do not know whether the same holds in general. Below, we give the definition of the strict order property of~\cite{shelah1995toward}. 




\begin{definition}\label{def:nsop}
	A formula $\phi(\bar{x};\bar{y})$ with $|\bar{x}|=|\bar{y}|$ has the \emph{strict order property} $\mathrm{SOP}$ if it defines a strict partial order with an infinite chain
    on $|\bar{x}|$-tuples. 
  A theory is $\mathrm{NSOP}$ if it has no $\mathrm{SOP}$ formula.
\end{definition}


\begin{lemma}\label{lem:NSOPomega}
	Let $T$ be an $\omega$-categorical theory. Then {a formula $\phi(\bar{x};\bar{y})$ in $T$ with $|\bar{x}|=|\bar{y}|$ has $\mathrm{SOP}$ if and only if} 
    $\phi(\bar{x};\bar{y})$ 
has $\mathrm{SOP}_n$ for all $n\geq 3$.
\end{lemma}

\begin{proof}
	{The forward implication is trivial.} For the other direction, for ${k\coloneqq}|\bar{x}|=|\bar{y}|$, let us consider the directed graph $G$ on $k$-tuples from the monster model $\fa^\star$ of $T$ whose directed edge relation is defined by the formula $\phi(\bar{x};\bar{y})$. Note that $\phi$ can be chosen so that the directed graph $G$ has no loops and is antisymmetric. For $n\in \mathbb{N}_{>0}$ 
    let us consider the binary relation $E_n(\bar{x};\bar{y})$ on $k$-tuples expressing that there is a path of length $n$ from $\bar{x}$ to $\bar{y}$ in $G$. By our assumption we know that each relation $E_n$ is irreflexive and has infinite chains. On the other hand, by $\omega$-categoricity we know that there are only finitely many $2k$-ary relations definable in $T$. This means that $E_i=E_j$ for some $i<j$, and then it is easy to see that in fact $E_{i+\ell}=E_{j+\ell}$ for all $\ell\in \mathbb{N}$. This implies that $R\coloneqq \bigcup_{i\leq n< j}E_n=\bigcup_{i\leq n}E_n$ is definable and transitive.
Moreover, it is clear that $R$ 
has infinite chains; and it is irreflexive since the {relations $E_n$ are}. This shows that $R$ witnesses that {$\phi(\bar{x};\bar{y})$} has $\mathrm{SOP}$.
\end{proof}

\begin{corollary}\label{cor:NSOP} 
	Let $T$ be $\omega$-categorical and $\mathrm{NSOP}$. Then its core companion $S$ is also $\mathrm{NSOP}$.
\end{corollary}
\begin{proof}
	We prove the contrapositive. Suppose that $S$ has the strict order property. Let $\phi(\bar{x};\bar{y})$ be an existential positive formula witnessing this. Clearly, $\phi(\bar{x};\bar{y})$ also has $\mathrm{SOP}_n$ for all $n\geq 3$ in $S$. Thus, by the proof of Proposition~\ref{prop:nsopn}, $\phi(\bar{x}; \bar{y})$ has $\mathrm{SOP}_n$ for all $n\geq 3$ also in $T$. From  Lemma~\ref{lem:NSOPomega}, we can deduce that $\phi$ has the $\mathrm{SOP}$, and so $T$ has the $\mathrm{SOP}$.
\end{proof}

\begin{remark} As pointed out in~\cite{day2026notion}, there is a straightforward generalisation of straight definability and semi-positive straight definability for higher arity model theoretic properties such as $\mathrm{NIP}_k$~\cite{shelah2007definable,chernikov2019n} or $\mathrm{NFOP}_k$~\cite{terry2021higher,abd2025higher}. {Both of these properties} can be seen as being semi-positively straightly definable~\cite[Proposition 3.3]{day2025results}. For this reason, we include $(\diamondsuit)$ in the relevant rows of Table~\ref{table1}. 
\end{remark}

\subsection{Model theoretic properties defined through type-counting}\label{sub:typecounting}
Certain model theoretic properties such as $\omega$-stability, superstability, and stability have natural definitions in terms of bounding the number of types over sets of parameters of given cardinalities. With further set theoretic assumptions (or by considering types over finite sets of parameters), also $\mathrm{NIP}$ can be characterised in terms of type counting~\cite[Theorem 4.6]{shelah1971stability}. 
We will see below that given a theory $T$ and its core companion  $S$, the supremum of the number of types that can be attained over a set of parameters of size $\lambda$ can only decrease when moving from $T$ to $S$. Hence, model theoretic properties which can be defined in terms of type counting are preserved by going to the core companion. 

\begin{lemma}\label{lem:typecounting} Let $S$ and $T$ be first-order theories in a {signature} $\tau$ such that $S$ is the core companion  of $T$. Let $\fa^\star$ and $\fb^\star$ be, respectively, monster models for $S$ and $T$ with a homomorphism $f \colon \fa^\star\to \fb^\star$. Let $C$ be a small subset of $\fa^\star$. Let $S_1(C)$ be the space of $1$-types of $S$ over $C$. Then, 
\[|S_1(C)|\leq |S_1(f(C))|\;.\]
\end{lemma}
\begin{proof} Let $\mu\coloneqq |S_1(C)|$. Let $(a_i)_{i< \mu}$ be a set of elements of $\fa^\star$, each realising a distinct type in $S_1(C)$. For each $i<\mu$, let $\mathrm{tp}^+(a_i/C)$ be the restriction of $\mathrm{tp}(a_i/C)$ to existentially positive formulas. Since $S$ is a model-complete core theory, every formula is equivalent modulo $S$ to an existential positive formula. In particular,  writing $S_1^+(C)$ for the space of types over $C$ restricted to existential positive formulas, the ``forgetful'' map $S_1(C)\mapsto S^+_1(C)$ restricting each type is a bijection.
By Lemma~\ref{negpres} this implies that the positive types $(\mathrm{tp}^+(f(a_i)/f(C))\vert i< \mu)$ are pairwise distinct. In particular, the types $(\mathrm{tp}(f(a_i)/f(C))\vert i< \mu)$ are pairwise distinct, yielding the desired statement.
\end{proof}

It is well-known that our earlier definition of stability coincides with the one below~\cite{shelah1990classification}:

\begin{definition}\label{def:superstab} Let $\lambda$ be an infinite cardinal. We say that a first-order theory $T$ is $\lambda$\textit{-stable} if whenever $C$ is a set of parameters such that $|C|\leq \lambda$, then $|S_1(C)|\leq \lambda$. We say that $T$ is \textit{stable} if it is $\lambda$-stable for some infinite cardinal $\lambda$. We say that $T$ is \textit{superstable} if it is $\lambda$-stable for all $\lambda\geq 2^{|T|}$.
\end{definition}

The following is a direct consequence of the above definitions and Lemma~\ref{lem:typecounting}:
\begin{corollary}\label{cor:supstabpres} Let $T$ be a first-order theory and $S$ be its core companion. Let $\lambda$ be an infinite cardinal. Then, if $T$ is $\lambda$-stable, then so is $S$. In particular, if $T$ is superstable/stable, then so is $S$.
\end{corollary}

\begin{remark} As mentioned earlier, $\mathrm{NIP}$ can also be characterized in terms of type counting under additional set theoretic assumptions, or by type-counting over finite sets~\cite{shelah1971stability}.

Hence, Lemma~\ref{lem:typecounting} also yields a different proof that $\mathrm{NIP}$ is preserved when moving to the core companion  of a theory. Consider the \textit{stability function} $f_T \colon \mathrm{Card}\to\mathrm{Card}$, of a complete first-order theory $T$, where $f_T$ is given by
    \[\lambda\mapsto\sup\{|S_1(\fa)|\ \vert \ \fa\models T, |\fa|= \lambda\}\;.\]
Lemma~\ref{lem:typecounting} yields that if $S$ is the core companion  of $T$, then $f_S(\lambda)\leq f_T(\lambda)$ for all cardinals $\lambda$. In the more specific context of complete countable first-order theories, there are only six possible behaviours for $f_T$~\cite{keisler1976six}:
\[\lambda, \lambda+2^{\aleph_0}, \lambda^{\aleph_0}, \mathrm{ded}(\lambda), (\mathrm{ded}(\lambda))^{\aleph_0}, 2^\lambda\;.\]
These are all witnessed (in a suitable model of set theory). Indeed, we have that:
\begin{itemize}
    \item $f_T(\lambda)>\lambda$ for some $\lambda$ $\Rightarrow$ $T$ is not $\omega$-stable $\Rightarrow$  $f_T(\lambda)\geq\lambda+2^{\aleph_0}$ for some $\lambda$;
    \item $f_T(\lambda)>\lambda+2^{\aleph_0}$ for some $\lambda$ $\Rightarrow$ $T$ is not superstable $\Rightarrow$  $f_T(\lambda)\geq\lambda^{\aleph_0}$ for some $\lambda$;
    \item $f_T(\lambda)>\lambda^{\aleph_0}$ for some $\lambda$ $\Rightarrow$ $T$ has the order property (i.e., is not stable) $\Rightarrow$  $f_T(\lambda)\geq\mathrm{ded}(\lambda)$ for some $\lambda$;
    \item $f_T(\lambda)>\mathrm{ded}(\lambda)$ for some $\lambda$ $\Rightarrow$ $T$ is multiply ordered (see~\cite[Theorem 4]{keisler1976six} for the definition) $\Rightarrow$  $f_T(\lambda)\geq(\mathrm{ded}(\lambda))^{\aleph_0}$ for some $\lambda$;
    \item $f_T(\lambda)>(\mathrm{ded}(\lambda))^{\aleph_0}$ for some $\lambda$ $\Rightarrow$ $T$ has the independence property (i.e., $\mathrm{IP}$) $\Rightarrow$  $f_T(\lambda)\geq 2^\lambda$ for some $\lambda$.
\end{itemize}
Hence, from Lemma~\ref{lem:typecounting}, the properties of Definition~\ref{def:superstab}, not being multiply ordered, and $\mathrm{NIP}$, are preserved when going to the core companion  of a theory. We also note that superstability and not being multiply ordered can be equivalently defined through forbidding certain sets of formulas from witnessing a particular configuration (see~\cite{keisler1976six, shelah1971stability}). Whilst they are not patterned properties (because the omitted pattern may use more than one formula), they may be shown to be preserved when going to the core companion  by a proof analogous to that of Theorem~\ref{thm:inheritance} (or Lemma~\ref{lem:MR}, which also involves multiple formulas in the omitted configurations).
\end{remark}

\subsection{Monadic model theoretic properties}\label{sub:monadic}
In this subsection, we prove that the two main monadic model theoretic properties of monadic stability and monadic $\mathrm{NIP}$ are preserved by going to the core companion. Such properties are of interest in model theory~\cite{baldwin1985second}, permutation group theory (in the $\omega$-categorical context)~\cite{braunfeld2022monadic, bodor2024classification}, and structural graph theory (in their generalisation to hereditary classes)~\cite{dreier2024first}.

\begin{definition}\label{def:monadicNIP} Let $\mathcal{P}$ be a straight pattern. We say that a first-order theory $T$ is monadically $\mathrm{N}\mathcal{P}$ if all expansions of 
$T$ by unary predicates are still $\mathrm{N}\mathcal{P}$.  
\end{definition}

	The two main monadic properties are monadic stability and monadic $\mathrm{NIP}$. We note that other classes (such as monadically $\mathrm{NSOP}$ or monadically simple) just collapse to monadic stability~\cite{monadicallyNSOP}. 

The definition of monadic $\mathrm{NIP}$  \textit{prima facie} does not seem amenable to being shown to be preserved by going to the core companion  of a theory by the same methods we used for straightly definable properties. Indeed, unlike the other properties we discuss in Section~\ref{sec:preserved}, neither monadic $\mathrm{NIP}$ nor monadic stability are preserved by interpretations (or even bi-interpretations). However, we will make use of an equivalent characterisation of monadic $\mathrm{NIP}$ which has a similar flavour to definitions in terms of omitting a straight pattern, and which is thus amenable to our techniques. This is the characterisation of monadic $\mathrm{NIP}$ in terms of not admitting a pre-coding configuration~\cite{braunfeld2021characterizations}. The definition we give below is not the original (of~\cite[Definition 3.10]{braunfeld2021characterizations}), but a simplified variant which is proven to be equivalent to the original in Lemma 1.7 of~\cite{braunfeld2024corrigenda}:

\begin{definition} A \textit{pre-coding configuration} consists of a formula $\phi(\bar{x},\bar{y},z)$ 
with $|z|=1$, and a sequence $\mathcal{I}:=(\bar{d}_i)_{ i\in\mathbb{Z}}$ such that for all $s<0<t$ from $\mathbb{Z}$, there is $c_{(s,t)}$ such that
\begin{align}
   & \vDash \phi(\bar{d}_s,\bar{d}_t, c_{(s,t)});\\
    \text{ for all } v>t, & \vDash \neg\phi(\bar{d}_s, \bar{d}_v,c_{(s,t)}); \text{ and}\\
    \text{for all } u<s, & \vDash\neg\phi(\bar{d}_u, \bar{d}_t, c_{(s,t)})\;.
\end{align}
\end{definition}

\begin{fact}[Theorem 1.1 \& Proposition 2.11 in~\cite{braunfeld2021characterizations}]\label{fact:monnipequivalent} The following are equivalent for a first-order theory $T$:
\begin{enumerate}[(i)]
    \item $T$ is monadically $\mathrm{NIP}$;
    \item $T$ does not admit a pre-coding configuration.
\end{enumerate}
\end{fact}

\begin{fact}[{\cite{baldwin1985second}}]\label{fact:monadicallystable} Let $T$ be a first-order theory. The following are equivalent:
\begin{enumerate}[(i)]
    \item $T$ is monadically stable;
    \item $T$ is stable and monadically $\mathrm{NIP}$.
\end{enumerate}
\end{fact}

\begin{proposition}\label{nip_stab} Let $S$ be the core companion  of $T$. Then 
\begin{enumerate}[(i)]
    \item if $T$ is monadically $\mathrm{NIP}$, then so is $S$;
    \item if $T$ is monadically stable, then so is $S$. 
\end{enumerate}
\end{proposition}

\begin{proof} The second statement follows from the first by Fact~\ref{fact:monadicallystable} and the fact that stability is preserved by taking the core companion  (Theorem~\ref{thm:inheritance}). We prove the first statement by contrapositive. Suppose that $S$ is not monadically $\mathrm{NIP}$. Then, it admits a pre-coding configuration  given by a formula $\phi(\bar{x},\bar{y},z)$, a sequence $(\bar{d}_i)_{i\in\mathbb{Z}}$, and $(c_{(s,t)}\vert s<0<t\in\mathbb{Z})$. Since $S$ is a model-complete core theory, we may take $\phi$ to be an existential positive formula. We also take $\psi(\bar{x},y,z)$ to be an existential positive formula which is equivalent modulo $S$ to $\neg\phi(\bar{x},\bar{y},z)$ (and so is an obstruction of $\phi(\bar{x},\bar{y},z)$). Take a homomorphism $f \colon \fa^\star\to\fb^\star$ where $\fb^\star$ is a monster model for $T$. Since existential positive formulas are preserved by homomorphisms and the fact that $\psi$ is an obstruction of $\phi$ is also true in $T$, we still have that $(f(\bar{d}_i))_{i\in\mathbb{Z}}$ and $(f(c_{(s,t)})\vert s<0<t\in\mathbb{Z})$ witness that $\phi(\bar{x},\bar{y},z)$ has a pre-coding configuration. In particular, $T$ is not monadically $\mathrm{NIP}$ by Fact~\ref{fact:monnipequivalent}.
\end{proof}

	Proposition~\ref{nip_stab} can also be phrased in the following way in the $\omega$-categorical setting.
	
\begin{corollary}\label{nip_stab_omega}
	Let $\fa$ be a countable $\omega$-categorical structure, and let $\fb$ be the model-complete core of $\fa$. Then
\begin{enumerate}[(i)]
    \item\label{it:monnip} if $\fa$ is monadically $\mathrm{NIP}$, then so is $\fb$;
    \item\label{it:monstab} if $\fa$ is monadically stable, then so is $\fb$. 
\end{enumerate}
\end{corollary}

	We remark that the second claim of Corollary~\ref{nip_stab_omega} also follows from an orbit growth argument, as it already has been observed in~\cite{bodor2024classification}. This uses a characterization of $\omega$-categorical monadically stable structures in terms of unlabelled growth.
	
\begin{definition}~\label{def:lu}
	Let $\fa$ be an $\omega$-categorical structure. We denote by $u_n(\fa)$ the number of orbits of the natural action of $\aut(\fa)$ on $n$-element subsets of $\fa$. The sequence $(u_n(\fa))_{n\in \mathbb{N}}$ is called the \emph{unlabelled growth} of $\fa$.
\end{definition}

\begin{theorem}[\cite{braunfeld2022monadic}, Theorem 1.1]\label{hereditarily_cellular}
	Let $\fa$ be a countable $\omega$-categorical structure. Then $\fa$ is monadically stable if and only if $\fa$ is stable and for all $c>1$ we have $u_n(\fa)<c^n$ if $n$ is large enough.
\end{theorem}

	We know that if $\fa$ is $\omega$-categorical and $\fb$ is its model-complete core then $u_n(\fb)\leq u_n(\fa)$ for all $n\in \mathbb{N}$ (see for instance~\cite{bodor2022csp}, Lemma 2.2.69). Therefore, Theorem~\ref{hereditarily_cellular} indeed implies that $\omega$-categorical monadically stable structures are closed under taking model-complete cores. We mention that a different orbit growth characterization (using labelled growth) of $\omega$-categorical monadically stable structures is given in~\cite{bodor2025labelled} which also can be used to show item~\ref{it:monstab} of Corollary~\ref{nip_stab_omega}.
	
	For $\omega$-categorical monadically $\mathrm{NIP}$ structures no characterization is known in terms of orbit growth functions, however the following has been conjectured in~\cite{braunfeld2022monadic} (Conjecture 2).
	
\begin{conjecture}\label{conj:sam}
	Let $\fa$ be a countable $\omega$-categorical structure. Then $\fa$ is monadically $\mathrm{NIP}$ if and only if $u_n(\fa)<c^n$ for some $c\in \mathbb{R}$.
\end{conjecture}

	Note that given Conjecture~\ref{conj:sam} the same argument as above would also imply the first item of Corollary~\ref{nip_stab_omega}.

\subsection{Finer features in an $\omega$-stable setting}\label{sub:finer}
In later sections we will be studying various classes of structures which are $\omega$-stable (such as structures interpretable in $(\mathbb{Q}; =)$, or Lachlan's class). We have already shown that $\omega$-stability is preserved when going to the model-complete core. The finer structure of definable sets in $\omega$-stable theories can be studied through the notion of Morley rank. Indeed, all $\omega$-stable theories are totally transcendental (in the sense that Morley rank is defined for all formulas), and for countable theories $\omega$-stability and total transcendence are the same property~\cite[Theorem 5.2.6]{Tent-Ziegler}. In this subsection we prove that Morley rank behaves well under the operation of going to the core companion. We also prove that the core companion of a strongly minimal theory is still strongly minimal.

\begin{lemma}\label{lem:algebraicity} Let $S$ be the core companion  of $T$. Let $\phi(\bar{x};\bar{y})$ be an existential positive formula and $a$ be a tuple of arity $|\bar{y}|$ from the monster model $\fa^\star$ of $S$. Let $f \colon \fa^\star\to\fb^\star$ be a homomorphism from $\fa^\star$ to a monster model $\fb^\star$ of $T$. If $\phi(\bar{x};\bar{a})$ has infinitely many realisations in $\fa^\star$, so does $\phi(\bar{x};f(\bar{a}))$ in $\fb^\star$.
\end{lemma}
\begin{proof} As usual, we prove this by contrapositive. Let $(\phi(\bar{b}_i, \bar{a}))_{i\in\mathbb{N}}$ be distinct realisations of $\phi(\bar{x},\bar{a})$ in $\fa^\star$. Since $S$ is a model-complete core theory, $\bar{x}_1\neq \bar{x}_2$ (where $|\bar{x}_1|=|\bar{x}_2|=|\bar{x}|$) is equivalent to an existential positive formula $\psi(\bar{x}_1,\bar{x}_2)$. In particular, $\psi(\bar{b}_i, \bar{b}_j)$ holds for all $i\neq j$ in $\fa^\star$ and so we have that $\psi(f(\bar{b}_i), f(\bar{b}_j))$ holds in $\fb^\star$ for all $i\neq j$. Since this is still an obstruction of $\bar{x}_1=\bar{x}_2$ in $T$, $\phi(\bar{x},\bar{a})$ has infinitely many realisations in $\fb^\star$.
\end{proof}

\begin{definition}\label{def:MR} Let $T$ be a first-order theory. Let $\phi(\bar{x};\bar{y})$ be a $\tau$-formula. Let $b$ be a tuple of arity $|\bar{y}|$. We define the Morley rank of $\phi(\bar{x};\bar{b})$ by induction on ordinals as follows: 
\begin{itemize}
    \item $\RM(\phi(\bar{x};\bar{b}))\geq 0$ if $\phi$ is consistent;
    \item $\RM(\phi(\bar{x};\bar{b}))\geq{\alpha}+1$ if there are $(\phi_i(\bar{x}))_{i\in\mathbb{N}}$ which imply $\phi$ and which are pairwise inconsistent and such that $\RM(\phi_i)\geq\alpha$ for all $i\in\mathbb{N}$;
    \item $\RM(\phi(\bar{x};\bar{b}))\geq \lambda$ for $\lambda$ a limit ordinal if $\RM(\phi)\geq\gamma$ for all $\gamma<\lambda$.
\end{itemize}
    If $\phi$ is inconsistent, we define $\RM(\phi)=-\infty$. If $\RM(\phi)>\alpha$ for all ordinals $\alpha$, we set $\RM(\phi)=\infty$. Otherwise, there is some $\alpha$ such that $\RM(\phi)\geq\alpha$ but $\RM(\phi)\not\geq\alpha+1$ and we set $\RM(\phi)=\alpha$.  When working with multiple theories, we may write $\RM_T$ to denote the Morley rank in the theory $T$. We say that a theory $T$ is \emph{totally transcendental} if all consistent formulas (with parameters from the monster model) have ordinal Morley rank. 
\end{definition}

It is clear how to define the positive version of Morley rank:
\begin{definition} Let $S$ be a positive theory with the joint homomorphism property. Let $\phi(\bar{x};\bar{y})$ be an existential positive $\tau$-formula. Let $b$ be a tuple of arity $|\bar{y}|$ (from the positive monster model of $S$). We say
\begin{itemize}
    \item $\RM^+(\phi(\bar{x};\bar{b}))\geq 0$ if $\phi$ does not imply $\bot$;
    \item $\RM^+(\phi(\bar{x};\bar{b}))\geq\alpha+1$ if there are existential positive $(\phi_i(\bar{x}))_{i\in\mathbb{N}}$ which imply $\phi$ such that for each $i<j\in\mathbb{N}$, $\phi_j$ is an obstruction of $\phi_i$ and such that $\RM^+(\phi_i)\geq\alpha$ for all $i\in\mathbb{N}$;
    \item $\RM^+(\phi(\bar{x};\bar{b}))\geq \lambda$ for $\lambda$ a limit ordinal if $\RM^+(\phi)\geq\gamma$ for all $\gamma<\lambda$.
\end{itemize}
    Then, we define the positive Morley rank of a formula exactly as in Definition~\ref{def:MR}. 
\end{definition}

The lemma below could also be deduced from Lemma~\ref{lem:tracedef} and~\cite[Proposition 7.36]{Walsberglong}.

\begin{lemma}\label{lem:MR} Let $S$ be the core companion  of $T$. Let $\phi(\bar{x};\bar{y})$ be an existential positive formula and $a$ be a tuple of arity $|\bar{y}|$ from the monster model $\fa^\star$ of $S$. Let $f \colon \fa^\star\to\fb^\star$ be a homomorphism from $\fa^\star$ to a monster model $\fb^\star$ of $T$. Then, the Morley rank of $\phi(\bar{x};\bar{a})$ in $S$ is smaller or equal to the Morley rank of $\phi(\bar{x};f(\bar{a}))$ in $T$. In particular, if $T$ is totally transcendental, so is $S$.
\end{lemma}
\begin{proof} In $S$ every formula is equivalent to an existential positive formula. Thus, every formula is assigned both a Morley rank and a positive Morley rank and the two coincide.

We shall prove by induction on ordinals that if $\RM(\phi(\bar{x},\bar{a}))\geq\alpha$ in $S$, then $\RM(\phi(\bar{x},f(\bar{a})))\geq\alpha$ in $T$. Since $f$ is a homomorphism from $\fa^\star$ into $\fb^\star$, if $\phi(\bar{x},\bar{a})$ is consistent in $S$, so is $\phi(\bar{x},f(\bar{a}))$ in $T$. Hence, we have dealt with the base case. Now, since the case of limit ordinals is trivial, we only need to deal with the case of $\alpha+1$. Suppose that $\RM(\phi(\bar{x},\bar{a}))\geq\alpha+1$. Since in $S$, $\RM^+$ and $\RM$ agree, there are existential positive formulas $(\phi_i(\bar{x};\bar{a}_i))_{i\in\mathbb{N}}$ which imply $\phi(\bar{x};\bar{a})$ and which are pairwise inconsistent.  Note that the formulas $\phi_i'(\bar{x}; {\bar{a}\bar{a}_i}):=\phi_i(\bar{x};\bar{a}_i)\wedge \phi(\bar{x};\bar{a})$ also witness $\RM(\phi(\bar{x},\bar{a}))\geq\alpha+1$, and we still have that $\phi_i'(\bar{x}; f({\bar{a}\bar{a}_i}))$ imply $\phi(\bar{x}; f(\bar{a}))$ in $T$. Moreover, by Lemma~\ref{negpres} we know that the formulas $\phi_i'(\bar{x}; f({\bar{a}\bar{a}_i}))$  are also pairwise inconsistent in $T$, and thus,
by the inductive hypothesis, they witness the fact that $\RM(\phi(\bar{x};f(\bar{a})))\geq \alpha+1$, completing the proof. Finally, suppose by contrapositive $S$ is not totally transcendental.
Then, some existential positive formula $\phi(\bar{x};\bar{a})$ is assigned Morley rank $\infty$. But then, taking an appropriate homomorphism $f \colon \fa^\star\to\fb^\star$, we get that $\phi(\bar{x},f(\bar{a}))$ also has Morley rank $\infty$, meaning that $T$ is not totally transcendental either. 
\end{proof}

\begin{remark}\label{rem:stronglymin} Given a formula $\phi(\bar{x};\bar{b})$ with $\RM(\phi(\bar{x};\bar{b}))=\alpha$, where $\alpha$ is an ordinal, let the \emph{Morley degree} of $\phi(\bar{x};\bar{b})$, $\dM(\phi(\bar{x};\bar{b}))$, be the largest number $n\in\mathbb{N}$ such that there are $\phi_1(\bar{x}), \dots, \phi_n(\bar{x})$ implying $\phi$, pairwise inconsistent and with $\RM(\phi_i(\bar{x}))=\alpha$ (such number exists; cf.~\cite[Lemma 2.6.4]{Tent-Ziegler}). We say that a first-order theory $T$ is \emph{strongly minimal} if every definable subset (with parameters) of its monster model $\fa^\star$ is either finite or cofinite. In particular, a complete theory $T$ is strongly minimal if and only if it is a theory of a finite structure or $(\RM_T(x=x), \dM_T({x=x}))=(1,1)$. 
If a positive theory is not a model-complete core theory, there may not be a reasonable notion of positive Morley degree of an existential positive formula of ordinal positive Morley rank. Moreover, let $S$ be the core companion  of $T$ and let $\fa^\star$ and $\fb^\star$ be monster models of the respective theories with a homomorphism $f \colon \fa^\star\to\fb^\star$. Let $\phi(\bar{x};\bar{b})$  be an existential positive formula with $\bar{b}$ being parameters from $\fa^\star$. Then, it is possible both that $\dM(\phi(\bar{x};\bar{b}))<\dM(\phi(\bar{x};f(\bar{b})))$ and that $\dM(\phi(\bar{x};\bar{b}))>\dM(\phi(\bar{x};f(\bar{b})))$. 

For the first case, let $T$ be the theory of an equivalence relation $E_1$ with two classes {together with the disequality relation}. Its core companion  $S$ is the theory of an equivalence relation $E_1$ with a single class {together with the disequality relation}. In particular, $\dM_T(x=x)=2$ and $\dM_S(x=x)=1$. For the second case, let $T'$ be the expansion of $T$ by a relation $E_1^\neg$ for the negation of $E_1$ and by an equivalence relation $E_2$ which agrees with $E_1$ on one of its classes, and which refines the other class into infinitely many classes. The core companion  $S'$ of $T'$ is the theory where the equivalence relation $E_2$ coincides with $E_1$, which still partitions the domain in two classes and is still negated by $E_1^\neg$. Now, $(\RM_{S'}(x=x),\dM_{S'}(x=x))=(1,2)$. But $(\RM_{T'}(x=x),\dM_{T'}(x=x))=(2,1)$.
\end{remark}

	In spite of the fact that the Morley degree of a theory does not necessarily behave well when moving to the core companion, strong minimality is still preserved.

\begin{corollary} 
	Let $S$ be the core companion  of $T$.  If $T$ is strongly minimal, then so is $S$.
\end{corollary}
\begin{proof} By the previous lemma, 
\[\RM_S(x=x)\leq \RM_T(x=x)\leq 1\;.\]
If $\RM_S(x=x)=0$ then $S$ is finite, so we are done. Thus, we can assume that $\RM_T(x=x)={\RM_S(x=x)}=1$. If $S$ is not strongly minimal, this is witnessed by two existential positive formulas $\phi_1(x,\bar{a}_1)$ and $\phi_2(x,\bar{a}_2)$ which are jointly inconsistent and both of Morley rank $\geq 1$. By Lemma~\ref{lem:algebraicity}, since $\phi_1$ and $\phi_2$ are non-algebraic (i.e., have infinitely many realisations) in $\fa^\star$, $\phi_1(x,f(\bar{a}_1))$ and $\phi_2(x,f(\bar{a}_2))$ are non-algebraic in $\fb^\star$. Hence, they both have Morley rank $\geq 1$ and by Lemma~\ref{negpres} they are still jointly inconsistent, yielding that $T$ is not strongly minimal. 
\end{proof}

\subsection{Trace definability of the core companion}\label{sub:trace}
In later sections, we will see that the {classes} of structures interpretable in $(\mathbb{N}, =)$ or in $(\mathbb{Q}, <)$ are not closed under taking core companions. In particular, it is possible that no model of a theory $T$ interprets a model of the core companion  $S$ {of $T$}. Nevertheless, in this subsection, we show that the core companion $S$ of a theory $T$ is interpretable in a weak sense in $T$, through the notion of \emph{trace definability}~\cite{Walsbergshort, Walsberglong}. Trace definability is a weak notion of interpretation that recently received attention in model theory.
We show that if a complete theory $T$ has a {model-complete} core companion $S$, then $T$ trace defines $S$. Several model theoretic properties can be defined in terms of not trace defining certain theories. Since trace definability is a transitive relation, showing that a theory trace defines its core companion  implies that all properties which can be defined in terms of not trace defining a given theory (or class of theories) are preserved when moving to the core companion.

\begin{definition} Let $\fa$ and $\fb$ be first-order structures in signatures $\sigma$ and $\tau$ respectively. A \emph{trace definition} of $\fb$ in $\fa$ is an injective map $\mathfrak{t}:\fb\to\fa^m$ for some $m\in\mathbb{N}$ such that for any $\tau\cup\{B\}$-formula $\phi(\overline{x})$ with $|\overline{x}|=n$, there is a $\sigma\cup\{A\}$-formula $\psi(\overline{y})$ with $|\overline{y}|=nm$ such that for any tuple $\overline{a}\in B^n$, we have that
\[\fb\vDash \phi(\overline{a})\Leftrightarrow \fa\vDash\psi(\mathfrak{t}(\overline{a})).\]
We say that $\fa$ \emph{trace defines} $\fb$ through the map $\mathfrak{t}$. We say that a complete theory $T$ trace defines another complete theory $S$ if some model of $S$ is trace definable in some model of $T$. We say that two theories (or two structures) are \emph{trace equivalent} if they trace define each other.
\end{definition}

It is easy to see that trace definability is a transitive relation and that if $\fa$ interprets $\fb$, then it also trace defines $\fb$~\cite{Walsbergshort}. In particular, we can think of trace definability as a weak 
form 
of interpretation. When a theory $T$ trace defines another theory $S$, several properties are 
transferred 
from $T$ to $S$. We list below some of the main properties preserved under trace definability. We refer the reader to~\cite{Walsbergshort} for the various definitions. Several of these properties are used to analyse the finer structure of stable and $\mathrm{NIP}$ theories~\cite{shelah2014strongly, SimonRankOne, aschenbrenner2016vapnik, basit2021zarankiewicz}.

\begin{fact}[{\cite{Walsbergshort, walsberg2026tracedefinabilityipreservation,walsberg2026tracedefinabilityiimodeltheoretic}}]\label{fact:tracedefpreserved} Let $k\in\mathbb{N}_{>0}$ and $\mu$ be a cardinal. Suppose that $T$ trace defines $S$. If $T$ has one of the following properties, then $S$ also does: stability, $\mathrm{NIP}$, $\mathrm{NIP}_k$, $\mathrm{NFOP}_k$, being totally transcendental, superstability, strong dependence, having finite $U$-rank, having Morley rank $\leq\mu$, having $\mathrm{dp}$-rank $\leq\mu$, having finite $\mathrm{op}$-dimension, linear $\mathrm{VC}$-density bounds, having the strong {Erd\H{o}s}-Hajnal 
property, having near linear Zarankiewicz bounds. 
\end{fact}

\begin{lemma}\label{lem:tracedef} Let $T$ be a first-order theory and $S$ be its core companion. Let $\fb$ be a model of $S$ and $\mathfrak{t} \colon \fb\to\fa$ be a homomorphism from $\fb$ into some model $\fa$ of $T$. Then $\mathfrak{t}$ is a trace definition.
In particular, $T$ trace defines $S$.
\end{lemma}
\begin{proof} 
{Let $\tau$ be the signature of $S$ and $T$.}
Notice that since $S$ is a model-complete core theory, $\mathfrak{t}$ is injective. In fact, there is some existential positive formula $\psi(x,y)$ which is equivalent to 
$x\neq y$ in $S$. Then $\psi(x,y)$ is still an obstruction to $x=y$ in $T$, and since  $\mathfrak{t}$ preserves $\psi$, this implies that $\mathfrak{t}$ is injective. Now, consider an existential positive $\tau$-formula $\phi(\overline{x}, \overline{y})$. 
{Since $S$ is a model-complete core theory, 
$\neg\phi(\overline{x}, \overline{y})$ is equivalent to an
existential positive formula 
$\theta(\overline{x}, \overline{y})$ modulo $S$.}
Take $\overline{b}\in B^{|\overline{y}|}$. Notice that for $\overline{a}\in B^{|\overline{x}|}$,
\[\fb\vDash \phi(\overline{a}, \overline{b})\Rightarrow \fa\vDash\phi(\mathfrak{t}(\overline{a}), \mathfrak{t}(\overline{b})), \] 
and 
\[\fb\vDash\neg\phi(\overline{a}, \overline{b})\Leftrightarrow \fb\vDash\theta(\overline{a}, \overline{b})\Rightarrow\fa\vDash\theta(\mathfrak{t}(\overline{a}), \mathfrak{t}(\overline{b}))\Rightarrow \fa\vDash\neg\phi(\mathfrak{t}(\overline{a}), \mathfrak{t}(\overline{b})),\]
where the last implication holds since $\theta$ is still an obstruction to $\phi$ in $T$. In particular, from the above implications we obtain that $\fb$ is trace defined in $\fa$, as desired.
\end{proof}

Putting together Fact~\ref{fact:tracedefpreserved} and Lemma~\ref{lem:tracedef}, we get:

\begin{corollary}\label{cor:trace}Let $k\in\mathbb{N}_{>0}$ and $\mu$ be a cardinal. Suppose that $S$ is the core companion of the theory $T$. If $T$ has one of the following properties, then $S$ also does: stability, $\mathrm{NIP}$, $\mathrm{NIP}_k$, $\mathrm{NFOP}_k$, being totally transcendental, superstability, strong dependence, having finite $U$-rank, having Morley rank $\leq \mu$, having $\mathrm{dp}$-rank $\leq \mu$, having finite $\mathrm{op}$-dimension, having linear $\mathrm{VC}$-density bounds, having the strong {Erd\H{o}s}-Hajnal property, having near linear Zarankiewicz bounds. 
\end{corollary}

\begin{remark}
    Whilst trace definitions preserve most model-theoretic properties inside the $\mathrm{NIP}$ world (or the $\mathrm{NIP}_k$ world for $k>1$), they do not preserve most properties outside of this domain: any structure with $\mathrm{IP}_k$ (for example, the random $k$-hypergraph) trace defines all homogeneous $k$-ary structures. In particular, simplicity, $\mathrm{NSOP}$, and $\mathrm{NSOP}_k$ for $k\geq 1$ are not preserved by trace definitions.
\end{remark}

\section{Properties not preserved by taking model-complete cores}\label{sec:notpreserved}
The previous section might leave the impression that virtually every meaningful model-theoretic property
is preserved by taking model-complete cores. However, this is not the case; we start with an instructive example that shows that total categoricity is not preserved by taking model-complete cores (Theorem~\ref{thm:total_mc}). We do not know whether total categoricity is preserved by the model companion.

Other classes of structures that are not preserved by taking model-complete cores are the class of all structures interpretable in  equality (Section~\ref{sect:i_n}), and the class of all structures interpretable in  $({\mathbb Q};<)$ (Section~\ref{sect:inter_q}). The latter is not even preserved by taking model companions.

\subsection{Total categoricity}
A theory is called \emph{totally categorical} if it has up to isomorphism precisely one model of each infinite cardinality.
A structure is called \emph{totally categorical} if its first-order theory is. 
	It is well known that totally categorical structures are not closed under taking reducts. An example for this is the digraph $\fa$ with infinitely many disjoint directed edges, which is totally categorical. However, let us consider the first-order reduct $\fa$ of this structure which only has two unary predicates, one for all sources and one for all targets. The theory of this structure has 3 non-isomorphic models of size $\omega_1$, because we can pick the size of the unary relations to be $\omega$ or $\omega_1$ as long as at least one of the relations has size $\omega_1$. This shows that $\fa$ is not totally categorical. 

    We now construct a totally categorical structure whose model-complete core is not
    totally categorical.
    
\begin{theorem}\label{thm:total_mc}
    The class of totally categorical structures is not closed under taking core companions.
\end{theorem}

\begin{proof}
    Let $A := {\mathbb N} \times \{0,1,2\}$. Let $\fa$ be the structure with domain $A$ and whose relations are defined as follows (also see Figure~\ref{fig:spider}). 
\begin{align*}
U_i & \coloneqq {\mathbb N} \times \{i\} \text{ for } i \in \{0,1,{2}\} \\
N & \coloneqq \{(u_1,u_2) \in  (U_1)^2\cup (U_2)^2 \mid u_1 \neq u_2\} \\
R & \coloneqq \{((n,0),(n,1)): n\in {\mathbb N}\}\cup \{((n,0),(n,2)): n\in {\mathbb N}\}\\
& \quad \cup \{((0,0),(n,1)): n\in {\mathbb N}\} \cup \{((0,0),(n,2)): n\in {\mathbb N}\}. 
\end{align*}
    
    Let $\fb$ be any model of the theory of $\fa$. Then, in $\fb$, $(U_0,U_1,U_2)$ is a partition of $B$ and $R$ is a union of a bijection from $U_0$ to $U_1$, a bijection from $U_0$ to $U_2$ and $\{v\}\times (U_1\cup U_2)$ for some unique vertex $v\in U_0$. In particular $|U_0|=|U_1|=|U_2|$. It is easy to see that if $|B|=\kappa$ then $\fb$ is isomorphic to the structure whose domain is $\kappa\times \{0,1,2\}$ and whose relations are defined the same way as for $\fa$. In particular $|B|$ determines $\fb$ up to isomorphism, i.e., $\Th(\fb) = \Th(\fa)$ is totally categorical.

    Now, let $C={\mathbb N}\times \{1,2\} \cup \{(0,0)\}$ and $\fc=\fa|_C$. Any endomorphism of $\fc$ is a union of $\id(\{(0,0)\})$ and self-injections of $U_1$ and $U_2$ which clearly can be approximated by automorphisms of $\fc$. Thus, $\fc$ is a model-complete core. Let us now observe that the map $A\rightarrow C$ which fixes every element in $U_1\cup U_2$ and maps every element of $U_0$ to $\{(0,0)\}$ is a homomorphism from $\fa$ to $\fc$. Therefore, $\fc$ is the model-complete core of $\fa$. Note, however, that $\fc$ is  interdefinable with the disjoint union of two infinite sets and a one-element structure. Then the same argument as above shows that $\Th(\fc)$ has 3 models of size $\omega_1$, and thus it is not totally categorical.
\end{proof}
\begin{figure}
\begin{center}
\includegraphics[scale=.6]{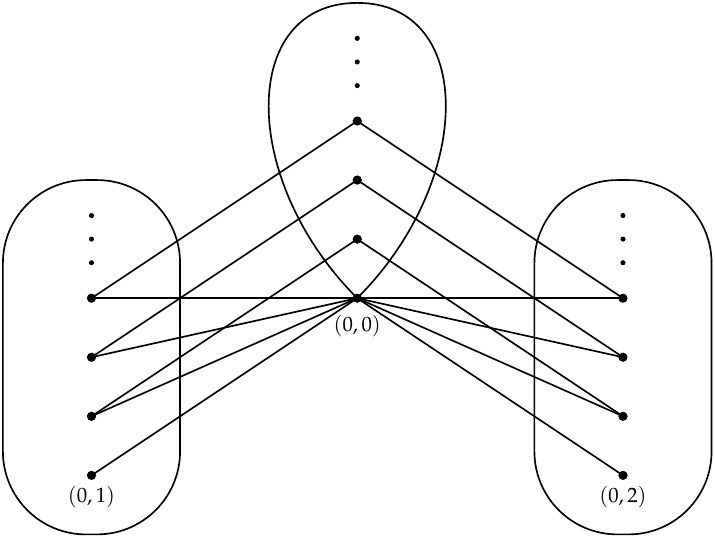}
\end{center}
\caption{An illustration of the structure from the proof of Theorem~\ref{thm:total_mc}.}
\label{fig:spider}
\end{figure}

\subsection{Structures interpretable in  equality}
\label{sect:i_n}
	In this section, we show that the class $\I(\atoms)$ is not closed under taking model-complete cores.
    We write ${\mathbb Z}_n$ for ${\mathbb Z}/n{\mathbb Z}$ and its elements will be identified with $\{0,\dots,n-1\}$ in the usual way.
    

\begin{definition}\label{def:counterexpl}
	Let $\fx$ denote the following structure. The domain of $\fx$ is $$X :=\{(a,b,m): a,b\in {\mathbb Q}, m\in \mZ_4, a\neq b\}.$$
    The relations of $\fx$ are
\begin{itemize}
\item $R:= \big \{((a,b,m),(a,b,m+1)) : a,b \in {\mathbb Q}, m \in \mZ_4 \big \}$,
\item $E \subseteq X^2$, the binary relation given by
\begin{align*}
& \{((a,b,2m),(a,c,2n)): n,m \in \mZ_4, b\neq c\} \\
\cup \;& \{((a,b,2m),(c,a,2n+1)): n,m \in \mZ_4, b\neq c\}  \\
\cup \; & \{((b,a,2m+1),(a,c,2n)): n,m \in \mZ_4 , b\neq c\} \\ 
 \cup \; & \{((b,a,2m+1),(c,a,2n+1)): n,m \in \mZ_4, b\neq c \}.
\end{align*}
\item $N \subseteq X^2$, the binary relation given by
$$N :=\{((a,b,m),(c,d,n)) : n,m \in \mZ_4, \{a,b\}\cap \{c,d\}=\emptyset\}.$$
\end{itemize}
\end{definition}

\begin{figure}
    \begin{center}
    \includegraphics[scale=.8]{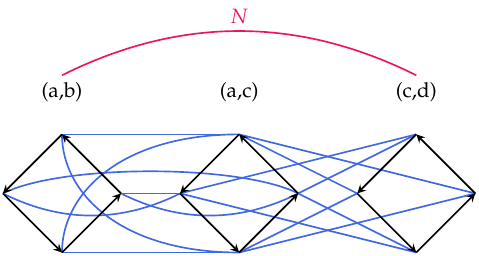}
    \end{center}
    \caption{An illustration of a finite substructure of the structure $\fx$ from Definition~\ref{def:counterexpl}.}
    \label{fig:Y}
\end{figure}

See Figure~\ref{fig:Y}.
For $m,n \in {\mathbb N}$, we write $m \equiv_k n$ if $m = n \mod k$.
We use the notation $(a,b)_0 :=a$ and $(a,b)_1:=b$ and consider the indices modulo 2. 
Then one can check that in $\fx$, for all $n,m \in \mZ_4$
$$((a,b,m),(c,d,n))\in E \text{ if and only if }
(a,b)_m \equiv_2 (c,d)_n \text{ and } |\{a,b\} \cap \{c,d\}|=1.$$

\begin{lemma}\label{interpretable}
	The structure $\fx$ is interpretable in $\atoms$.
\end{lemma}

\begin{proof}
	By Corollary~\ref{inter_finite} we know that $\atoms$ interprets the structure $(\{0,1,2,3\};0,1,2,3)$. Let $\iota$ be an interpretation witnessing this. Then we define
\begin{itemize}
\item $X':={\{(a,b)\in\mathbb{Q}^2}: a\neq b\}\times \dom(\iota)$,
\item $R':=\{((a,b,\bar{m}),(c,d,\bar{n}))\in (X')^2: a=c \wedge b=d \wedge \iota(\bar{n}) \equiv_4 \iota(\bar{m})+1\}$,
\item $E' \subseteq (X')^2$, the relation containing all tuples $(a,b,\bar{m},c,d,\bar{n})$ that satisfy 
\begin{align*}
&(a=c \wedge b\neq d \wedge \iota(\bar{m})\equiv_2 0\wedge \iota(\bar{n})\equiv_2 0)  \\
\vee \; &(a=d\wedge b\neq c \wedge \iota(\bar{m}) \equiv_2 0 \wedge \iota(\bar{n}) \equiv_2 1) \\
\vee \; &(b=c\wedge a\neq d \wedge \iota(\bar{m}) \equiv_2 1 \wedge \iota(\bar{n})\equiv_2 0) \\
\vee \; & (b=d\wedge a\neq c \wedge \iota(\bar{m}) \equiv_2 1 \wedge \iota(\bar{n}) \equiv_2 1)\},
\end{align*}
\item $N':=\{((a,b,\bar{m}),(c,d,\bar{n}))\in ({X}')^2: \neg (a=c\vee a=d\vee b=c \vee b=d)\}$.
\end{itemize}

	Let us define $I \colon (a,b,\bar{m})\mapsto (a,b,\iota(\bar{m}))$. Then we have $I^{-1}(R)=R'$, $I^{-1}(E)=E'$, $I^{-1}(N)=N'$, and these relations are clearly definable in $\atoms$. Finally, we have $\dom(I)=X'$ and $$I^{-1}(=)=\{((a,b,{\bar{m}}),(c,d,{\bar{n}})): a=c\wedge b=d \wedge \iota(\bar{m})=\iota(\bar{n})\}$$ which are also definable in $\atoms$. Therefore, the map $I$ gives rise to an interpretation of $\fx$ in  $\atoms$.
\end{proof}

	From the lemma above, it follows that the structure $\fx$ from Definition~\ref{def:counterexpl} is $\omega$-categorical, and hence has a model-complete core. We claim that the model-complete core of $\fx$ does not have a first-order interpretation in  $\atoms$. In fact, we will show (in Corollary~\ref{mc_core}) that the model-complete core of $\fx$ is isomorphic to the substructure $\fy$ of $\fx$ induced on the set 
\begin{align}
    Y & :=\{(a,b,m): a,b\in \mathbb{Q}, m\in \mZ_4, a< b\}. & \label{eq:B}
\end{align}

\begin{remark}\label{z_inter_q}
	Note that the same argument as in Lemma~\ref{interpretable} shows that $\fy$ is interpretable in $\order$.
\end{remark}

\begin{lemma}\label{hom_equiv}
	The structures $\fx$ and $\fy$ are homomorphically equivalent.
\end{lemma}

\begin{proof}
	Since $\fy$ is a substructure of $\fx$, it is enough to show that there exists a homomorphism from $\fx$ to $\fy$. We define this homomorphism as follows: 
\begin{itemize}
\item if $a<b$, then $h((a,b,m)):=(a,b,m)$; 
\item if $a>b$, then $h((a,b,m)):=(b,a,m+1)$.
\end{itemize}
	It is clear that the image of $h$ is {$Y$}. Now we show that $h$ preserves the relations $R,E,N$.

	Let us assume that $((a,b,m),(c,d,n))\in R$. Then $a=c$, $b=d$, and ${n=m+1} \mod 4$. In this case, if $a<b$, then $h$ fixes $(a,b,m)$ and $(c,d,n)=(a,b,m+1)$, and if $b>a$, then 
\begin{align*}
(h((a,b,m)),h((c,d,n))) & =(h((a,b,m)),h((a,b,m+1))) \\
& = ((b,a,m+1),(b,a,m+2))\in R.
\end{align*}
	By definition, $((a,b,m),(c,d,n))\in N$ iff $\{a,b\}\cap \{c,d\}=\emptyset$, so it is clear from the definition of $h$ that $h$ preserves $N$.

	Now let us assume that $((a,b,m),(c,d,n))\in E$. Then $(a,b)_m=(c,d)_n\eqqcolon u$. If $a<b$ then $(a,b)_m=u$, and if $a>b$ then $(b,a)_{m+1}=(a,b)_m=u$. Similarly if $c<d$ then $(c,d)_n=u$, and if ${c>d}$ then $(d,c)_{n+1}=(c,d)_n=u$. Therefore if $h((a,b,m))=(a',b',m')$ and $h((c,d,n))=(c',d',n')$ then we have $u=(a',b')_{m'}=(c',d')_{n'}$ and thus $$(h((a,b,m)),h((c,d,n)))=((a',b',m'),(c',d',n'))\in E.$$ Therefore, $h$ preserves the relation $E$.
\end{proof}

\subsubsection{The Johnson graph}
To prove that $\fy$ (see~\eqref{eq:B}) is a model-complete core and not interpretable in $\atoms$, we need to revisit basic facts about the so-called Johnson graph (and its finite covers).

\begin{definition}\label{def:Johnson}
	The \emph{Johnson graph} (denoted by $\john$) is the graph whose vertex set is $\upairs$ (that is, the set of $2$-element subsets of ${\mathbb Q}$), and where two (distinct) vertices $x,y\in \upairs$ are adjacent, i.e., connected by an edge, if and only if $x\cap y$ is nonempty.
	We use the notation $E$ for the set of edges, and $N$ for the set of nonedges of $\john$, i.e., 
    \begin{align*}
        E & :=\{(x,y\in \upairs^2: |x\cap y|=1\} \\
        \text{ and } \quad N & :=\{(x,y\in \upairs^2: |x\cap y|=0\}.
        \end{align*}
	If $\alpha$ is a permutation of ${\mathbb Q}$, then the natural action of $\alpha$ on $\upairs$ is denoted  by $\tilde{\alpha}$. 
\end{definition}

	We chose the domain set of $\john$ to be $\upairs$, but clearly we could use any countable set instead of $\mathbb Q$. 

\begin{remark}\label{rem:Jaut}
	It is well known and easy to see that 
    $\aut(\john)=\{\tilde{\alpha}: \alpha\in \sym(\mathbb {Q})\}$ (see, e.g.,~\cite{bodirsky2025structures}). 
\end{remark}

\begin{notation}
	We denote by $P_a$ the set $\big \{\{a,b\}: b\in {\mathbb{Q}} \setminus \{a\} \big \}$.
\end{notation}

\begin{lemma}\label{stars}
	Let $a\in {\mathbb Q}$, and let $e\in \eend(\john)$. Then there exists a unique $b\in {\mathbb Q}$ such that $e(P_a)\subseteq P_b$.
\end{lemma}

\begin{proof}
	Let $c_1,c_2,c_3,c_4\in {\mathbb Q}\setminus \{a\}$ be distinct. Then the elements $e(\{a,c_1\})$, $e(\{a,c_2\})$, $e(\{a,c_3\})$, and $e(\{a,c_4\})$ are pairwise adjacent in $\john$. It is easy to verify that this is only possible if there exists a number $b$ which is contained in $e(\{a,c_i\})$ for all $i \in \{1,2,3,4\}$, that is, $e(\{a,c_i\})\in P_b$. If we choose any $c_4'\in {\mathbb Q}
    \setminus \{a,c_1,c_2,c_3\}$, then the same argument shows that $e(\{a,c_4'\})\in P_b$.
    This implies that $e$ maps $P_a$ into $P_b$. The uniqueness of $b$ is {then} obvious.
\end{proof}

	The Johnson graph is not a core, because it is homomorphically equivalent to an infinite clique. However, it becomes a model-complete core if we expand it by the relation $N$.

\begin{lemma}\label{core1}
	The structure $\big (\upairs;E,N \big)$ is a model-complete core.
\end{lemma}

\begin{proof}
	We show that $\eend(\upairs;E,N)=\overline{\aut(\upairs;E,N)}=\overline{\aut(\john)}$. Let $e$ be an endomorphism of $(\upairs;E,N)$. Then $e$ is also an endomorphism of $\john$. By Lemma \ref{stars}, 
    there exists a map $\alpha \colon {\mathbb Q} \rightarrow {\mathbb Q}$ such that $P_a$ is mapped to $P_{\alpha(a)}$. We claim that $\alpha$ is injective. Arbitrarily choose distinct $a,b,c,d\in {\mathbb Q}$. 
    Then, {$(\{a,c\},\{b,d\})\in N$, and thus $(e(\{a,c\}),e(\{b,d\}))\in N$,}  that is, $e(\{a,c\})\cap e(\{b,d\})=\emptyset$. By definition, $e(\{a,c\})\in P_{\alpha(a)}$ and $e(\{b,d\})\in P_{\alpha(b)}$, that is, $\alpha(a)\in e(\{a,c\})$ and $\alpha(b)\in e(\{b,d\})$. It follows that $\alpha(a)\neq \alpha(b)$, which shows that $\alpha$ is injective. 
    
    Let $F$ be a finite subset of $\upairs$ and let $F':=\cup F$. Then $F'$ is a finite subset of ${\mathbb Q}$. Since $\alpha$ is injective, it follows that there exists a permutation $\beta$ of ${\mathbb Q}$ so that $\beta|_{F'}=\alpha|_{F'}$. Then $\tilde{\beta}\in \aut(\john)$. We claim that $\tilde{\beta}|_{F}=e|_{F}$. Let $\{a,b\}\in F$. Then $a,b\in F'$, and by definition $\alpha(a),\alpha(b)\in e(\{a,{b}\})$. Since $\alpha$ is injective,  
    $$e(\{a,b\})=\{\alpha(a),\alpha(b)\}=\{\beta(a),\beta(b)\}=\tilde{\beta}(\{a,b\}).$$ This proves that $\eend(\upairs;E,N)=\overline{\aut(\john)}$. Therefore, $(\upairs;E,N)$ is a model-complete core.
\end{proof}

\subsubsection{Finite covers of the Johnson graph}

Let $\fa$ and $\fb$ be structures with domains $A$ and $B$
and let $\pi \colon A \to B$ be a function. 
The \emph{kernel} of $\pi$ is the equivalence relation $\{(a_1,a_2) \mid \pi(a_1) = \pi(a_2)\} \subseteq A^2$. 
If $\pi$ is surjective and the kernel of $\pi$ is preserved by $\en(\fa)$, we write $\mu_\pi$ for the mapping from 
$\en(\fa)$ to $B \to B$ defined as follows.
If $e \in \en(\fa)$ and $b \in B$, we choose $a \in \pi^{-1}(b)$ (which we can since $\pi$ is surjective). Then 
$\mu_\pi(e)(b) := \pi(e(a))$ is well-defined since the kernel of $\pi$ is preserved by $e$. 
\begin{definition}\label{def:cover}
    Let $\fa$ and $\fb$ be structures with domains $A$ and $B$.
    A \emph{finite covering map $\pi \colon \fa \to \fb$} is a 
    surjective function $\pi \colon A \to B$ such that 
\begin{itemize}
    \item the preimage of each $w \in B$ under $\pi$ is finite,
    \item  the kernel of $\pi$ is preserved by $\en(\fa)$, and
    \item the image of $\aut(\fa)$ under $\mu_\pi$ equals $\aut(\fb)$. 
\end{itemize}
\end{definition}

\begin{remark}
    Note that finite covering maps $\pi \colon \fa \to \fb$ are usually defined using a variant of $\mu_\pi$, namely the version where it is only defined on $\aut(\fa)$ (see, for example,~\cite{EvansIvanovMacpherson,bodirsky2021permutation}).
    For us it will be practical to apply $\mu_\pi$ also to endomorphisms. Note that the assumption in Definition~\ref{def:cover} that the kernel of $\pi$ should be preserved by $\en(\fa)$ rather than $\aut(\fa)$ is not a severe restriction, 
because if we replace $\fa$ and $\fb$ with their expansions by all first-order definable relations, and the kernel of $\pi$ being preserved by $\aut(\fa)$, then it is also preserved by $\en(\fa)$.
\end{remark}

\begin{lemma}\label{fin_cover}
    Let $\pi \colon Y \rightarrow \upairs$ be given by  $(a,b,m)\mapsto \{a,b\}$. Then $\pi \colon \fy \to \john$ is a finite covering map. 
    Moreover, 
    every $e \in \en(\fy)$ preserves the relation 
    $D:=\{((a,b,m),(a,b,n)) \mid a,b \in {\mathbb Q}, a<b, m,n \in \mZ_4\}$
    and $\mu_{\pi}(e)$ is an endomorphism of $(\upairs;E,N)$.
\end{lemma}

\begin{proof}
    Clearly, the preimage of every element $\{a,b\} \in \upairs$ under $\pi$ is finite. We first show that the kernel of $\pi$ is preserved by $\en(\fy)$ by proving that it is existentially positively definable in $\fy$: 
    we have $\pi(x) = \pi(y)$ if and only if 
    $$x=y \vee R(x,y) \vee R(y,x) \vee \exists z \big (R(x,z) \wedge R(z,y) \big ).$$
    Next, we show that $\mu_{\pi}(e)$ preserves the relations $E$ and $N$ for every $e\in \en(\fy)$.

	Let $(\{a,b\},\{c,d\})\in E$, that is, $|\{a,b\}\cap \{c,d\}|=1$. We may assume without loss of generality that $a<b$ and $c<d$. By definition, $((a,b,m),(c,d,n))\in E$ for some $m,n\in \{0,1\}$. Then $(e((a,b,m)),e((c,d,n)))\in E$, which implies that $$|\mu_{\pi}(e)(\{a,b\})\cap \mu_{\pi}(e)(\{c,d\})| = |\pi(e(a,b,m)) \cap \pi(e(c,d,n))| =1,$$
    that is,  $(\mu_{\pi}(e)(\{a,b\}),\mu_{\pi}(e)(\{c,d\}))\in E$.

	Let $(\{a,b\},\{c,d\})\in N$, that is, $\{a,b\}\cap \{c,d\}=\emptyset$. We may again assume that $a<b$ and $c<d$. By definition, $((a,b,0),(c,d,0))\in N$, and thus $(e((a,b,0)),e((c,d,0)))\in N$. This implies that $\mu_{\pi}(e)(\{a,b\})\cap \mu_{\pi}(e)(\{c,d\})=\emptyset$, that is, $(\mu_{\pi}(e)(\{a,b\}),\mu_{\pi}(e)(\{c,d\}))\in N$.
    
	Now we show that $\pi$ is a finite covering map. A map $\gamma \colon Y \to Y$ is an automorphism of $\fy$ if and only if $\gamma$ is bijective and both $\gamma$ and $\gamma^{-1}$ are 
    endomorphisms of $\fy$. In this case, $\mu_{\pi}(\gamma)$ is also bijective, and it follows from the second part of the lemma, which we have already shown above, that both $\mu_{\pi}(\gamma)$ and $(\mu_{\pi}(\gamma))^{-1}=\mu_{\pi}(\gamma^{-1})$ are endomorphisms of $(\upairs,E,N)$. Therefore,  $\mu_{\pi}(\gamma)\in \aut(\upairs,E,N)=\aut(\john)$. 
    
    Conversely, if $\gamma'$ is an automorphism of $\john$ then it is of the form $\gamma'=\tilde{\alpha}$ for some $\alpha\in \sym({\mathbb Q})$ (Remark~\ref{rem:Jaut}). Consider the function $\gamma$ which maps $(a,b,m)$ to $(\alpha(a),\alpha(b),m)$ if $\alpha(a)<\alpha(b)$ and to $(\alpha(b),\alpha(a),m+1)$ if $\alpha(b)<\alpha(a)$. It is clear that $\gamma$ preserves $D$, and its action on $\upairs$ is $\gamma'$. We claim that $\gamma$ is an automorphism of $\fy$. The only nontrivial thing we have to check is that $\gamma$ preserves the relation $E$. Let $((a,b,m),(c,d,n))\in E$. Then $(a,b)_m=(c,d)_n$. Let $u:=\alpha((a,b)_m)=((c,d)_n)$. If $\alpha(a)<\alpha(b)$ then $u=(\alpha(a),\alpha(b))_m$, and if $\alpha(b)<\alpha(a)$ then $u=(\alpha(b),\alpha(a))_{m+1}$. Similarly, if $\alpha(c)<\alpha(d)$ then $u=(\alpha(c),\alpha(d))_n$, and if $\alpha(d)<\alpha(c)$ then $u=(\alpha(d),\alpha(c))_{n+1}$. By checking all four cases one can verify that $(\gamma(a,b,m),\gamma(c,d,n))\in E$. This implies that the action of $\aut(\fy)$ on $\upairs$ is exactly the automorphism group of $\john$, and concludes the proof that $\pi$ is a finite covering map.
\end{proof}

\subsubsection{Verifying that $\fy$ is a model-complete core}
    For $S\subseteq \upairs$, let  $\rotate_S$ be the set of permutation of $Y$ which maps $(a,b,m)$ to $(a,b,m+1)$ if $\{a,b\}\in S$ and $a<b$, and fixes every other element of $Y$. It is easy to see that $\mu_{\pi}(\rotate_S)=\id(\upairs)$ for all $S\subseteq \upairs$. Let $K$ be the group generated by the permutations $\rotate_S$. Then every element of $K$ can be written uniquely as a composition $\rotate_{S_1}\rotate_{S_2}^2\rotate_{S_3}^3$ for some disjoint subsets $S_1,S_2,S_3$ of $\upairs$. Let $\iota \colon K \to ({\mathbb Z}_4)^{\omega}$ be defined as follows. Suppose that $g\in K$ is of the form $g=\,\rotate_{S_1}\rotate_{S_2}^2\rotate_{S_3}^3$
    such that $S_1,S_2,S_3 \subseteq \upairs$ are disjoint. 
    Let $\iota(g) \colon {{\mathbb Q}\choose 2} \to \{0,1,2,3\}$ be the function which maps every pair $\{a,b\}$ to $i$ if $\{a,b\}\in S_i$ for $i \in \mZ_4$, and to $0$ otherwise. Then $\iota$ is well defined and a group isomorphism from $K$ to $({\mathbb Z}_4)^{\omega}$. In particular, $K$ is abelian. We define $K_0:=\{\rotate_S^2:S\subseteq \upairs\}$. Then $\iota(K_0)$ is isomorphic to $({\mathbb Z}_2)^{\omega}$. In particular, $K_0$ is a subgroup of $K$.

\begin{lemma}\label{local_flip}
	Let $H$ be a subset of ${\mathbb Q}$ of size at least 3 and let $e\in \eend(\fy)$ be such that $\mu_{\pi}(e)$ fixes every element in ${H\choose 2}$. Then there exists $S\subseteq \upairs$ such that $$e|_{\pi^{-1}({H\choose 2})}=(\rotate_S^2)|_{\pi^{-1}({H\choose 2})}.$$
\end{lemma} 

\begin{proof}
	Let $\{a,b\}\in {H\choose 2}$ be such that $a<b$. Since $\mu_{\pi}(e)$ preserves $\{a,b\}$, it follows that $e$ maps $(a,b,0)$ to $(a,b,m)$ for some $m \in \mZ_4$. Since $e$ preserves the relation $R$, it follows that $e((a,b,n))=(a,b,m+n)$ for all $n \in \mZ_4$. 
	
	
	Let $S_i$, for $i \in \{1,2,3\}$, be the set of those subsets $\{a,b\}$ {of $H$} with $a<b$ for which $(a,b,0) \in Y$ is mapped by $e$ to $(a,b,i) \in Y$. Let $g:=\,\rotate_{S_1}\rotate_{S_2}^2\rotate_{S_3}^3$. Then the above argument shows that $e$ and $g$ have the same action on ${\pi^{-1}({H\choose 2})}$. 

	We show that in fact $S_1=S_3=\emptyset$. Let us assume on the contrary that $S_1$ is nonempty. Let $\{a,b\}\in S_1, a<b$. Since $|H|\geq 3$ there exists some $c\in H$ which is different from $a$ and $b$. We distinguish two cases.
    \begin{itemize}
        \item $a<c$. In this case, $((a,b,0),(a,c,0))\in E$. We have $$(e((a,b,0)),e((a,c,0)))=((a,b,1),(a,c,m)) \in E$$ for some $m\in \mZ_4$. Then $b=(a,c)_m$ (the index is considered modulo $2$), and in particular $b\in \{a,c\}$, which is a contradiction.
        \item $c<a$. In this case, $((a,b,0),(c,a,1))\in E$. We have $$(e((a,b,0)),e((c,a,1)))=((a,b,1),(c,a,m)) \in E$$ for some $m\in \mZ_4$. Again, $b\in \{a,c\}$, a contradiction.
    \end{itemize}
    If $S_3$ is nonempty, we analogously reach a contradiction. 
	Therefore, $S_1=S_3=\emptyset$. By choosing $S=S_2$ we obtain that $e$ and $\rotate_S^2$ agree on $\pi^{-1}({H\choose 2})$.
\end{proof}

\begin{corollary}\label{cor:kernel}
	$\mu_{\pi}^{-1}(\{\id_{\upairs}\}) = K_0 \leq \aut({\fy})$.
\end{corollary}

\begin{proof}
	If $g\in K_0$ then it is clear that $\mu_{\pi}(g) =\id_{\upairs}$. 
	Conversely, if $\mu_\pi(g) = \id_{\upairs}$ then by applying Lemma \ref{local_flip} to $H=\mathbb Q$ we obtain that $g\in K_0$.
\end{proof}

\begin{lemma}\label{core2}
	The structure $\fy$ is a model-complete core.
\end{lemma}

\begin{proof}
    Let $e\in \eend(\fy)$, and let $F$ be a finite subset of $Y$. Then we have to show that there exists $\alpha \in \Aut(\fy)$ such that $e$ and $\alpha$ agree on $F$. We can choose a finite subset $H$ of ${\mathbb Q}$ of size at least 3 such that $F\subseteq \pi^{-1}({H\choose 2})$. By Lemma~\ref{fin_cover} it follows that $\mu_{\pi}(e)$ is an endomorphism of $(\upairs;E,N)$. By Lemma \ref{core1}, the structure $(\upairs;E,N)$ is a model-complete core. Therefore, there exists an automorphism $\gamma'$ of $\john$ such that $\gamma'$ and $\mu_{\pi}(e)$ agree on ${H\choose 2}$. Since $\pi$ is a finite covering map, it follows that there exists an automorphism $\gamma$ of $\fy$ such that $\mu_{\pi}(\gamma)=\gamma'$. Then ${\gamma^{-1}e}$ is an endomorphism of $\fy$ so that $\mu_{\pi}({\gamma^{-1}e})$ fixes every element in ${H\choose 2}$. Then by Lemma \ref{local_flip} there exists a permutation $h\in K_0$ such that ${\gamma^{-1}e}$ and $h$ agree on ${\pi^{-1}}({H\choose 2})\supseteq F$. Then $e$ and ${\gamma h}\in \aut(\fy)$ agree on $F$.
\end{proof}

\begin{corollary}\label{mc_core}
	$\fy$ is the model-complete core of $\fx$.
\end{corollary}

\begin{proof}
Follows from Lemma~\ref{hom_equiv} and Lemma~\ref{core2}.
\end{proof}

\subsubsection{Verifying that $\fy$ does not have an interpretation in  equality}

An \emph{involution} (on a set $B$) is a permutation $\alpha$ of $B$ such that $\alpha^2 = \id_B$.

\begin{lemma}\label{invol}
	The set of involutions of $\aut(\fy)$ is $K_0$. In particular, they form an abelian subgroup of $\aut(\fy)$.
\end{lemma}

\begin{proof}
	It is clear that every element of $K_0$ is an involution. Now let $\gamma$ be an involution of $\fy$. Then either $\mu_{\pi}(\gamma)=\id(\upairs)$ or $\mu_{\pi}(\gamma)$ is {of order 2}. In the former case we know that $\gamma\in K_0$ by Corollary \ref{cor:kernel}. So let us assume that {the latter holds}. We know that the map $\alpha\mapsto \tilde{\alpha}$ defines a group homomorphism between $\sym({\mathbb Q})$ and $\aut(\john)$ (Remark~\ref{rem:Jaut}). This means that there exists some {permutation} $\alpha\in \sym({\mathbb Q})$ of order 2 such that $\tilde{\alpha}=\mu_{\pi}(\gamma)$. {Then we can pick some $a,b\in \mathbb{Q}$ such that $\alpha(a)=b$ and $\alpha(b)=a$. We can assume without loss of generality that $b>a$. Let $c>b$ be arbitrary.} Then $((a,b,0),(a,c,0))\in E$. {Since the permutation $\mu_{\pi}(\gamma)=\tilde{\alpha}$ preserves the set $\{a,b\}$ it follows that}, $\gamma((a,b,0))=(a,b,m)$ for some $m \in {\mathbb Z}_4$. {On the other hand}, $\gamma((a,c,0))=(b,\alpha(c),n)$ or $\gamma((a,c,0))=(\alpha(c),b,n)$ for some $n$ depending on whether $\alpha(c)$ is smaller or bigger than $b$. Since $\gamma\in \aut(\fy)$ it follows that {$((a,b,m),(b,\alpha(c),n))\in E$ or $((a,b,m),(\alpha(c),b,n))\in E$. In either case, 
    $m$ must be odd.} If $\gamma((a,b,0))=(a,b,1)$ then since $\gamma$ preserves the relation $R$ it follows that $\gamma((a,b,1))=(a,b,2)$. Therefore $\gamma^2((a,b,0))=(a,b,2)$, which contradicts our assumption that $\gamma$ is an involution. We arrive at the same contradiction in the case where $\gamma((a,b,0))=(a,b,3)$.
\end{proof}

\begin{lemma}\label{when_inter}
	Let $\fb$ be a structure which is interpretable in  $\atoms$. Then either $\fb$ is finite, or $\fb$ has a subgroup isomorphic to $\sym({\mathbb Q})$.
\end{lemma}
	
\begin{proof}
	Let us fix an interpretation $I$ of $\fb$ in $\atoms$ and let $d$ be its dimension. 
    Let $h_I$ be the continuous group homomorphism $h_I \colon \sym(\mathbb{Q})\rightarrow \aut(\fb)$ from Remark~\ref{rem:action}. The kernel of $h_I$ is a closed normal subgroup of $\sym({\mathbb Q})$. It is a well-known fact that every closed normal subgroup of $\sym({\mathbb Q})$ is trivial~\cite{SchreierUlam}. If the kernel of {$h_I$} equals $\{\id_{{\mathbb Q}}\}$, then $h_I$ is an isomorphism between $\sym({\mathbb Q})$ and a subgroup of $\aut(\fb)$, and we are done. Otherwise, $\kernel(h)=\sym({\mathbb Q})$. For $u,v\in {\mathbb Q}^d$ we write $u\same v$ if the same coordinates of $u$ and $v$ are equal, i.e., $u_i=u_j$ if and only if $v_i=v_j$. Then $\same$ is an equivalence relation with finitely many classes. If $u\same v$ then $\alpha(u)=v$ for some $\alpha\in \sym({\mathbb Q})$. Since $h(\alpha)$ is the identity, we obtain that $I(u) = I(v)$.  This implies that $\fb$ is finite.
\end{proof}

\begin{lemma}\label{not_inter}
	The structure $\fy$ is not interpretable in  $\atoms$.
\end{lemma}

\begin{proof}
	Suppose for contradiction that $\fy$ has an interpretation in $\atoms$. Since $\fy$ is infinite, it follows from Lemma~\ref{when_inter} that $\aut(\fy)$ has a subgroup which is isomorphic to $\sym({\mathbb Q})$. In particular, $\aut(\fy)$ contains two involutions that do not commute. This contradicts Lemma \ref{invol}.
\end{proof}

\begin{corollary}\label{not_closed_mc}
	The class of $\I(\atoms)$ is not closed under taking model-complete cores.
\end{corollary}

\begin{proof}
	By Lemma \ref{interpretable} the structure $\fx$ is interpretable in  $\atoms$. By Corollary \ref{mc_core} the model-complete core of $\fx$ is $\fy$, and by Lemma \ref{not_inter} the structure $\fy$ is not interpretable in  $\atoms$.
\end{proof}


    \begin{remark} 
    We mention without proof that both of the structures $\fx$ and $\fy$ in this subsection are finitely homogenizable 
    and thus do not give an answer to Question~\ref{quest:fin-hom} stated in the introduction. 
\end{remark}

\subsection{Structures interpretable {in} $({\mathbb Q};<)$}
\label{sect:inter_q}
In this subsection we show that the class of structures with an interpretation in $({\mathbb Q};<)$ is not closed under taking model companions. 
We start with a general observation about structures in $\I((\mathbb{Q};<))$ (stated in Lemma~\ref{nice_reducts}) which essentially allows us to get rid of taking quotients in the definition of interpretations, i.e., equality is always interpreted as equality of tuples with some additional finite information stored in unary relations.

\begin{notation}\label{not:qup}
	{For $d \in {\mathbb N}$,}  we define $\qup{d}\coloneqq \{(a_1,\dots,a_d)\in \mathbb{Q}^d: a_1<\dots<a_d\}$. 
	Let {$e \in {\mathbb N}$,} $R\in \{=,\neq,<,>\}$ and $\bar{a}=(a_1,\dots,a_d){\in \mathbb{Q}^d},\bar{b}=(b_1,\dots,b_e)\in \mathbb{Q}^e$. For $i,j \in \mathbb{N}$, we write $\bar{a}R_{ij}\bar{b}$ if $i\leq d$, $j\leq e$, and $a_i R b_j$. We write $R_i$ for $R_{ii}$. We denote by $(\bar{a}R\bar{b})$ the set of all coordinates $i$ such that $\bar{a}R_i\bar{b}$ holds.
	For $a,b\in \mathbb{Q}$ we write $[[a,b]]$ for the interval $[\min(a,b), \max(a,b)]$.
	For $A,B\subseteq \mathbb{Q}$ we write $A<B$ if $a<b$ for all $a\in A$ and $b\in B$.
	For a tuple $\bar{a} \in {\mathbb Q}^d$ we denote by $\pi_i(\bar{a})$ its $i$-th coordinate.
    For $d\in \mathbb{N}$ we write $\johnord{d}$ for the structure whose domain is $\qup{d}$ and whose relations are $<_{ij}$ and $=_{ij}$ for $i,j\in [d]$.
    	 
	For a finite sequence $\bar{d}=(d_1,\dots,d_{\ell})\in {\mathbb N}^{\ell}$ 
    we denote by $\johnord{\bar{d}}$ the structure whose domain is $\bigcup_{i\in [\ell]}{\{i\}\times \qup{d_i}}$ and whose relations are  
    \begin{itemize}
        \item for every $i\in [\ell]$ the unary relation $\pi_1^{-1}(i)$, and 
        \item for all $i,j\in \{1,\dots,\max(d_1,\dots,d_{\ell})\}$ the binary relations 
\begin{align*}
<^*_{ij}\, \coloneqq& \bigcup_{1\leq n,m\leq \ell}\{((n,\bar{a}),(m,\bar{b})): \bar{a}<_{ij}\bar{b}\} \\
\text{ and } \quad =^*_{ij}\, \coloneqq& \bigcup_{1\leq n,m\leq \ell}\{((n,\bar{a}),(m,\bar{b})): \bar{a}=_{ij}\bar{b}\}.
\end{align*}
    \end{itemize}
\end{notation}
{Note that by choosing $d_1=\cdots=d_{\ell}=0$, we may obtain every finite structure as a first-order reduct of $\johnord{\bar{d}}$.}

\begin{proposition}\label{johnhom}
	The structure $\johnord{\bar{d}}$ is homogeneous, finitely bounded, has a binary signature, is Ramsey, and NIP.
\end{proposition}

\begin{proof}[Proof (sketch)]
	The proof of homogeneity and finite boundedness is essentially the same as the proof of Theorem 19 in~\cite{bodirsky2025structures}. 
    The natural action of $\aut(\mathbb{Q};<)$ on $\dom(\johnord{\bar{d}})$ is continuous and surjective to $\aut(\johnord{\bar{d}})$, {and it is injective unless $d_1=\cdots=d_{\ell}=0$, that is $\dom(\johnord{\bar{d}})$ is finite.} Therefore, {either $\johnord{\bar{d}}$ is finite, or} $\aut(\mathbb{Q}{;<})$ and $\aut(\johnord{\bar{d}})$ are isomorphic as topological groups, and thus $\johnord{\bar{d}}$ is Ramsey.
    NIP is clearly preserved by interpretability as well.
\end{proof}

\begin{lemma}\label{nice_reducts}
    Let $\fa$ be a structure with a $d$-dimensional interpretation $I$ in $({\mathbb Q};<)$. Then there exists some $\ell\in \mathbb{N}$ and $\bar{d}\in {\mathbb N}^{\ell}$ such that $\fa$ is isomorphic to a first-order reduct of $\johnord{\bar{d}}$.
\end{lemma}

\begin{proof}
	Let $O_1,\dots,O_{\ell}$ be the orbits of the full power $\mathbb{Q}^{[d]}$ which are contained in $\dom(I)$, and pick some elements $\bar{a}^i=(a_1^i,\dots,a_d^i)\in {O_i}$ for $i \in [\ell]$. Observe that for all $i,j\in [\ell]$ either $I(O_i)\cap I(O_j) =  \emptyset$ or $I(O_i)=I(O_j)$. If the latter holds, then by restricting $I$ to $\dom(I)\setminus O_j$ we still get an interpretation of $\fa$. By iterating this process we can assume without loss of generality that the sets $I(O_i)$, for $i\in [\ell]$, are pairwise disjoint.

    For $\bar{u} \in O_i$, let $d_i \in [d]$ be the number of distinct entries of some (equivalently: every) tuple in $O_i$, and write $\bar{u}^*$ for the tuple in $\mathbb{Q}$ which lists all entries of $\bar{u}$ in increasing order, i.e., $\bar{u}^*=(u^*_1,\dots,u_{d_i}^*)$ where 
    $u_1^*<\dots<u_{d_i}^*$ and $\{u_1^*,\dots,u_{d_i}^*\}=\{u_1,\dots,u_d\}$.
	Let $J_i'$ be the function that maps each tuple $\bar{u}\in O_i$ to $(i,\bar{u}^*)$. Note that in this case $J_i'$ defines a bijection from $O_i$ to $\qup{d_i}$. Moreover, $J_i'$ induces an isomorphism between $\aut(\mathbb{Q}^{[d]})|_{O_i}$ and $\aut(\johnord{d_i})$, because both of these groups consist of induced actions of automorphisms of $\mathbb{Q}$. Now let
\[
E_i\coloneqq \{(\bar{u},\bar{v})\in (\qup{d_i})^2: I((J_i')^{-1}(\bar{u}))=I((J_i')^{-1}(\bar{v}))\}.
\]
	Then $E_i$ is an $\aut(\qup{d})$-invariant equivalence relation on $\qup{d_i}$. By Corollary 25 in~\cite{bodirsky2025structures} we know that in this case there exists some $S_i\subseteq [{d_i}]$ such that $E_i=\{(\bar{u},\bar{v}): S_i \subseteq (\bar{u}=\bar{v})\}$. Let $e_i \coloneqq {|S_i|}$. For a tuple $\bar{u}\in \qup{d_i}$ let us write $\bar{u}/E_i$ for the tuple which lists all elements of $\{\pi_j(\bar{u}): j{\in} S_i\}$ in increasing order. Note that in this case $\{\bar{u}/E_i: \bar{u}\in \qup{d_i}\}=\qup{e_i}$. Finally, we define $J_i \colon \bar{u}\mapsto J_i'(\bar{u})/E_i$, and $J\coloneqq\bigcup_{i\in [\ell]} J_i$. By definition,  $J_i(\bar{u})=J_i(\bar{v})$ holds if and only if $I(\bar{u})=I(\bar{v})$. Therefore, $J\circ I^{-1}$ is well-defined and it defines a bijection from $A$ to the image of $J$, which is $\bigcup_{i=1}^{\ell}({\{i\}\times \qup{e(i)}})$. Moreover, the $J$-image of every relation of the full product $(\mathbb{Q},<)^{[d]}$ is definable from the relations $\pi^{-1}(i)$, $<_{ij}^*$, and $=_{ij}^*$. This implies that for every relation $R$ of $\fa$ the relation $(J\circ I^{-1})(R)$ is definable in 
    $\johnord{\bar{{e}}}$, where ${\bar e} := ({e_1,\dots,e_{\ell}})$. 
    Hence, $\fa$ is isomorphic to a first-order reduct of $\johnord{{\bar{e}}}$, which finishes the proof of the lemma.
\end{proof}

\begin{remark}\label{rem:ramsey-exp}
	Lemma~\ref{nice_reducts} together with Proposition~\ref{johnhom} provide an alternative proof of Corollary~\ref{reduct_hom}: 
    every structure in $\I(({\mathbb Q};<))$ has a homogeneous finitely bounded binary expansion which is Ramsey and NIP.
\end{remark} 

	\begin{remark} 
    A crucial property of $(\mathbb{Q};<)$ that makes the previous proof work is the fact that for every imaginary element $\bar{a}/\theta$ there exists a formula $\varphi(\bar{x},\bar{y})$ such that there exists a unique tuple $\bar{b}\in \mathbb{Q}^d$ such that $\varphi(\bar{c},\bar{b})\Leftrightarrow \bar{c}/\theta=\bar{a}/\theta$. This property is called \emph{elimination of imaginaries}~\cite{HodgesLong}. Indeed, in our setting $\bar{b}$ is defined to be the unique element in $\qup{d}$ which lists all elements that appear in all representations of $\bar{a}/\theta$. Notably, the proof above does not work for $\atoms$, 
    essentially because any finite subset of $\mathbb{Q}$ might have multiple representations which cannot be distinguished in a definable way.
    \end{remark}

\subsubsection{Model-complete cores of structures interpretable in $({\mathbb Q};<)$}\label{sect:in_q}
	In this subsection we show that $(\mathbb{Q};<)$ with the generic bipartition, the generic permutation, and the dense local order  are model companions of  structures
    interpretable in $(\mathbb{Q};<)$.

\begin{notation}\label{not:S2}
	We write 
    \begin{itemize}
        \item $(\mathbb{Q};<)*(\mathbb{Q};<)$ for the \emph{generic permutation}, i.e., the Fra\"{i}ss\'{e} limit of the class of all finite $\{<_1,<_2\}$-structures $\fa$ where both $<_1^{\fa}$ and $<_2^{\fa}$ are strict linear orders.
	\item $(\mathbb{Q};<,S,T)$ for \emph{$(\mathbb{Q};<)$ with a generic  partition}, i.e., the Fra\"{i}ss\'{e} limit of the class of all finite $\{<,S,T\}$-structures $\fa$ where $<^{\fa}$ is a strict linear order, and for each $a\in A$ exactly one of $a\in S^{\fa}$ and $a\in T^{\fa}$ holds.
	\item  $S(2)$ for the \emph{dense local order}, which is defined as follows. The domain set of $S(2)$ is a countable dense subset $\cc$ of a unit circle which does not contain opposite, (i.e., antipodal) points. Then $S(2) :=(\cc;\prec)$ where $a\prec b$ if the arc from $a$ to $b$ going in the positive direction is shorter than $\pi$. We know that these properties characterize $S(2)$ uniquely up to isomorphism, and that $S(2)$ is homogeneous~\cite{CameronOrbits,MacphersonSurvey}. 
    \end{itemize}
\end{notation}

	We first show that 
    $({\mathbb Q};<) * ({\mathbb Q}; <)$, $({\mathbb Q};<,S,T)$, and $S(2)$
    are model companions of structures in $\I((\mathbb{Q};<))$.
    


\begin{lemma}\label{companion_perm}
	For $i=1,2$ let $\prec_i$ denote the relation $<_i \cup \, (=_i\cap <_{3-i})$ on $\mathbb{Q}^2$. Then $(\mathbb{Q};<)*(\mathbb{Q};<)$ is the model companion of $(\mathbb{Q}^2;\prec_1,\prec_2)$.
\end{lemma}

\begin{proof}
	The structure $(\mathbb{Q};<)*(\mathbb{Q};<)$ is  by definition homogeneous; thus, it is enough to show that $(\mathbb{Q};<)*(\mathbb{Q};<)$ and $(\mathbb{Q}^2;\prec_1,\prec_2)$ have the same age, i.e., for a finite structure $\fa = (A;\prec_1,\prec_2)$, we have $\fa \in \age(\mathbb{Q}^2;\prec_1,\prec_2)$ if and only if $\prec_1$ and $\prec_2$ are linear orders on $A$. The `only if' direction is clear. For the other direction let us assume that $A=\{c_1,\dots,c_n\}$, and that $\prec_1$ and $\prec_2$ are linear orders on $A$. Choose rational numbers $a_1,\dots,a_n,b_1,\dots,b_n$ such that $a_i<a_j$ iff $c_i\prec_1c_j$ and $b_i<b_j$ iff $c_i\prec_2c_j$. Then $c_i\mapsto (a_i,b_i)$ is an embedding from $(A;\prec_1,\prec_2)$ into $(\mathbb{Q}^2;<_1,<_2)$, and thus also into $(\mathbb{Q}^2;\prec_1,\prec_2)$.
\end{proof}

\begin{lemma}\label{mc_core_part}
	$(\mathbb{Q};<,S,T)$ is the model companion of the $\{<,S,T\}$-structure $\fa$ whose domain is  $A=\mathbb{Q}\times [2]$, and
\begin{itemize}
\item $<^{\fa}=\{((a,i),(b,j)): a<b\vee (a=b\wedge i<j) \}$, 
\item $S^{\fa}=\pi_2^{-1}(1)$, and 
\item $T^{\fa}=\pi_2^{-1}(2)$.
\end{itemize}
\end{lemma}

\begin{proof}
	We know that $(\mathbb{Q};<,S,T)$ is homogeneous, and hence model complete. We have to show that $(\mathbb{Q};<,S,T)$ and $\fa$ are bi-embeddable. Let us consider the map \[e\colon \mathbb{Q}\mapsto \mathbb{Q}\times [2], \text{ such that} r\mapsto
\begin{cases*}
(r,1) & if $r\in S$,\\
(r,2) & if $r\in T$.
\end{cases*}
\]
	This embeds $(\mathbb{Q};<,S,T)$ into $\fa$. The existence of an embedding in the other direction follows from the fact that $\age(\fa)\subseteq \age(\mathbb{Q};<,S,T)$.
\end{proof}

	Note that $(\mathbb{Q};<)$ interprets the structure $\fa$ with domain $A = \mathbb{Q}\times [2]$ from Lemma~\ref{mc_core_part}.
The following lemma establishes a connection between $(\mathbb{Q};<,S,T)$ and $S(2)$.
	
\begin{lemma}\label{s2_part}
	Let $c\in \cc$ be arbitrary. Then $(S(2);c)$ interprets $(\mathbb{Q};<,S,T)$, and $S(2)$ is isomorphic to a quantifier-free reduct of $(\mathbb{Q};<,S,T)$.
\end{lemma}

\begin{proof}
	Let us consider the structure $(\cc\setminus \{c\};<,S,T)$ where $S=\{a: c\prec a\}$, $T=\{a: c\succ a\}$ and 
\begin{align*}
<\, \coloneqq \{(a,b): \bigl((a\in S\Leftrightarrow b\in S)\wedge a\prec b\bigr)\vee \bigl((a\in S\Leftrightarrow b\in T)\wedge a\succ b\bigr)\}.
\end{align*}

	We can think of the relation $<$ in the following way. We identify each point in $T$ with its opposite point and then consider the restriction of the order defined on the half-circle containing $S$ going from $c$ to its opposite point. This implies immediately that $<$ linearly orders $S\cup T=\cc\setminus \{c\}$ and both $S$ and $T$ are dense according to $<$. Therefore, $(\cc\setminus \{c\};<,S,T)$ is isomorphic to $(\mathbb{Q};<,S,T)$, and thus $(\mathbb{Q};<,S,T)$ is interpretable in $(S(2);c)$.
	
	For the second claim we can reverse the construction, i.e., we can define $\prec$ as
\[
	\prec\, \coloneqq \{(a,b): \bigl(a<b\wedge (a\in S\Leftrightarrow b\in S)\bigr)\vee \bigl(a>b\wedge (a\in S\Leftrightarrow b\in T)\bigr)\}
\]
This defines $S(2)|_{\cc\setminus c}$ which is isomorphic to $S(2)$.
\end{proof}

\begin{remark}
	Our constructions given in the proof of Lemma~\ref{s2_part} come from the correspondence described in Proposition 6.1 in~\cite{CameronOrbits}. 
\end{remark}

\begin{corollary}\label{companion_s2}
    $S(2)$ is the model companion of a structure in $\I(({\mathbb Q};<))$.
\end{corollary}

\begin{proof}
    By Lemma~\ref{mc_core_part} we know that $({\mathbb Q};<,S,T)$ is the model companion of a structure $\bA$ with an interpretation in $({\mathbb Q};<)$, and by Lemma~\ref{s2_part} we know that $S(2)$ is isomorphic to a quantifier-free  reduct of $({\mathbb Q};{<},S,T)$.  Hence, if we use the same formulas to form a quantifier-free reduct $\bB$
    of $\bA$, then 
    the same embeddings that show that $({\mathbb Q};<,S,T)$ is a companion of $\bA$ show that $S(2)$ is a companion of $\bB$. 
    Since the structure 
    $S(2)$ is homogeneous, it is the model companion of $\bB$, and hence a model companion of a structure with a first-order {interpretation} in $({\mathbb Q};<)$.
\end{proof}

\subsubsection{Non-interpretability results}
    Here 
    we show that none of the three structures $(\mathbb{Q};<)*(\mathbb{Q};<)$,  $(\mathbb{Q};<,S,T)$, and $S(2)$ is interpretable in $\order$. By Lemmas~\ref{companion_perm}, ~\ref{mc_core_part}, and Corollary~\ref{companion_s2} we know that 
    each of these structures is the model companion of a structure
    interpretable in $\order$, so 
    any of these three structures shows that $\I(\order)$ is not closed under taking model companions.
    
	We first show that every order which is definable in $\johnord{d}$ equals a lexicographic order up to a permutation of the coordinates and the reversal of the order of some of the coordinates. We say that an interval of $\mathbb{Q}$ is \emph{non-degenerate} if it contains at least 2 elements. 
	
\begin{lemma}\label{most_sign}
	Let $d\in \mathbb{N}$, let $I_1<I_2<\dots<I_d$ be some non-degenerate intervals in $\mathbb{Q}$,
    and let $\prec$ be an order which is first-order definable in $(\qup{d};<_1,\dots,<_d)$. Then there are $k\in [d]$ and $R\in \{\prec,\succ\}$ such that for all $\bar{a},\bar{b}\in \mathcal{I}\coloneqq \prod_{i=1}^dI_i$ the relation $a_k<b_k$ implies $\bar{a}R\bar{b}$.
\end{lemma}

\begin{proof}
	We first prove the following statement by induction on $m \geq 1$:
	
	\emph{There exists some $j\in [m]$ and $R\in \{\prec,\succ\}$ such that for all $\bar{a},\bar{b}\in {\mathcal{I}}$ such that $a_j<b_j$ and $a_i=b_i$ for all $i>m$ we have $\bar{a}R\bar{b}$. $(\ast)$}
	
	For $m=d$ we obtain the statement of the lemma.
	
	For $m=1$ the statement is obvious. For the induction step from $m-1$ to $m$, let us fix some tuple $\bar{c}\in {\mathcal{I}}$, and let $j$ be the index for which the induction hypothesis holds. We define
\begin{align*}
\gamma\colon&I_j\times I_m\rightarrow \mathbb{Q}^d\\
&(a_1,a_2)\mapsto(c_1,\dots,c_{j-1},a_1,c_{j+1},\dots,c_{m-1},a_2,c_{m+1},\dots,c_d).
\end{align*}
	
	{\bf Claim.} For all $(a_1,a_2),(b_1,b_2)\in I_j\times I_m$ either $\tp(a_1,b_1)$ or $\tp(a_2,b_2)$ determines the $\prec$-order of $(\gamma(a_1,a_2),\gamma(b_1,b_2))$.
	
	Pick some elements $s_1,t_1,u_1,v_1\in I_j$ and $s_2,t_2,u_2,v_2\in I_m$ with $s_1<t_1<u_1<v_1$ and $s_2<t_2<u_2<v_2$. Suppose first that the pairs $(\gamma(t_1,t_2),\gamma(v_1,v_2))$ and $(\gamma(t_1,t_2),\gamma(v_1,s_2))$ are ordered in the same way according to $\prec$. We assume that $\gamma(t_1,t_2)\prec\gamma(v_1,v_2)$, the other case is symmetric. Then $\tp(\gamma(t_1,t_2),\gamma(u_1,u_2))=\tp(\gamma(t_1,t_2),\gamma(v_1,v_2))$ and $\tp(\gamma(u_1,u_2),\gamma(v_1,t_2))=\tp(\gamma(t_1,t_2),\gamma(v_1,s_2))$. Therefore, $\gamma(t_1,t_2)\prec \gamma(u_1,u_2)\prec \gamma(v_1,t_2)$. This implies that for all $a_1,a_2,b_1,b_2$ as in our assumption, if $a_1<b_1$, then $\gamma(a_1,a_2)\prec \gamma(b_1,b_2)$. Similarly we are done in the case when the pairs $(\gamma(t_1,t_2),\gamma(v_1,v_2))$ and $(\gamma(t_1,t_2),\gamma(s_1,v_2))$ are ordered in the same way according to $\prec$. In the remaining case we get that the pairs $(\gamma(t_1,t_2),\gamma(s_1,v_2))$ and $(\gamma(t_1,t_2),\gamma(v_1,s_2))$ are ordered in the same way according to $\prec$. However, this is  impossible, because $\tp(\gamma(t_1,t_2),\gamma(v_1,s_2))=\tp(\gamma(s_1,v_2),\gamma(t_1,t_2))$.

	By reversing the order $\prec$ if necessary we can assume without loss of generality that either $a_1<b_1$ or $a_2<b_2$ implies $\gamma(a_1,a_2)\prec \gamma(b_1,b_2)$ for all $a_1,a_2,b_1,b_2$ as above. We handle the two cases separately.
	
	\emph{Case 1. $a_1<b_1$ implies $\gamma(a_1,a_2)\prec \gamma(b_1,b_2)$ for all $(a_1,a_2),(b_1,b_2)\in I_j\times I_m$.} First note that if $c'\in I_m$ and $a_1<b_1$ then for $\bar{u}=\gamma(a_1,c'), \bar{v}=\gamma({b_1},c')$ we have $\bar{u}\prec \bar{v}$ and $(\bar{u}=\bar{v})\supseteq [d]\setminus [m-1]$. Thus, by the induction hypothesis for all tuples $\bar{u},\bar{v}\in I$ with $u_j<v_j$ and $u_i=v_i$ for all $i\geq m$ we have $\bar{u}\prec \bar{v}$.

	Let us now consider some tuples $\bar{u},\bar{v}\in {\mathcal{I}}$ such that $u_j<v_j$ and $u_i=v_i$ for all $i\geq m+1$. We have to show that $\bar{u}\prec \bar{v}$. Let us pick a $\bar{w}\in {\mathcal{I}}$ such that $u_j<w_j<v_j, w_m=u_m$ and $w_i=v_i$ for $i\neq j,m$.	Then by our argument above it follows that $\bar{u}\prec \bar{w}$. On the other hand, if we pick some $(a_1,a_2),(b_1,b_2)\in I_j\times I_m$ such that $a_1<b_1$ and $\tp(a_2,b_2)=\tp(u_m,v_m)=\tp(w_m,v_m)$ then we have $\tp(\bar{w},\bar{v})=\tp(\gamma(a_1,a_2),\gamma(b_1,b_2))$. By our assumption we have $\gamma(a_1,a_2)\prec\gamma(b_1,b_2)$. Therefore $\bar{w}\prec \bar{v}$. Combining this with the previous inequality we obtain $\bar{u}\prec \bar{v}$.

	\emph{Case 2. $a_2<b_2$ implies $\gamma(a_1,a_2)\prec \gamma(b_1,b_2)$ for all $(a_1,a_2),(b_1,b_2)\in I_j\times I_m$.} By the induction hypothesis there exists some $R\in \{<,>\}$ such that for all tuples $\bar{u},\bar{v}\in {\mathcal{I}}$ if $u_jRv_j$ then $\bar{u}\prec \bar{v}$. Let $\bar{u},\bar{v}\in {\mathcal{I}}$ such that $u_m<v_m$ and $u_i=v_i$ for $i\geq m+1$. We have to show that $\bar{u}\prec \bar{v}$. Let $\bar{w}\in {\mathcal{I}}$ such that $w_m=u_m, u_jRw_j$ and $v_i=w_i$ for $i\neq j,m$. Then by the definition of $R$ we have $\bar{u}\prec \bar{w}$. On the other hand, if we pick some $(a_1,a_2),(b_1,b_2)\in I_j\times I_m$ such that {$a_2<b_2$ and $\tp(a_1,b_1)=\tp(w_j,v_j)$} then $\tp(\bar{w},\bar{v})=\tp(\gamma({a_1,a_2}),\gamma({b_1,b_2}))$. This implies $\bar{w}\prec \bar{v}$, and therefore $\bar{u}\prec \bar{v}$.
\end{proof}

\begin{lemma}\label{lexicographic}
	Let $d\in \mathbb{N}$ and let $\prec$ be an order which is first-order definable in $\johnord{d}$. Then there are  $R_1,\dots,R_d\in \{<,>\}$ and $\sigma \in \Sym([d])$ such that for all $\bar{a},\bar{b}\in \qup{d}$ we have $\bar{a}\prec \bar{b}$ if and only if there exists some $j\leq d$ such that 
    $\{\sigma(1),\dots,\sigma({j-1})\} \subseteq (\bar{a}=\bar{b})$, and $a_{\sigma(j)}R_jb_{\sigma(j)}$.
\end{lemma}

\begin{proof}
	We show by induction on $m \in {\mathbb N}$, $m \leq d$, that we can choose some $R_1,\dots,R_m \in \{<,>\}$ and distinct $\sigma(1),\dots,\sigma(m) \in [d]$ such that for all $\bar{a},\bar{b}\in \qup{d}$, whenever there exists some $j\leq m$ such that $(\bar{a}=\bar{b})\supseteq \{\sigma(1),\dots,\sigma(j-1)\}$ and $a_{\sigma(j)} R_j b_{\sigma(j)}$, then $\bar{a}\prec \bar{b}$. $(\dag_m)$.
	
	For $m=d$ we obtain the statement of the lemma.

	We prove the statement by induction on $m$. For $m=0$ there is nothing to prove.
	
	For the induction hypothesis let us assume we already chose some $R_1,\dots,R_m$ and $\sigma(1),\dots,\sigma(m)$ such that $(\dag_m)$ holds. We introduce the following notation for ``shuffling'' tuples. 
    For $\bar{u}\in \mathbb{Q}^{d-m}$ and $\bar{c}\in \mathbb{Q}^m$ we write $\bar{u}\shuff_{\sigma} \bar{c}$ for the tuple whose $\sigma(i)$-th coordinate is $c_i$ for all $i \in [m]$, and such that the tuple formed by listing all the other coordinates in order is $\bar{u}$. Let us fix $\bar{c}\in \mathbb{Q}^m$ such that $$D\coloneqq \{\bar{u}: \bar{u}\shuff_{\sigma} \bar{c}\in \qup{d}\}\neq \emptyset.$$

	Note that in this case $\{\bar{u}\shuff_{\sigma} \bar{c}: \bar{u}\in \mathbb{Q}^{d-m}\}\cap \qup{d}$ is a $\prec$-interval by the induction hypothesis. Let us choose some non-degenerate intervals $I_1<\dots<I_{d-m}$ of $\mathbb{Q}$ such that $\mathcal{I}\coloneqq \prod_{i=1}^{d-m}I_i\subseteq D$. Note that for $\bar{u},\bar{v}\in D$ the $((<_{ij}))_{i,j\in [d]}$-type of the pair $(\bar{u}\shuff_{\sigma} \bar{c}, \bar{v}\shuff_{\sigma} \bar{c})$ is uniquely determined by the $((<_i))_{i\in [d]}$-type of the pair $(\bar{u},\bar{v})$. This implies that the order $$\prec^*\coloneqq \{(\bar{u},\bar{v}): \bar{u}\shuff_{\sigma} \bar{c}\prec \bar{v}\shuff_{\sigma} \bar{c})\}$$ is definable in $(\mathcal{I};<_1,\dots,<_d)$. Therefore, by Lemma~\ref{most_sign} we obtain that there exists some $k$ and some $R\in \{<,>\}$ such that $u_kRv_k$ implies $\bar{u}\prec^*\bar{v}$ for all $\bar{u},\bar{v}\in \mathcal{I}$. Since all sequences of intervals $(I_1,\dots,I_{d-m})$ as above can be mapped to each other by an automorphism of $\aut(\mathbb{Q})$ fixing $\bar{c}$, it follows that in fact the {relation $R$ and the} value $k$ does not depend on our particular choice of the intervals $I_i$.
	
	Now we show that in fact for all $\bar{u},\bar{v}\in D$ the relation $u_kRv_k$ implies $\bar{u}\prec^*\bar{v}$. For simplifying the notation we assume that $R$ is $<$; the proof of the other case is analogous. So suppose that $u_k < v_k$. We define the tuples $\bar{u}^0,\dots,\bar{u}^d,\bar{v}^0,\dots,\bar{v}^d$
    as follows. Let $J$
    be an inclusion-wise minimal non-degenerate interval {such that any of its existing endpoints belongs to $\bar{c}$.}
   Let $u_i<u_{i+1}<\dots<u_j$ be the elements of $\bar{u}$ which are contained in $J$. Note that in this case $v_i<v_{i+1}<\dots<v_j$ are also exactly the elements of $\bar{v}$ in $J$. Let us pick some elements $w_i<\dots<w_j$ {in $J$} such that
	
\begin{itemize}
\item if $\ell<k$ then $w_{\ell}<\min(u_i,v_i)$,
\item if $\ell>k$ then $w_{\ell}>\max(u_j,v_j)$,
\item $u_k<w_k<v_k$.
\end{itemize}
For $\ell\in \{i,\dots,j\} \setminus \{k\}$ we define
\[
u^n_{\ell}\coloneqq 
\begin{cases*}
w_{\ell} & if $\ell<\min(k,i+n)$ or $\ell>\max(k,j-n)$ \\
u_{\ell} & otherwise.
\end{cases*}
\]
	The element $v^n_{\ell}$ is defined analogously. Finally, we define the elements $u^n_k$ and $v^n_k$ in such a way that
\[
u_k=u^0_k<u^1_k<\dots<u^d_k<u_{k+1},w_k,
\]
	and
\[
v_k=v^0_k>v^1_k>\dots>v^d_k>v_{k-1},w_k
\]
	{where $u_{d-m+1}$ and $v_0$ are considered to be $\infty$ and $-\infty$, respectively.}

	Note that by definition all tuples $\bar{u}^n$, $\bar{v}^n$, and $\bar{w}$ are in $D$, because they all have the same type over $\bar{c}$. Moreover, it follows from our choice that for all $n\in \mathbb{N}$ the intervals $[[u^n_{\ell},u^{n+1}_{\ell}]]: \ell\in [d-m]$ are pairwise disjoint. By making these intervals slightly bigger {we can assume that} we are in the setting of Lemma~\ref{most_sign}. Since ${u^n_k<u^{n+1}_k}$, we have that $\bar{u}^n\prec^*\bar{u}^{n+1}$. A similar argument shows that $\bar{u}^d\prec^*\bar{w}$. Therefore, we get $\bar{u}\prec^*\bar{w}$. By repeating the argument for the $v$-tuples we obtain $\bar{w}\prec^*\bar{v}$. Therefore, $\bar{u}\prec^*\bar{v}$, as we claimed.

	Now we can finish the induction step. We pick $\sigma({m+1})$ to be the position of the $k$-th coordinate after applying the shuffling map $\bar{u}\mapsto \bar{u}\shuff_{\sigma}\bar{c}$, and we put $R_{m+1}\coloneqq R$. Then clearly $\sigma(m+1) \not\in \{\sigma(1),\dots,\sigma(m)\}$. Now let $\bar{a},\bar{b}\in \qup{d}$ be arbitrary. If $(\bar{a}=\bar{b})\not\supseteq \{\sigma(1),\dots,\sigma(m)\}$ then we are done by the induction hypothesis. Otherwise, there exist some $\bar{u},\bar{v}\in \qup{d-m}$ and $\bar{c}'\in \qup{m}$ such that $\bar{a}=\bar{u}\shuff \bar{c}'$ and $\bar{b}=\bar{v}\shuff \bar{c}'$. Let us assume that $a_{\sigma(m+1)}Rb_{\sigma(m+1)}$. We have to show that $\bar{a}\prec \bar{b}$. This is enough to show in the case where  $\bar{c}'=\bar{c}$. In this case $u_k<v_k$, and therefore by our claim above $\bar{u}\prec^*\bar{v}$, that is, $\bar{a}=\bar{u}\shuff \bar{c}\prec \bar{v}\shuff \bar{c}=\bar{b}$.
\end{proof}

	Our next goal is to show that $(\mathbb{Q};<,S,T)$ is not interpretable in $(\mathbb{Q};<)$. In our first step we want to use Lemma~\ref{nice_reducts} but we want to reduce the number $\ell$ as much as possible. In order to do this, assuming that $(\mathbb{Q};<,S,T)$ is a reduct of $\johnord{\bar{d}}$, we colour the elements of $\mathbb{Q}$ according to their first coordinates. We first need the following general observation about colourings of $\mathbb{Q}$.
	
\begin{lemma}\label{colours}
	Let $\lambda \colon \mathbb{Q}\rightarrow [n]$ be a colouring of $\mathbb{Q}$ such that $\lambda(S)\cap \lambda(T)=\emptyset$. Then there exists a non-degenerate interval $I$ and some colours $s\in \lambda(S)$ and $t\in \lambda(T)$ such that both $\lambda^{-1}(s)$ and $\lambda^{-1}(t)$ are dense in $I$.
\end{lemma}

\begin{proof}
	We prove this by induction on $n$. For $n=2$ there is nothing to prove. Otherwise, by relabelling the colours we can assume that $\lambda(S)\subseteq [k]$ and $\lambda(T)\subseteq [n]\setminus [k]$ for some $k\in [n-1]$. Then either $k\geq 2$ or $n-k\geq 2$. Without loss of generality we may assume that the latter holds. 
	
	Let us first assume that $\lambda^{-1}(n)$ is dense in $\mathbb{Q}$. In this case, let $\lambda'\coloneqq a\mapsto \min(\lambda(a),n-1)$, and let us apply the induction hypothesis to the coloring $\lambda'$. We can conclude that there exists some non-degenerate interval $I$ and some $s\in {\lambda(S)}$ such that $\lambda^{-1}(s)$ is dense in $I$, and by our assumption $\lambda^{-1}(n)$ is dense in $I$.
	
	Now let us assume that $\lambda^{-1}(n)$ is not dense in $\mathbb{Q}$.
Then we can find some non-degenerate interval $I$ such that $\lambda^{-1}(n)\cap I=\emptyset$. By shrinking the interval $I$ further we can assume that it is open.
Then $(I;{<}|_I,S \cap I,T \cap I)$ is isomorphic to $(\mathbb{Q};<,S,T)$ and thus we can apply the induction hypothesis to the colouring $\lambda|_I$ which by our assumption avoids the colour $n$.
\end{proof}

	Now we are ready to show one of the main results of this section.

\begin{theorem}\label{not_interpretable_part}
	$(\mathbb{Q};<)$ does not interpret $(\mathbb{Q};<,S,T)$.
\end{theorem}

\begin{proof}
	Suppose for contradiction that $\mathbb{Q}$ interprets $(\mathbb{Q};<,S,T)$. By Lemma~\ref{nice_reducts} we know that there exists a $\bar{d}\in {\mathbb N}^{\ell}$, 
    a first-order reduct $\fc$ of  
    $\johnord{{\bar{d}}}$,
    and an isomorphism $\alpha \colon \fc \to (\mathbb{Q};<,S,T)$. 
    Let us now consider the colouring $\lambda := \pi_1 \alpha^{-1} \colon {\mathbb Q} \to [\ell]$. Note that for a given $i \in [\ell]$ all elements in $\pi_1^{-1}(i)$ lie in the same orbit in $\aut(\fc)$. This implies that $\lambda(S)\cap \lambda(T)=\emptyset$, and thus we can apply Lemma~\ref{colours}. Let $s\in {\lambda(S)},t\in {\lambda(T)}$ and $I'$ be the interval from the conclusion of Lemma~\ref{colours}. We can assume without loss of generality that $s=1$ and $t=2$. Let $\fb=\johnord{d_1,d_2}\subseteq \johnord{\bar{d}}$, and let $\prec\coloneqq \alpha_B^{-1}(<)$ and $I\coloneqq \alpha_B^{-1}(I')$. Then by our construction $\prec$ is definable in $\fb$, $I$ is a $\prec$-interval in $B$, and $\pi_1^{-1}(1)$ and $\pi_1^{-1}(2)$ are both dense in $(I;\prec)$. By abusing notation we also write $S=\pi_1^{-1}(1)$ and $T=\pi_1^{-1}(2)$.

	Let us now apply Lemma~\ref{lexicographic} to the order ${\prec}|_S$. 
    We can conclude that there exist some non-degenerate subinterval $J'$ of $I$, some $j_1\in [d_1]$, and a relation $R_1\in \{<,>\}$ such that for all $(1,a),(1,b)\in J'$ we have $(a=b)\supseteq [d_1]\setminus \{j_1\}$ and $a\prec b$ if and only if $a_{j_1}R_1b_{j_1}$. By repeating the argument for the order ${\prec}|_T$ we can find some non-degenerate subinterval $J$ of $J'$, some $j_2\in [d_2]$, and a relation $R_2\in \{<,>\}$ such that for all $(2,a),(2,b)\in J$ we have $(a=b)\supseteq [d_2]\setminus \{j_2\}$ and $a\prec b$ if and only if $a_{j_2}R_2b_{j_2}$. Define 
\[
\alpha_1 \colon J\cap S\rightarrow \mathbb{Q}, (1,a)\mapsto a_{j_1}, \alpha_2 \colon  J\cap T\rightarrow \mathbb{Q}, (2,a)\mapsto a_{j_2}, \alpha\coloneqq \alpha_1\cup \alpha_2.
\]

	Note that the maps $\alpha_1$ and $\alpha_2$ are both either order-preserving or order-reversing, and the images $\alpha_1$ and $\alpha_2$ are both intervals in $(\mathbb{Q};<)$. Let $$A\coloneqq \{a_1,\dots,a_{j_1-1},a_{j_1+1},\dots,a_{d_1},b_1,\dots,b_{j_2-1},b_{j_2+1},\dots,b_{d_2}\}$$ for either (or equivalently for all) $(1,a),(2,b)\in J$. By shrinking the interval $J$ even further we can assume without loss of generality that the image of $\alpha$ is disjoint from $A$. Note that in this case the type of any pair $(u,v)\in J^2$ is determined by the types of $u$ and $v$ and $(\alpha(u),\alpha(v))$. Now if the intervals $\alpha(J\cap S)=\alpha_1(J\cap S)$ and $\alpha(J\cap T)=\alpha_2(J\cap T)$ are disjoint then each pair $(u,v)\in (J\cap S,J\cap T)$ has the same type {with respect to} $\fb$ and thus also in $(I;\prec,S,T)$ which contradicts our assumption that both $S$ and $T$ are dense in $J$. So we can assume that this is not the case, {i.e., $\alpha(J\cap S)\cap \alpha(J\cap T)\neq \emptyset$. {In this case this intersection must contain some non-degenerate interval $L\subseteq (\mathbb{Q};<)$}. Let $r\in {L}$} and put $u\coloneqq \alpha_1^{-1}(r),v\coloneqq \alpha_2^{-1}(r)$. We can assume without loss of generality that $u\prec v$. Since {$S$ is dense} in $J$ we obtain that there exists some $w\in J\cap S$ such that $u\prec w\prec v$. Let us first assume that $\alpha_1(w)>\alpha_1(u)$, the other case is analogous. In this case $\alpha_1$ is order-preserving. Let $u'\in J\cap S$ be arbitrary. If $\alpha_1(u')\leq \alpha_1(u)$ then $u'\preceq u\prec v$. On the other hand, if $\alpha_1(u')>\alpha_1(u)$ then the pairs $(w,v)$ and $(u',v)$ have the same type with respect to $\fb$. In particular, $u'\prec v$. Since the same argument can be repeated for all ${r\in L}$ we obtain that for all $(u,v)\in {\alpha_1^{-1}(L)\times \alpha_2^{-1}(L)}$ the inequality $u\prec v$ holds. In this case however neither $S$ nor $T$ is dense in $J$, a contradiction.
\end{proof}

	Lemma~\ref{mc_core_part} and Theorem~\ref{not_interpretable_part} imply the following.
	
\begin{corollary}\label{not_closed_mc_q}
	$\I((\mathbb{Q};<))$ is not closed under taking model companions.
\end{corollary}

	We next show using Theorem~\ref{not_interpretable_part} that the generic permutation and the dense local order are also not interpretable in $(\mathbb{Q};<)$. 

\begin{lemma}\label{prem_inter_part}
	$(\mathbb{Q};<)*(\mathbb{Q};<)$ interprets $(\mathbb{Q};<,S,T)$.
\end{lemma} 

\begin{proof}
	We write $<_1$ and $<_2$ for the two orders of $(\mathbb{Q};<)*(\mathbb{Q};<)$. Let $I=\id_D$ where $$D=\{(a,b)\in ((\mathbb{Q};<)*(\mathbb{Q};<))^2: a<_1b\}.$$ We define the relations $<_1,S,T$ on $D$ as follows.
\begin{align*}
	(a,b)<(c,d)&\Leftrightarrow (a <_1 c\vee (a=c\wedge b <_1 d))\\
	(a,b)\in S&\Leftrightarrow a<_2b\\
	(a,b)\in T&\Leftrightarrow a>_2b.
\end{align*}

	Clearly, $I$ is an interpretation of $(D;<,S,T)$ in $(\mathbb{Q};<)*(\mathbb{Q};<)$. We now show that $(D;<,S,T)$ is isomorphic to $(\mathbb{Q};<,S,T)$. It is clear that $(D;<)$ is isomorphic to $(\mathbb{Q};<)$ and $(S,T)$ is a partition of $D$. It remains to show that $S$ and $T$ are dense in $D$. For this it is enough to show that $S$ and $T$ intersect all intervals of the form $J= [(a,b),(a,c)]: a<b<c$. By the genericity of $(\mathbb{Q};<)*(\mathbb{Q};<)$ we can find $d,e$ such that $a<_1b<_1d<_1e<_1c$ and $d<_2{a}<_2e$. Then by definition $(a,d),(a,e)\in J, (a,e)\in S, (a,d)\in T$, so $J\cap S,J\cap T\neq\emptyset$. This shows that $(D;<,S,T)$ is isomorphic to $(\mathbb{Q};<,S,T)$ which implies that $(\mathbb{Q};<)*(\mathbb{Q};<)$ interprets $(\mathbb{Q};<,S,T)$.
\end{proof}

\begin{corollary}\label{not_interpretable_perm}
	$(\mathbb{Q};<)$ does not interpret $(\mathbb{Q};<)*(\mathbb{Q};<)$.
\end{corollary}

\begin{proof}
	Follows from Theorem~\ref{not_interpretable_part} and Lemma~\ref{prem_inter_part}.
\end{proof}

\begin{corollary}\label{not_interpretable_s2}
	$(\mathbb{Q};<)$ does not interpret $S(2)$.
\end{corollary}

\begin{proof}
	Suppose for contradiction that $(\mathbb{Q};<)$ interprets $S(2)$. Let $c\in \cc$. Then by Corollary~\ref{q_add_constants2} we know that $(\mathbb{Q};<)$ also interprets $(S(2),c)$, and thus by the first statement of Lemma~\ref{s2_part} also $(\mathbb{Q};<,S,T)$. However, this is impossible by Theorem~\ref{not_interpretable_part}.
\end{proof}

	Finally, we show a strengthening of Corollary~\ref{not_interpretable_s2} which will be relevant in our discussion in Subsection~\ref{sect:slow}.
	By~\cite[Section 9]{Bennett-thesis}, we know that the structure $S(2)$ has five nontrivial first-order reducts. However, four of these are bidefinable with reducts of $(\mathbb{Q};<)$. The automorphism group of the remaining one can be written as the union of $\aut(S(2))$ and the set of anti-automorphism of $S(2)$, i.e., the bijections of $S(2)$ reversing $\prec$. On the relational side, this first-order reduct, similarly as in the case of $(\mathbb{Q};<)$, can be described by the betweenness relation $\betw$ defined as
$$\betw\coloneqq \{(x,y,z)\in \cc^3: x\prec y\prec z\vee z\prec y\prec x\}.$$
	For the sake of completeness we provide a proof for this correspondence, and then we show that $(\cc,\betw)$ is not interpretable in $(\mathbb{Q};<)$.
	
\begin{proposition}\label{s2betw}
	$\aut(\cc;\betw)$ is the union of the set of automorphisms and anti-automorphisms of $S(2)$.
\end{proposition}

\begin{proof}
	Clearly, every automorphism and anti-automorphism of $S(2)$ preserves $\betw$. 
    For the other direction, choose $\alpha\in \aut(\cc;\betw)$ and 
    $a,b,c,d \in \cc$ arbitrarily such that $a \prec b$ and $c \prec d$. We have to show that the pairs $(\alpha(a),\alpha(b))$ and $(\alpha(c),\alpha(d))$ are ordered in the same way. It follows from the definition of $\prec$ that we can find $e,f\in \cc$ such that $a\prec b\prec e\prec f\prec c\prec d$. Then $(a,b,e),(b,e,f),(e,f,c),(f,c,d)\in \betw$. Since $\alpha$ preserves $\betw$, this implies that either \begin{align*}
        \alpha(a) & \prec \alpha(b)\prec \alpha(e)\prec \alpha(f)\prec \alpha(c)\prec \alpha(d) \\
        \text{ or } \quad \alpha(a) & \succ \alpha(b)\succ \alpha(e)\succ \alpha(f)\succ \alpha(c)\succ \alpha(d),
    \end{align*} 
        which implies what we wanted to show.
\end{proof}

	Proposition~\ref{s2betw} immediately implies the following.
	
\begin{corollary}\label{local_easy}
	Let $c,d\in \cc, c\neq d$. Then $\aut(\cc;\betw)_{c,d}=\aut(\cc;\prec)_{c,d}$; in particular, $(\cc;\prec)$ is a first-order reduct of $(\cc;\betw,c,d)$.
\end{corollary}

\begin{corollary}\label{not_interpretable_circular}
	$(\mathbb{Q};<)$ does not interpret $(\cc;\betw)$.
\end{corollary}

\begin{proof}
	Suppose for contradiction that $(\mathbb{Q};<)$ interprets $(\cc;\betw)$. Let $c,d\in \cc, c\neq d$. Then by Corollary~\ref{q_add_constants2} we know that $(\mathbb{Q};<)$ also interprets $(\cc;\betw,c,d)$, and thus also $S(2)$ by Corollary~\ref{local_easy}. This contradicts Corollary~\ref{not_interpretable_s2}.
\end{proof}

\subsubsection{Slow unlabelled growth and interpretability}\label{sect:slow}
Recall that we denote the number of orbits of the natural action of $\aut(\fa)$ on $n$-element subsets of $\fa$ by $u_n(\fa)$.
	By Theorem 10.4 in~\cite{bodor2025structures} we know (by stating the contrapositive) that if an $\omega$-categorical structure $\fa$ is not interpretable in $\order$ then there exists some polynomial $p$ such that $u_n(\fa)>2^n/p(n)$ for all $n$. On the other hand, let us observe that the unlabelled growth of $(\mathbb{Q};<,S,T)$ is $2^n$ meaning that outside $\I(\order)$ the unlabelled growth of $(\mathbb{Q};<,S,T)$ is close to being as slow as possible. A slightly slower unlabelled growth is attained by $S(2)$ and its betweenness reduct. In the case of $S(2)$ this is the Sloane sequence~\href{https://oeis.org/A000016}{A000016} which is asymptotically $2^{n-1}/n$~\cite{brouwer1980enumeration}, and thus $u_n((\cc;\betw))\sim 2^{n-2}/n$. As far as we know $(\cc;\betw)$ has the slowest possible unlabelled growth outside $\I(\order)$.
	
\begin{conjecture}\label{conj:slow1}
	Let $\fa$ be an $\omega$-categorical structure which is not interpretable in $\order$. Then $u_n(\fa)\geq u_n((\cc;\betw))$ if $n$ is large enough.
\end{conjecture}

    	We note that the theorem we cited above and Conjecture~\ref{conj:slow1} have analogs for primitive structures, as shown and discussed in~\cite{braunfeld2022monadic}. 
	
\begin{theorem}[\cite{braunfeld2022monadic}, Theorem 1.2]
	Let $\fa$ be a primitive $\omega$-categorical structure which is not isomorphic to a first-order reduct of $\order$. Then there exists some polynomial $p$ such that $u_n(\fa)>2^n/p(n)$ for all $n$.
\end{theorem}

	Interestingly, $(\cc;\betw)$ also shows up as the primitive structure with the slowest known unlabelled growth apart from $\order$ and
    its reducts.

\begin{conjecture}[\cite{braunfeld2022monadic}, Conjecture 4]\label{conj:slow2}
	Let $\fa$ be a primitive $\omega$-categorical structure for which $2^n/p(n)<u_n(\fa)<2^n/q(n)$ for some polynomials $p$ and $q$. Then $\fa$ is bidefinable with $S(2)$ or its betweenness reduct.
\end{conjecture}

	We remark that although Conjectures~\ref{conj:slow1} and~\ref{conj:slow2} seem to be connected, there does not seem to be any immediate implication between them in either direction.



\section{Lachlan's Class}
\label{sect:lachlan}
As we have seen in the previous section, 
the class of structures interpretable in equality is not closed under taking model-complete cores. The closure of this class under taking model-complete cores is contained in an interesting class of structures that was introduced and studied by Lachlan; in fact, we conjecture that Lachlan's class \emph{equals} the 
{class of structures that are interdefinable with} model-complete cores of structures interpretable in equality (Conjecture~\ref{conj:D}).



Adapting terminology from Remark~\ref{rem:stronglymin}, we say that a structure $\fa$ is \emph{strongly minimal} if its theory is. Note that in a strongly minimal $\omega$-categorical structure $\fa$ all definable subsets (with parameters) are finite or cofinite. We say that $\fa$ is \emph{strictly minimal} if it is strongly minimal and has primitive automorphism group. 
The following is a consequence/reformulation of the Cherlin-Mills-Zilber theorem about strictly minimal sets (\cite{CherlinHarringtonLachlan}, Theorem 2.1).

\begin{theorem}\label{cml}
	Let $\fa$ be a strictly minimal $\omega$-categorical structure. Then one of the following holds.
\begin{enumerate}
\item $\fa$ is interdefinable with $(A;=)$. 
\item $\fa$ is a first-order reduct of some infinite affine space over a finite field.
\item $\fa$ is a first-order reduct of some infinite projective space over a finite field.
\end{enumerate}
\end{theorem}

Following
Lachlan~\cite{LachlanIndiscernible}, we denote by $\mathcal{D}$ the class of $\omega$-categorical $\omega$-stable structures $\fa$ such that all strictly minimal structures which are interpretable in $\fa$ with parameters are bidefinable with a pure set, and we call it \emph{Lachlan's class}. This class can be characterized in many equivalent ways as shown by the following theorem, which largely follows from the results of~\cite{LachlanIndiscernible} {and~\cite{macpherson1991interpreting}}.

\begin{theorem}
\label{lachlan_extra}
	For any structure $\fa$, the following are equivalent.
\begin{enumerate}
\item\label{it:d} $\fa\in \mathcal{D}$.
\item\label{it:omega_stab} $\fa$ is $\omega$-stable and interpretable in  $({\mathbb Q};<)$.
\item\label{it:stab} $\fa$ is stable and interpretable in  $({\mathbb Q};<)$.
\item\label{it:nsop} $\fa$ is NSOP and interpretable in  $({\mathbb Q};<)$. 
\item\label{it:hom} $\fa$ is $\omega$-stable and a first-order reduct of a finitely homogeneous relational structure. 
\item\label{it:homramsey} $\fa$ is $\omega$-stable and a first-order reduct of a finitely bounded homogeneous Ramsey structure. 
\end{enumerate}
\end{theorem}

\begin{proof}
	The implications (\ref{it:d}) $\Rightarrow$ (\ref{it:omega_stab}) and (\ref{it:stab}) $\Rightarrow$ (\ref{it:d}) follow from~\cite{LachlanIndiscernible}, Theorem 4.6.
	
	The implication (\ref{it:omega_stab}) $\Rightarrow$ (\ref{it:stab}) is trivial.
	
	(\ref{it:stab}) $\Leftrightarrow$ (\ref{it:nsop}). It is well-known that a theory is stable if and only if it is NIP and NSOP (see, e.g., ~\cite[Theorem 2.67]{simon2015guide},~\cite{shelah1990classification}).

	The implication {(\ref{it:omega_stab})} $\Rightarrow$ (\ref{it:homramsey}) follows from 
Corollary~\ref{reduct_hom}.
	
	The implication (\ref{it:homramsey}) $\Rightarrow$ (\ref{it:hom}) is trivial.
	
	(\ref{it:hom}) $\Rightarrow$ (\ref{it:d}). This is essentially due to~\cite[Theorem 3.4]{macpherson1991interpreting}. We show the contrapositive. Let us assume that $\fa\not\in \mathcal{D}$. Then by Theorem~\ref{cml} we know that $\fa$ interprets with parameters a nontrivial first-order reduct of the affine space or the projective space. The automorphism group of these are contained in $\PGammaL(\omega,q)$ or $\AGammaL(\omega,q)$ for some $q$ (i.e., the group of semi-projective/semi-affine transformations), see~\cite{EvansIvanovMacpherson}, Section 1.6. Since $|\PGammaL(\omega,q):\PGL(\omega,q)|=|\AGammaL(\omega,q):\AGL(\omega,q)|$ is finite (it is the dimension of $\mathbb{F}_q$ over its prime field), by adding finitely many constants we obtain that $\fa$ interprets some infinite projective or affine space over a finite field. {By adding a constant to an affine space we can get a vector space. Thus, we obtain that $\fa$ interprets with parameters an infinite vector space or an infinite projective space. However, by Theorems 1.1 and 3.3 in~\cite{macpherson1991interpreting} we know that a finitely homogeneous relational structure cannot interpret (with parameters) any infinite group, or an infinite projective space, a contradiction.}
\end{proof}

\begin{corollary}\label{d_mc}
	Lachlan's class $\mathcal{D}$ is closed under taking model-complete cores.
\end{corollary}

\begin{proof}
	Follows from the equivalence (\ref{it:d}) $\Leftrightarrow$ (\ref{it:homramsey}) in Theorem~\ref{lachlan_extra}, Corollary~\ref{cor:supstabpres} and Theorem~\ref{mc_transfer}.
\end{proof}

\begin{remark}\label{rem:lachlans-strict}
Note that our results show in particular that Lachlan's class is not equal to the class of structures interpretable in  equality.
{This was apparently folklore~\cite{Blog}. In fact, Lachlan's class is closed under taking model-complete cores (Corollary~\ref{d_mc}) while the class of structures interpretable in  equality is not (Corollary~\ref{not_closed_mc}).}

\end{remark} 

\begin{conjecture}\label{conj:D}
    Lachlan's class is equal to the class of structures that are interdefinable with a model-complete core of a structure interpretable in  equality. 
\end{conjecture}

	Conjecture~\ref{conj:D} 
can equivalently be phrased as follows. 

\begin{conjecture}\label{conj:D2}
	Every structure in Lachlan's class is homomorphically equivalent to a structure interpretable in  equality. 
\end{conjecture}

To prove that Conjecture~\ref{conj:D2} is equivalent to Conjecture~\ref{conj:D}, we first prove the following lemma. 

\begin{lemma}\label{lem:MH}
    Let $\mathcal{C}$ be a class of 
$\omega$-categorical 
structures that is closed under taking first-order reducts. Then every structure in $\M({\mathcal C})$ is homomorphically equivalent to some structure in $\mathcal{C}$.
\end{lemma}
\begin{proof}
    Let $\fa$ be a structure 
    which is interdefinable with the model-complete core $\fc$ of some structure $\fb \in \mathcal C$. 
We may assume that $\fc$ is a substructure of $\fb$. 
The definition of $\fa$ in $\fc$ can be made existential positive, because $\fc$ is a model-complete core. 
Let $\fd$ be the reduct of $\fb$ obtained by evaluating the existential positive formulas in $\fb$ rather than $\fc$. By assumption, $\fd \in {\mathcal C}$. 
Moreover, $\fa$ is a substructure of $\fd$,
and the homomorphism from $\fb$ to $\fc$ is also a homomorphism from $\fd$ to $\fa$. Hence, {$\fa$ and $\fd$ are homomorphically equivalent.}
\end{proof} 

\begin{proposition}
    Conjecture~\ref{conj:D} is equivalent to Conjecture~\ref{conj:D2}. 
\end{proposition} 
\begin{proof}
    Lemma~\ref{lem:MH} shows that Conjecture~\ref{conj:D} implies Conjecture~\ref{conj:D2}. 
    Now assume Conjecture~\ref{conj:D2}. 
    Let $\fa$ be in Lachlan's class.
    The expansion
    $\fa^*$ of $\fa$ by all first-order definable relations is also contained in Lachlan's class. 
    Moreover, $\fa^*$ is 
    a model-complete core, and by Conjecture~\ref{conj:D2} it is homomorphically equivalent to a structure interpretable in equality. Thus, $\fa^*$ is the model-complete core of a structure interpretable in  equality, and 
    $\fa^*$ is interdefinable with $\fa$, which proves Conjecture~\ref{conj:D}. 
\end{proof}

\subsection{Trace definability of structures in Lachlan's class}\label{sec:tracedeflach}
When introducing the notion of trace definability, Walsberg proposed the conjecture that there are only three binary finitely homogeneous structures up to trace equivalence: $(\mathbb{Q}; =)$, $(\mathbb{Q};<)$, and the Rado graph~\cite[Conjecture 2.11]{walsberg2026tracedefinabilityipreservation} (cf.~\cite[Conjecture 13.1]{Walsberglong}). 
A positive answer to such a conjecture would tell us, in a sense, that there are only two ``binary'' model theoretic tameness properties: stability (equivalently, not trace defining $(\mathbb{Q}; <)$) and $\mathrm{NIP}$ (equivalently, not trace defining the Rado graph). Since every binary homogeneous structure with $\mathrm{IP}$ is trace-equivalent to the Rado graph, Walsberg's conjecture breaks down into two conjectures, both of which are posed in~\cite{Walsbergshort, Walsberglong}:
\begin{conjecture}[{\cite[Section 9.4]{Walsbergshort}, \cite[p.157]{Walsberglong}}]\label{conj: answered}
    Every stable finitely homogeneous structure is \emph{trace minimal}, i.e., trace definable in $(\mathbb{Q}; =)$.
\end{conjecture}
\begin{conjecture}[{\cite[Section 9.4 (3)]{Walsbergshort}, cf. \cite[Conjecture 13.1]{Walsberglong}}]\label{conj: unanswered}
    Every binary $\mathrm{NIP}$ finitely homogeneous structure is trace definable in $(\mathbb{Q}; <)$. 
\end{conjecture}


In this section, we verify Conjecture~\ref{conj: answered}. In fact, we show that every structure in Lachlan's class is trace minimal; by the equivalence of items~\ref{it:d} and~\ref{it:hom} of Theorem~\ref{lachlan_extra} this indeed implies Conjecture~\ref{conj: answered}. We note that our result would also be implied by Conjecture~\ref{conj:D} using Lemma~\ref{lem:tracedef}. In a similar way, we also ask a stronger version of Conjecture~\ref{conj: unanswered} with respect to model complete cores in the conclusion to this paper.

	We start by defining the structures $\johnneq{d}$ and $\johnneq{\bar{d}}$ which are essentially the same as $\johnord{d}$ and $\johnord{\bar{d}}$, respectively, but we replace all inequality relations with the corresponding disequality relation.
	In the following definition we extend our notation from Notation~\ref{not:qup}.
    
    
	
\begin{notation} 
	{For $d \in {\mathbb N}$,} we define $\qneq{d}\coloneqq \{(a_1,\dots,a_d)\in \mathbb{Q}^d: |\{a_1,\dots,a_d\}|=d\}$, and we write $\johnneq{d}$ for the structure whose domain is $\qneq{d}$ and whose relations are $\neq_{ij}$ and $=_{ij}$ for $i,j\in [d]$ and for a tuple $\bar{d}=(d_1,\dots,d_{\ell})\in {\mathbb N}^{\ell}$ we denote by $\johnneq{\bar{d}}$ the structure whose domain is $\bigcup_{i\in \ell}{\{i\}\times \qneq{d_i}}$ and whose relations are 

\begin{itemize}
        \item for every $i\in [\ell]$ the unary relation $\pi_1^{-1}(i)$, and 
        \item for all $i,j\in \{1,\dots,\max(d_1,\dots,d_{\ell})\}$ the binary relations 
\begin{align*}
\neq^*_{ij}\, \coloneqq& \bigcup_{1\leq n,m\leq \ell}\{((n,\bar{a}),(m,\bar{b})): \bar{a}\neq_{ij}\bar{b}\} \\
\text{ and } \quad =^*_{ij}\, \coloneqq& \bigcup_{1\leq n,m\leq \ell}\{((n,\bar{a}),(m,\bar{b})): \bar{a}=_{ij}\bar{b}\}.
\end{align*}
    \end{itemize}
    
\end{notation}

	
\begin{proposition}\label{johnhomneq}
	For every tuple $\bar{d}\in \mathbb{N}^{\ell}$ the structure $\johnneq{\bar{d}}$ is homogeneous, finitely bounded, has a binary signature and $\omega$-stable. Moreover, {if $\johnneq{\bar{d}}$ is infinite then} the coordinatewise action of $\sym(\mathbb{Q})$ on $\dom(\johnneq{\bar{d}})$ defines a topological isomorphism from $\sym(\mathbb{Q})$ to $\aut(\johnneq{\bar{d}})$.
\end{proposition}

\begin{proof}
	Follows similarly to the proof of Proposition~\ref{johnhom}.
\end{proof}

\begin{corollary}\label{positive_hom}
	Let $\fa\in \{\johnord{\bar{d}},\johnneq{\bar{d}}\}$. Then every type in $\fa$ is isolated by a positive quantifier-free formula.
\end{corollary}

\begin{proof}
	It follows from $\omega$-categoricity of $\fa$ that every type is isolated by a single formula. By the homogeneity of $\fa$, which follows from either Proposition~\ref{johnhom} or Proposition~\ref{johnhomneq}, we can assume that this formula is quantifier-free. Finally, let us observe that in both structures $\johnord{\bar{d}}$ and $\johnneq{\bar{d}}$ the negation of the basic relations are definable by positive quantifier-free formulas, which also allows us to get rid of negations.
\end{proof}

\begin{notation}
 	Let $d\in {\mathbb N}$, and $\bar{d}\in \mathbb{N}^{\ell}$.
	For a tuple $\bar{u}\in \dom(\fa)^n$ where $\fa\in \{\johnord{d},\johnord{\bar{d}}\}$ we write $\typ_<(\bar{u})$ for its type in $\fa$ and for a tuple $\bar{v}\in \dom(\fb)$ where $\fb\in \{\johnneq{d},\johnneq{\bar{d}}\}$ we write $\typ_{\neq}(\bar{v})$ for its type in $\fb$.
	For a relation $S \subseteq \dom(\johnord{\bar{d}})^n$ we write $$S^{\neq}\coloneqq \{\bar{u}\in \dom(\johnneq{\bar{d}})^n: \exists \bar{v}\in S \text{ such that} \typ_{\neq}(\bar{u})=\typ_{\neq}(\bar{v})\}.$$
	For a $\tau$-structure $\fa$ with domain $\dom(\johnord{\bar{d}})$, we write $\fa^{\neq}$ for the $\tau$-structure with domain $\dom(\johnneq{\bar{d}})$ such that for each relation symbol $R \in \tau$ we have $R^{\fa^{\neq}}=(R^{\fa})^{\neq}$.
\end{notation}

\begin{lemma}\label{remove_order}
	If $S$ is a relation definable in $\johnord{\bar{d}}$, then $S^{\neq}$ is definable in $\johnneq{\bar{d}}$.
\end{lemma}

\begin{proof}
	We can assume without loss of generality that $S$ is a single $n$-orbit of $\aut(\johnord{\bar{d}})$ for some $n\in {\mathbb N}$. Let $p$ be the type corresponding to this orbit. By Corollary~\ref{positive_hom} we know that $p$ is isolated by some quantifier-free positive formula $\psi$. Let $\varphi$ be the conjunction of all atomic formulas contained in $p$. Then clearly $\varphi$ implies $\psi$, and thus $\varphi$ also isolates $p$. Let $\varphi^{\neq}$ be the formula obtained from $\varphi$ by replacing each occurrence of the relation symbol $<_{ij}^*$ by $\neq_{ij}^*$. Let $P$ be the set of those pairs $(s,i)$ such that $s\in [n]$ and $i\in [d_{\pi_1(u_s)}]$. Note that by definition for each pair $(s,i),(t,j)\in P$ the formula  $\varphi^{\neq}$ contains one of the following conjuncts: $x_s=_{ij}^*x_t$,  $x_s\neq_{ij}^*x_t$, $x_t\neq_{ij}^*x_s$. This implies that $\varphi^{\neq}$ also isolates some type $q$ in $\johnneq{\bar{d}}$. We claim that $\bar{u}\in S^{\neq}$ iff $\varphi^{\neq}(\bar{u})$ holds. This clearly implies the statement of the lemma.
	
	Let us first assume that $\bar{u}\in S^{\neq}$. Then we know that there exists some $\bar{v}\in S$ such that $\typ_{\neq}(\bar{u})=\typ_{\neq}(\bar{v})$. In particular, $\varphi^{\neq}(\bar{v})$ and thus also $\varphi^{\neq}(\bar{u})$ holds. For the other direction let us assume that $\varphi^{\neq}(\bar{u})$ holds, and let $\bar{v}$ be a tuple realizing $p$ in $\johnord{\bar{d}}$. Then we have $\typ_{\neq}(\bar{v})=q=\typ_{\neq}(\bar{u})$, and thus $\bar{u}\in S^{\neq}$.
\end{proof}

\begin{lemma}\label{remove_sub}
	Let $\fa$ be a reduct of
    $\johnord{\bar{d}}$ which is stable. Then $\fa$ is a substructure of $\fa^{\neq}$.
\end{lemma}

\begin{proof}
	Let $R$ be a relation of $\fa$ of arity $n$. We have to show that $R^{\neq}\cap \dom(\johnord{\bar{d}})^n=R$.

	The inclusion ``$\supseteq$'' is obvious. For the other direction let us assume for contradiction that $R^{\neq}\cap \dom(\johnord{\bar{d}})^n\supset R$, and let us pick some element $\bar{u}\in (R^{\neq}\cap \dom(\johnord{\bar{d}})^n)\setminus R$. By definition we can find some $\bar{w}\in R$ such that $\typ_{\neq}(\bar{u})=\typ_{\neq}(\bar{w})$. Let $m$ be the number of different entries in the tuples in ${\pi_2}(u_1),\dots,{\pi_2}(u_n)$. Since $\typ_{\neq}(\bar{u})=\typ_{\neq}(\bar{w})$ we know that $m$ is also equal to the same number defined analogously for the tuple $\bar{w}$. We can assume without loss of generality that these coordinates for both tuples are from $[m]$. In this case there exists some $\sigma\in \sym([m])$ such that $\sigma(\bar{w})=\bar{u}$. We can assume without loss of generality that $\bar{u}$ is chosen in a way that the number of inversions in $\sigma$ is minimal. Since $\sigma$ cannot be the identity this means that there exists some $i\in [m]$ such that $\sigma^{-1}(i)>\sigma^{-1}(i+1)$. Let $\tau\coloneqq (i,i+1)$ and $\bar{v}\coloneqq \tau(\bar{u})=(\tau\circ \sigma)(\bar{w})$. We claim that $\bar{v}\in \dom(\johnord{\bar{d}})^n$. Suppose that this is not the case. Then there is some $s\in [n]$ and $j_1<j_2$ such that $(\pi_2(v_s))_{j_1}=i+1$ and $(\pi_2(v_s))_{j_2}=i$ since otherwise $\bar{u}$ would already not be in $\dom(\johnord{\bar{d}})$. Note however that by definition $(\tau\circ \sigma)^{-1}(i)<(\tau\circ \sigma)^{-1}(i+1)$ meaning that $(\pi_2(w_s))_{j_1}{>}(\pi_2(w_s))_{j_2}$ which is impossible since $w_s\in \dom(\johnord{\bar{d}})$. 
	
	We have obtained that $\bar{v}\in \dom(\johnord{\bar{d}})^n$. By definition it is clear the $\typ_{\neq}(\bar{v})=\typ_{\neq}(\bar{w})$. Since the number of inversions in $\tau\circ \sigma$ is lower than in $\sigma$ it follows by the minimality of $\bar{u}$ that $\bar{v}$ must be in $R$. Let $\bar{u}'$ be the subtuple of $\bar{u}$ containing those coordinates $u_k$ for which one of the coordinates of $\pi_2(u_k)$ is $i$. 
    Note that $i+1$ cannot show up in the second coordinate in any of the entries of $\bar{u}'$. Indeed, if both $i$ and $i+1$ showed up as a second coordinate in some entry $u_k$ of $\bar{u}$ then $v_k=\tau(u_k)$ would not be in $\dom(\johnord{\bar{d}})$. By reordering the coordinates of $\bar{u}$ we can assume without loss of generality that $\bar{u}=\bar{u}'\bar{u}''$. Now let us pick some rational numbers $p_1<q_1<p_2<q_2<\dots$ in the open interval $(i-1,i+2)$, and let $\bar{a}_s\coloneqq (i,p_s)(\bar{u}'), \bar{b}_t\coloneqq (i+1,q_t)(\bar{u}'')$ {where $(j,k)( \cdot )$ denotes the natural action of the transposition $(j,k)$ on the corresponding element in $\dom(\johnord{\bar{d}})$}. Note that in this case all tuples $\bar{a}_s$ and $\bar{b}_t$ are from $\johnord{\bar{d}}$. Moreover, if $s\leq t$ then $\bar{a}_s\bar{b}_t=g_{st}(\bar{u})$ where $g_{st}$ is any automorphism of $(\mathbb{Q};<)$ which fixes $\{1,\dots,i-1,i+2,\dots,m\}$ elementwise, maps $i$ to $p_s$ and $i+1$ to $q_t$. Similarly, if $s>t$ then $\bar{a}_s\bar{b}_t=h_{st}(\bar{v})$ where $h_{st}$ is any automorphism of $(\mathbb{Q};<)$ which fixes $\{1,\dots,i-1,i+2,\dots,m\}$ elementwise, maps $i+1$ to $p_s$ and $i$ to $q_t$. Since $\bar{u}\not\in R$ and $\bar{v}\in R$, this implies that $\bar{a}_s\bar{b}_t\in R$ if and only if $s>t$. We have defined a half-graph on $\fa$ which contradicts our assumption that $\fa$ is stable.
\end{proof}
	
	Now, we are ready to prove the main result of this subsection.

\begin{theorem}\label{trace_def}
	Every structure in Lachlan's class is trace minimal.
\end{theorem}

\begin{proof}
	Let $\fa$ be a structure in Lachlan's class. By Lemma~\ref{nice_reducts} we may assume without loss of generality that $\fa$ is a reduct of $\johnord{\bar{d}}$ for some $\ell\in \mathbb{N}$ and $\bar{d}\in \mathbb{N}^{\ell}$. Let $\fa^*$ be the expansion of $\fa$ by all first-order definable relations. Then clearly $\fa^*$ is stable since $\fa$ is stable. By Lemma~\ref{remove_sub} we obtain that the identity map is an embedding from $\fa^*$ into $(\fa^*)^{\neq}$. Since $\fa^*$ is closed under first-order expansions, this embedding is a trace definition of $\fa^*$ in $(\fa^*)^{\neq}$.
    By Lemma~\ref{remove_order}  we know that $(\fa^*)^{\neq}$ is interpretable, in particular trace definable, in $(\mathbb{Q};=)$. Combining these observations we obtain that $\fa$ is trace definable in $(\mathbb{Q};=)$.
\end{proof}


\section{Conclusion}
	Theorem~\ref{chain} below summarises various results of this article.
	
	Let $\mathcal E$ be the class of all NIP structures that are first-order reducts of binary finitely bounded homogeneous Ramsey structures.

\begin{theorem}\label{chain}
We have the following inclusions and strict inclusions. 
    $$\I(\atoms)\subset \MI(\atoms)\subseteq \mathcal{D}\subset \I(\order)\subset \MI(\order) \subseteq {\mathcal E}$$
\end{theorem}
\begin{proof}
    The inclusion $\I(\atoms) \subseteq \MI(\atoms)$ is trivial, since the expansion of every structure in $\I(\atoms)$ 
    by all first-order definable relations is
    a model-complete core. 
    The inclusion is strict by Corollary~\ref{not_closed_mc}.
    
   It is clear that interpretations preserve stability, in particular every structure in $\I(\atoms)$ 
    is stable (Theorem~\ref{thm:mc-pres}) and hence in ${\mathcal D}$ by Theorem~\ref{lachlan_extra}. This implies the inclusion $\MI(\atoms) \subseteq {\mathcal D}$, because ${\mathcal D}$ is closed under taking model-complete cores by Corollary~\ref{d_mc}.

    The inclusion ${\mathcal D} \subseteq \I(\order)$ holds by Theorem~\ref{lachlan_extra}, and the inclusion is strict, because $\order$ itself is not stable and therefore not in ${\mathcal D}$. 
    The inclusion $\I(\order) \subseteq \MI(\order)$ is again trivial, and strict by Corollary~\ref{not_closed_mc_q}.
    
    Finally, the inclusion $\I(\order) \subseteq {\mathcal E}$ follows from Remark~\ref{rem:ramsey-exp}, and Theorems~\ref{mc_transfer} and~\ref{thm:mc-pres} imply that ${\mathcal E}$ is preserved by taking model-complete cores. Therefore, we have $\MI(\order) \subseteq {\mathcal E}$.
\end{proof}

	Conjecture~\ref{conj:D} asks whether the containment of $\MI(\atoms)\subseteq \mathcal{D}$ could be an equality. Similarly, Question~\ref{quest:E} below asks whether this could be the case also for the final containment in Theorem~\ref{chain}.
   
\begin{question}\label{quest:E}
    Is it true that the class ${\mathcal E}$ equals the class of structures that are interdefinable with a model-complete core of a structure  interpretable in $\order$? 
\end{question}

This question is open even if we drop the requirements in the definition of ${\mathcal E}$ that the expansion is Ramsey and finitely bounded:

\begin{question}\label{quest:E2} Is every binary finitely homogeneous $\mathrm{NIP}$ structure 
interdefinable with a structure in $\MI(\order)$?
\end{question}
Note that a positive answer to Question~\ref{quest:E2} would confirm Walsberg's Conjecture~\ref{conj: unanswered}.
This, together with Theorem~\ref{trace_def} would imply that there are only three binary homogeneous structures up to trace definability, and so, in a sense, as mentioned in Section~\ref{sec:tracedeflach}, that stability and $\mathrm{NIP}$ are the only two ``binary'' model theoretic properties.




\subsection*{Acknowledgements}
We would like to thank Erik Walsberg for helpful discussions on trace definability, and Gabriel Day and Scott Mutchnik for noticing an issue in an earlier version of the arguments in Section 3. 
This work was supported by UKRI EP/X024431/1. Funded by the European Union (ERC, POCOCOP, 101071674). Views and opinions expressed are however those of the authors only and do not necessarily reflect those of the European Union or the European Research Council Executive Agency. Neither the European Union nor the granting authority can be held responsible for them.

\bibliographystyle{alpha}
\bibliography{local}

@STRING{LICS = {Proceedings of the Annual Symposium on Logic in Computer Science (LICS)} }

@preamble{"\def\cprime{$'$} "}

@inproceedings{BojanczykTorunczyk18,
  author    = {Miko{\l}aj Boja{\'{n}}czyk and
               Szymon Toru{\'{n}}czyk},
  title     = {On computability and tractability for infinite sets},
  booktitle = {Proceedings of the 33rd Annual {ACM/IEEE} Symposium on Logic in Computer
               Science ({LICS}), Oxford, UK, July 09-12, 2018},
  pages     = {145--154},
  year      = {2018}
}

@misc{rosy,
      title={Dependent finitely homogeneous rosy structures}, 
      author={Alf Onshuus and Pierre Simon},
      year={2021},
      eprint={2107.02727},
      archivePrefix={arXiv},
      primaryClass={math.LO},
      url={https://arxiv.org/abs/2107.02727}, 
}

@article{SimonRankOne,
	author = {Simon, Pierre},
	journal = {Proceedings of the London Mathematical Society},
	number = {6},
	pages = {1253-1331},
	title = {{NIP} $\omega$-categorical structures: The rank 1 case},
	volume = {125},
	year = {2022}}

@inbook{Anscombe2024,
	address = {Cham},
	author = {Anscombe, Sylvy and Karemaker, Valentijn and Kisak{\"u}rek, Zeynep and Mehmeti, Vler{\"e} and Pagano, Margherita and Paladino, Laura},
	editor = {Abdellatif, Ramla and Karemaker, Valentijn and Smajlovic, Lejla},
	pages = {29--61},
	publisher = {Springer International Publishing},
	title = {A Survey of Local--Global Methods for Hilbert's Tenth Problem},
	year = {2024}}

@article{OrbitFinDim,
    title      = {Orbit-Finite-Dimensional Vector Spaces and Weighted Register Automata},
    author     = {Mikołaj Bojańczyk and Joanna Fijalkow and Bartek Klin and Joshua Moerman},
    url        = {https://theoretics.episciences.org/11208},
    doi        = {10.46298/theoretics.24.13},
    journal    = {TheoretiCS},
    issn       = {2751-4838},
    volume     = {Volume 3},
    eid        = 13,
    year       = {2024},
    month      = {May},
}

@article{Orbit-finite-LP, author = {Ghosh, Arka and Hofman, Piotr and Lasota, S\l{}awomir}, title = {Orbit-finite Linear Programming}, year = {2025}, issue_date = {February 2025}, publisher = {Association for Computing Machinery}, address = {New York, NY, USA}, volume = {72}, number = {1}, issn = {0004-5411}, url = {https://doi.org/10.1145/3703909}, doi = {10.1145/3703909}, journal = {J. ACM}, month = jan, articleno = {1}, numpages = {39} }

@preamble{
   "\def\cprime{$'$} "
}

@article{DBLP:journals/corr/BojanczykKL14,
  author    = {Miko{\l}aj Boja{\'{n}}czyk and
               Bartek Klin and
               S\l awomir Lasota},
  title     = {Automata theory in nominal sets},
  journal   = {Logical Methods in Computer Science},
  volume    = {10},
  number    = {3},
  year      = {2014}
}

@inproceedings{DBLP:conf/lics/BojanczykKLT13,
  author    = {Miko{\l}aj Boja{\'{n}}czyk and
               Bartek Klin and
               S\l awomir Lasota and
               Szymon Toru{\'{n}}czyk},
  title     = {Turing Machines with Atoms},
  booktitle = {28th Annual {ACM/IEEE} Symposium on Logic in Computer Science, {LICS}
               2013, New Orleans, LA, USA},
  pages     = {183--192},
  year      = {2013}
}

@inproceedings{KlinLOT14-short,
  author    = {Bartek Klin and
               S\l awomir Lasota and
               Joanna Ochremiak and
               Szymon Toru{\'{n}}czyk},
  title     = {Turing machines with atoms, constraint satisfaction problems, and
               descriptive complexity},
  booktitle = {Proceedings of Computer
               Science Logic {(CSL)} and Symposium
               on Logic in Computer Science (LICS), 
               {CSL-LICS}'14, Vienna, Austria.},
  pages     = {58:1--58:10},
  year      = {2014}
}

@Article{Macpherson-RapidGrowth,
author = {Dugald Macpherson}, 
title = {Infinite permutation groups of rapid growth}, 
journal = {J. London Math. Soc.},
volume = {35},
number = {2},
year = {1987}, 
pages = {276-286}}

@Article{wonderland,
author={Libor Barto and Jakub Opr\v{s}al and Michael Pinsker},
title={The wonderland of reflections},
journal = {Israel Journal of Mathematics},
volume=223,
number=1,
year=2018,
pages={363-398}
}

@phdthesis{Bennett-thesis,
title={The reducts of some infinite homogeneous graphs and tournaments},
author={James H. Bennett},
school={Rutgers University},
year=1997
}

@Article{SchreierUlam,
author = {Joseph Schreier and Stanis{\l}aw Marcin Ulam},
title = {{\"U}ber die {P}ermutationsgruppe der nat\"urlichen {Z}ahlenfolge},
journal = {Studia Mathematica},
volume = {4},
pages = {134-141},
year = {1933}}

@Article{Saracino,
author = {Dan Saracino},
title = {Model companions for $\Aleph_0$-categorical theories},
journal = {Proceedings of the AMS},
volume = {39},
year = {1973},
pages = {591-598}}

@Article{CherlinHarringtonLachlan,
author={Gregory Cherlin and Leo Harrington and Alistair H. Lachlan},
title={$\Aleph_0$-categorical, $\Aleph_0$-stable Structures},
journal = {Annals of Pure and Applied Logic},
pages = {103-135},
volume = {28},
year = {1985}}

@Article{HrushovskiTotallyCategorical,
title = {Totally categorical structures},
author = {Ehud Hrushovski},
journal = {Transactions of the American Mathematical Society},
volume = {313},
year = {1989},
number = {1},
pages = {131-159}}

@Article{MatiyasevichDoklady,
author = {Yuri Matiyasevich},
title = {Enumerable sets are {D}iophantine},
journal = {Doklady Akademii Nauk SSSR},
volume = {191},
pages = {279-282}, 
year = {1970}}

@Article{MacphersonSurvey,
title = {A survey of homogeneous structures},
author = {Dugald Macpherson},
Volume = {311}, 
number = {15},
year = {2011},
pages = {1599-1634},
journal = {Discrete Mathematics}}

@Article{Topo-Dynamics,
author = {Alexander Kechris and Vladimir Pestov and Stevo Todor\v{c}evi\'c},
title = {Fraiss\'e Limits, {R}amsey Theory, and topological dynamics of automorphism groups},
journal = {Geometric and Functional Analysis},
volume = {15},
number = {1},
pages = {106-189},
year = {2005}}

@Article{EvansIvanovMacpherson,
author = {David Evans 
and Alexandre A. Ivanov 
and Dugald Macpherson},
title = {Finite Covers}, 
journal = {In `Model Theory of Groups and Automorphism Groups', LMS Lecture Note Series, 244},
publisher = {Oxford University Press},
pages = {1-72}, 
year = {1997}}

@article{cores,
author = {Pavol Hell and Jaroslav Ne\v{s}et\v{r}il},
title = {The core of a graph}, 
journal = {Discrete Mathematics},
volume = {109},
year = {1992},
pages = {117-126}}

@BOOK{Tent-Ziegler,
author = {Katrin Tent and Martin Ziegler}, 
title = {A course in model theory},
series = {Lecture Notes in Logic}, 
publisher = {Cambridge University Press}, 
year = {2012}}

@BOOK{HodgesLong,
author = {Wilfrid Hodges},
title = {Model theory},
publisher = {Cambridge University Press},
adress = "Cambridge",
year = {1993}}

@Article{CameronOrbits,
author = {Peter J. Cameron},
title = {Orbits of Permutation Groups on Unordered Sets, {II}},
journal = {Journal of the London Mathematical Society},
volume = {2},
year = {1981},
pages = {249--264}}

@Article{LachlanIndiscernible,
author = {Alistair H. Lachlan},
title = {Structures Coordinatized by Indiscernible Sets},
journal = {Annals of Pure and Applied Logic},
volume = {34},
pages = {245--273},
year = {1987}}

@InProceedings{LachlanSurvey,
author = {A. H. Lachlan},
title = {Stable finitely homogeneous structures: A Survey},
booktitle = {Algebraic Model Theory, NATO ASI Series},
editors = {T.T.Hart and Alistair H. Lachlan and Matthew A. Valeriote},
volume = {496},
pages = {145-159},
year = {1996}}

@Article{CherlinLachlan,
author = {G. Cherlin and A. H. Lachlan},
year = {1986},
title = {Stable finitely homogeneous structures},
journal = {TAMS},
volume = {296}, 
pages = {815-850}}

@Book{HNBook, 
author = {Pavol Hell and Jaroslav Ne\v{s}et\v{r}il},
title = {Graphs and Homomorphisms},
publisher = {Oxford University Press},
address = {Oxford},
year = 2004}

@article{BodHilsMartin-Journal,
  author = {Manuel Bodirsky and Martin Hils and Barnaby Martin},
  title = {On the scope of the universal-algebraic approach to constraint satisfaction},
  journal = {Logical Methods in Computer Science (LMCS)},
  note = {An extended abstract that announced some of the results appeared in the proceedings of Logic in Computer Science (LICS'10)},
  year = {2012},
  volume = {8},
  number = {3}
  }

@Article{Cores-journal,
   author =       "Manuel Bodirsky",
   title = "Cores of countably categorical structures",
   journal = "Logical Methods in Computer Science ({LMCS})",
   volume    = {3},
   pages = {1-16},
  number    = {1},
   year = {2007}}

@article{BodJunker,
  author    = {Manuel Bodirsky and
               Markus Junker},
  title     = {$\Aleph_0$-categorical structures: interpretations and endomorphisms},
  journal   = {Algebra Universalis},
  volume = {64}, 
  number = {3-4}, 
  pages = {403-417},
  year      = {2011}
}

@article{Bod-New-Ramsey-classes,
    AUTHOR = {Manuel Bodirsky},
     TITLE = {{New Ramsey} classes from old},
     journal = {Electronic Journal of Combinatorics}, 
     volume = {21},
     number = {2},
     year = {2014}
}

@article{braunfeld2022monadic,
  title={Monadic stability and growth rates of $\omega$-categorical structures},
  author={Braunfeld, Samuel},
  journal={Proceedings of the London Mathematical Society},
  volume={124},
  number={3},
  pages={373--386},
  year={2022},
  publisher={Wiley Online Library}
}

@article{baldwin1985second,
  title={Second-order quantifiers and the complexity of theories.},
  author={Baldwin, John T. and Shelah, Saharon},
  journal={Notre Dame Journal of Formal Logic},
  volume={26},
  number={3},
  pages={229--303},
  year={1985},
  publisher={University of Notre Dame}
}

@Unpublished{Book,
author = {Manuel Bodirsky},
title  = {Complexity of Infinite-Domain Constraint Satisfaction},
year   = {2020},
note = {Submitted for publication in the LNL Series, Cambridge
University Press}}

@article{bodirsky2021permutation,
  title={Permutation groups with small orbit growth},
  author={Bodirsky, Manuel and Bodor, Bertalan},
  journal={Journal of Group Theory},
  year={2021},
  publisher={De Gruyter}
}

@article{MottetPinskerCores,
  author    = {Antoine Mottet and
               Michael Pinsker},
  title     = {Cores over {Ramsey} structures},
  journal   = {Journal of Symbolic Logic},
  volume    =   86,
  number=1,
  pages={352-361},
  year=2021
}

@article{bodor2024classification,
  title={Classification of $\omega$-categorical monadically stable structures},
  author={Bodor, Bertalan},
  journal={The Journal of Symbolic Logic},
  volume={89},
  number={2},
  pages={460--495},
  year={2024},
  publisher={Cambridge University Press}
}

@Misc{Blog,
title = {$\omega$-categorical, $\omega$-stable structure with trivial geometry not definable in the pure set}, 
author = {Szymon Toru\'nczyk and James E. Hanson}, 
note ={Mathoverflow discussion. The example is apparently given by Ehud Hrushovski in a personal communication.
\url{https://mathoverflow.net/questions/231791}}
}

@article{macpherson1991interpreting,
  title={Interpreting groups in $\omega$-categorical structures},
  author={Macpherson, Dugald},
  journal={The Journal of Symbolic Logic},
  volume={56},
  number={4},
  pages={1317--1324},
  year={1991},
  publisher={Cambridge University Press}
}

@article{bodirsky2025structures,
  title={Structures preserved by primitive actions of {$S_{\omega}$}},
  author={Bodirsky, Manuel and Bodor, Bertalan},
  journal={arXiv preprint arXiv:2501.03789},
  year={2025}
}

@book{bodirsky2021complexity,
  title={Complexity of infinite-domain constraint satisfaction},
  author={Bodirsky, Manuel},
  volume={52},
  year={2021},
  publisher={Cambridge University Press}
}

@phdthesis{bodor2022csp,
  title={{CSP} dichotomy for $\omega$-categorical monadically stable structures},
  author={Bodor, Bertalan},
  year={2022},
  note={\url{https://nbn-resolving.org/urn:nbn:de:bsz:14-qucosa2-774379}},
school={Technische Universit\"{a}t Dresden, Dresden},
  publisher={Qucosa, TU Dresden}
}

@article{PosModT,
  title={Dividing lines between positive theories},
  author={Dmitrieva, Anna and Gallinaro, Francesco and Kamsma, Mark},
  journal={The Journal of Symbolic Logic},
  pages={1--25},
  year={2023},
  publisher={Cambridge University Press}
}

@book{shelah1990classification,
  title={Classification theory: and the number of non-isomorphic models},
  author={Shelah, Saharon},
  volume={92},
  year={1990},
  publisher={Elsevier}
}

@article{shelah1980simple,
  title={Simple unstable theories},
  author={Shelah, Saharon},
  journal={Annals of mathematical logic},
  volume={19},
  number={3},
  year={1980}
}

@article{dvzamonja2004maximality,
  title={On $\lhd_*$-maximality},
  author={D{\v{z}}amonja, Mirna and Shelah, Saharon},
  journal={Annals of Pure and Applied Logic},
  volume={125},
  number={1-3},
  pages={119--158},
  year={2004},
  publisher={Elsevier}
}

@article{shelah1995toward,
  title={Toward classifying unstable theories},
  author={Shelah, Saharon},
  journal={Annals of Pure and Applied Logic},
  volume={80},
  pages={229--255},
  year={1996}
}

@article{bailetti2024walk,
  title={A Walk on the Wild Side: Notions of maximality in first-order theories},
  author={Bailetti, Michele},
  journal={The Journal of Symbolic Logic},
  pages={1--30},
  year={2024},
  publisher={Cambridge University Press}
}

@article{keisler1976six,
  title={Six classes of theories},
  author={Keisler, Howard Jerome},
  journal={Journal of the Australian Mathematical Society},
  volume={21},
  number={3},
  pages={257--266},
  year={1976},
  publisher={Cambridge University Press}
}

@article{mutchnik2025textup,
  title={On $\mathrm{NSOP}_2$ theories},
  author={Mutchnik, Scott},
  journal={Journal of the European Mathematical Society},
  year={2025}
}

@article{ben2003positive,
  title={Positive model theory and compact abstract theories},
  author={Ben-Yaacov, Itay},
  journal={Journal of Mathematical Logic},
  volume={3},
  number={01},
  pages={85--118},
  year={2003},
  publisher={World Scientific}
}

@article{dobrowolski2022kim,
  title={Kim-independence in positive logic},
  author={Dobrowolski, Jan and Kamsma, Mark},
  journal={Model Theory},
  volume={1},
  number={1},
  pages={55--113},
  year={2022},
  publisher={Mathematical Sciences Publishers}
}

@article{haykazyan2019spaces,
  title={Spaces of types in positive model theory},
  author={Haykazyan, Levon},
  journal={The Journal of Symbolic Logic},
  volume={84},
  number={2},
  pages={833--848},
  year={2019},
  publisher={Cambridge University Press}
}

@article{braunfeld2021characterizations,
  title={Characterizations of monadic {NIP}},
  author={Braunfeld, Samuel and Laskowski, Michael},
  journal={Transactions of the American Mathematical Society, Series B},
  volume={8},
  number={30},
  pages={948--970},
  year={2021}
}

@article{braunfeld2024corrigenda,
  title={Corrigenda to “Characterizations of monadic {NIP}”},
  author={Braunfeld, Samuel and Laskowski, Michael},
  journal={Transactions of the American Mathematical Society, Series B},
  volume={11},
  number={34},
  pages={1226--1232},
  year={2024}
}

@article{bodor2025structures,
  title={Structures with not too fast unlabelled growth},
  author={Bodor, Bertalan},
  journal={arXiv preprint arXiv:2507.16985},
  year={2025}
}

@article{brouwer1980enumeration,
  title={The enumeration of locally transitive tournaments},
  author={Brouwer, Andries E.},
  journal={Math. Centr. Report ZW138, Amsterdam}, 
  year={1980}
}

@book{pillay1996geometric,
  title={Geometric stability theory},
  author={Pillay, Anand},
  year={1996},
  publisher={Oxford University Press}
}

@article{shelah1971stability,
  title={Stability, the f.c.p., and superstability; model theoretic properties of formulas in first order theory},
  author={Shelah, Saharon},
  journal={Annals of Mathematical Logic},
  volume={3},
  number={3},
  pages={271--362},
  year={1971},
  publisher={Elsevier}
}

@inproceedings{monadicallyNSOP,
  author    = {Eleftheriadis, Ioannis and Papadopoulos, Aristomenis-Dionysios},
  title     = {{NSOP} in Classes of Graphs},
  booktitle = {Proceedings of the 13th Panhellenic Logic Symposium},  
  pages     = {86--91}, 
  year      = {2022}
}

@article{bodor2025labelled,
  title={Labelled growth rates of $\omega$-categorical structures and applications in choiceless set theory},
  author={Bodor, Bertalan and Braunfeld, Samuel and Hanson, James E.},
  journal={arXiv preprint arXiv:2509.12656},
  year={2025}
}

@inproceedings{dreier2024first,
  title={First-order model checking on monadically stable graph classes},
  author={Dreier, Jan and Eleftheriadis, Ioannis and M{\"a}hlmann, Nikolas and McCarty, Rose and Pilipczuk, Micha{\l} and Toru{\'n}czyk, Szymon},
  booktitle={2024 IEEE 65th Annual Symposium on Foundations of Computer Science (FOCS)},
  pages={21--30},
  year={2024},
  organization={IEEE}
}

@book{simon2015guide,
  title={A guide to NIP theories},
  author={Simon, Pierre},
  year={2015},
  publisher={Cambridge University Press}
}

@book{wagner2000simple,
  title={Simple theories},
  author={Wagner, Frank Olaf},
  volume={260},
  year={2000},
  publisher={Springer}
}

@book{kim2013simplicity,
  title={Simplicity theory},
  author={Kim, Byunghan},
  volume={53},
  year={2013},
  publisher={OUP Oxford}
}

@misc{kamsma2025positivelogicintroductionmodel,
      title={Positive Logic: An Introduction for Model Theorists}, 
      author={Mark Kamsma},
      year={2025},
      eprint={2511.10167},
      archivePrefix={arXiv},
      primaryClass={math.LO},
      url={https://arxiv.org/abs/2511.10167}, 
}

@article{kamsma2023bilinear,
  title={Bilinear spaces over a fixed field are simple unstable},
  author={Kamsma, Mark},
  journal={Annals of Pure and Applied Logic},
  volume={174},
  number={6},
  pages={103268},
  year={2023},
  publisher={Elsevier}
}

@article{Walsbergshort,
  title={Notes on trace equivalence},
  author={Walsberg, Erik},
  journal={arXiv preprint arXiv:2101.12194},
  year={2021}
}

@article{Walsberglong,
  title={Trace Definability},
  author={Walsberg, Erik},
  journal={arXiv preprint arXiv:2504.05566v1},
  year={2025}
}

@article{shelah2014strongly,
  title={Strongly dependent theories},
  author={Shelah, Saharon},
  journal={Israel Journal of Mathematics},
  volume={204},
  number={1},
  pages={1--83},
  year={2014},
  publisher={Springer}
}

@article{aschenbrenner2016vapnik,
  title={{V}apnik-{C}hervonenkis density in some theories without the independence property, {I}},
  author={Aschenbrenner, Matthias and Dolich, Alf and Haskell, Deirdre and Macpherson, Dugald and Starchenko, Sergei},
  journal={Transactions of the American Mathematical Society},
  volume={368},
  number={8},
  pages={5889--5949},
  year={2016}
}

@inproceedings{basit2021zarankiewicz,
  title={Zarankiewicz’s problem for semilinear hypergraphs},
  author={Basit, Abdul and Chernikov, Artem and Starchenko, Sergei and Tao, Terence and Tran, Chieu-Minh},
  booktitle={Forum of Mathematics, Sigma},
  volume={9},
  pages={e59},
  year={2021},
  organization={Cambridge University Press}
}

@article{shelah2007definable,
  title={Definable groups for dependent and 2-dependent theories},
  author={Shelah, Saharon},
  journal={arXiv preprint math/0703045},
  year={2007}
}

@article{chernikov2019n,
  title={On $n$-Dependence},
  author={Chernikov, Artem and Palacin, Daniel and Takeuchi, Kota},
  journal={Notre Dame Journal of Formal Logic},
  volume={60},
  number={2},
  pages={195--214},
  year={2019}
}

@misc{Radowithconstant,
  title = {The {R}ado graph admits only one oligomorphic action of its automorphism group},
howpublished = {\url{https://www.mimuw.edu.pl/~atoms/?p=1240}},
url={https://www.mimuw.edu.pl/~atoms/?p=1240},
year={2016},
  note = {Blog post. Atompress.},
author={Toru\'{n}czyk, Szymon}
}

@article{saharon2000not,
  title={On what {I} do not understand (and have something to say), model theory},
  author={Shelah, Saharon},
  journal={Mathematica Japonica},
  volume={51},
  pages={329--377},
  year={2000}
}

@article{kruckman2024new,
  title={A new Kim’s lemma},
  author={Kruckman, Alex and Ramsey, Nicholas},
  journal={Model Theory},
  volume={3},
  number={3},
  pages={825--860},
  year={2024},
  publisher={Mathematical Sciences Publishers}
}

@article{walsberg2026tracedefinabilityipreservation,
      title={Trace definability I: preservation and characterizations}, 
      author={Erik Walsberg},
      year={2026},
      journal={arXiv preprint arXiv:2504.05566},
      eprint={2504.05566},
      archivePrefix={arXiv},
      primaryClass={math.LO},
      url={https://arxiv.org/abs/2504.05566}, 
}

@article{guingona2015common,
  title={On a common generalization of {Shelah's} 2-rank, dp-rank, and o-minimal dimension},
  author={Guingona, Vincent and Hill, Cameron Donnay},
  journal={Annals of Pure and Applied Logic},
  volume={166},
  number={4},
  pages={502--525},
  year={2015},
  publisher={Elsevier}
}

@article{walsberg2026tracedefinabilityiimodeltheoretic,
      title={Trace definability II: model-theoretic linearity}, 
      author={Erik Walsberg},
      year={2026},
      journal={arXiv preprint arXiv:2605.12323},
      eprint={2605.12323},
      archivePrefix={arXiv},
      primaryClass={math.LO},
      url={https://arxiv.org/abs/2605.12323}, 
}

@article{day2026notion,
  title={On the notion of a patterning property in model theory},
  author={Day, Gabriel and Mutchnik, Scott},
  journal={arXiv preprint arXiv:2606.18533},
  year={2026}
}

@article{day2025results,
  title={Results on Colored Tree Properties},
  author={Day, Gabriel},
  journal={arXiv preprint arXiv:2507.06977},
  year={2025}
}

@article{shelah2008more,
  title={More on $\mathrm{SOP}_1$ and $\mathrm{SOP}_2$},
  author={Shelah, Saharon and Usvyatsov, Alexander},
  journal={Annals of Pure and Applied Logic},
  volume={155},
  number={1},
  pages={16--31},
  year={2008},
  publisher={Elsevier}
}

@article{abd2025higher,
  title={Higher arity stability and the functional order property},
  author={Abd Aldaim, A and Conant, Gabriel and Terry, Caroline},
  journal={Selecta Mathematica},
  volume={31},
  number={3},
  pages={59},
  year={2025},
  publisher={Springer}
}

@article{terry2021higher,
  title={Higher-order generalizations of stability and arithmetic regularity},
  author={Terry, Caroline and Wolf, Julia},
  journal={arXiv preprint arXiv:2111.01739},
  year={2021}
}

\end{document}